\pdfoutput=1
\documentclass[a4paper, 11pt, reqno]{amsart}

\usepackage{amssymb}
\usepackage{amsmath}
\usepackage{amsthm}
\usepackage[numbers]{natbib}
\usepackage{mathtools}
\usepackage{enumitem}
\usepackage{calc}
\usepackage{setspace}
\usepackage{graphicx}
\usepackage{tikz-cd}
\usepackage{mathscinet}
\usepackage{hyperref}
\newcommand{\noop}[1]{}
\usepackage[top=3 cm, bottom=3 cm, left=3 cm, right=3cm]{geometry}

\tikzset{
  symbol/.style={
    draw=none,
    every to/.append style={
      edge node={node [sloped, allow upside down, auto=false]{$#1$}}}
  }
}

\newcounter{dummy}
\numberwithin{dummy}{section}
\newtheorem{thm}[dummy]{Theorem}
\newtheorem{defn}[dummy]{Definition}
\newtheorem{conj}[dummy]{Conjecture}
\newtheorem{lem}[dummy]{Lemma}
\newtheorem{prop}[dummy]{Proposition}
\newtheorem{cor}[dummy]{Corollary}
\theoremstyle{definition}
\newtheorem{rmk}[dummy]{Remark}
\numberwithin{equation}{section}
\newtheorem*{claim}{Claim}

\renewcommand\Re{\operatorname{Re}}
\renewcommand\Im{\operatorname{Im}}

\DeclareMathOperator{\id}{id}
\DeclareMathOperator{\der}{der}

\DeclareMathOperator{\ad}{ad}

\DeclareMathOperator{\ws}{ws}
\DeclareMathOperator{\red}{red}

\DeclareMathOperator{\Gal}{\mathrm{Gal}}
\DeclareMathOperator{\GL}{\mathrm{GL}}
\DeclareMathOperator{\Lie}{\mathrm{Lie}}

\DeclareMathOperator{\SL}{\mathrm{SL}}

\DeclareMathOperator{\GSp}{\mathrm{GSp}}
\DeclareMathOperator{\Spec}{\mathrm{Spec}}

\DeclareMathOperator{\Hom}{\mathrm{Hom}}
\DeclareMathOperator{\Res}{\mathrm{Res}}

\DeclareMathOperator{\ZP}{\mathrm{ZP}}

\DeclareMathOperator{\unif}{\mathrm{unif}}
\DeclareMathOperator{\cl}{cl}
\DeclareMathOperator{\ab}{ab}

\makeatletter
\renewcommand{\@biblabel}[1]{[#1]\hfill}
\makeatother

\begin{document}
\title{Distinguished Categories and the Zilber-Pink Conjecture}

\author[F. Barroero]{Fabrizio Barroero}
\address[F. Barroero]{Universit\`a degli studi Roma Tre, Dipartimento di Matematica e Fisica, Largo San Murialdo 1, 00146 Roma, Italy}
\email{fbarroero@gmail.com}

\author[G. A. Dill]{Gabriel A. Dill}
\address[G. A. Dill]{Institut de Math\'ematiques, Universit\'e de Neuch\^atel, Rue Emile-Argand 11, 2000 Neuch\^atel, Switzerland}
\email{gabriel.dill@unine.ch}

\date{\today}
\maketitle

\begin{abstract}
We propose an axiomatic approach towards studying unlikely intersections by introducing the framework of distinguished categories. This includes commutative algebraic groups and mixed Shimura varieties. It allows us to define all basic concepts of the field and to prove some fundamental facts about them, e.g., the defect condition.

In some categories that we call very distinguished, we are able to show some implications between Zilber-Pink statements with respect to base change. This yields unconditional results, i.e., the Zilber-Pink conjecture for a complex curve in $\mathcal{A}_2$ that cannot be defined over $\bar{\mathbb{Q}}$, a complex curve in the $g$-th fibered power of the Legendre family, and a complex curve in the base change of a semiabelian variety over $\bar{\mathbb{Q}}$.
\end{abstract}

\tableofcontents

\section{Introduction}

This article is inspired by recent advances in the field of unlikely intersections, most importantly by \cite{HP}. Its main purpose is to introduce the general framework of distinguished categories. A distinguished category is a category with a functor to the category of algebraic varieties over an algebraically closed field of characteristic $0$ that satisfies four axioms (see Section \ref{sec:dc} for the precise definition and the list of the axioms \ref{ax:1} to \ref{ax:4}). Algebraic varieties are always irreducible in this article; while (mixed) Shimura varieties as usually defined might not be irreducible, we consider here only connected (mixed) Shimura varieties, which are irreducible. We sometimes omit the word ``connected" when speaking informally. Fields are always of characteristic $0$ and subvarieties are always closed and defined over the field of definition of the ambient variety.

Equivalents or close analogues of some results in this paper were obtained independently by Cassani \cite{Cass2, Cass3} in the case of connected mixed Shimura varieties (of Kuga type).

The fundamental insight that underlies this article is that many statements about unlikely intersections are highly formal in nature, hence one can prove them without using any more properties of the studied objects than the ones codified in the definition of a distinguished category.

To give an idea of the power of this approach, we note that it encompasses algebraic tori, (semi-)abelian varieties, pure and mixed Shimura varieties, and even powers of the additive group as well as general connected commutative algebraic groups. Many (but not all) of these examples of distinguished categories also satisfy a fifth axiom \ref{ax:5}, which is needed for some of the results in this article.

Such an axiomatic approach has advantages as well as disadvantages: In the case of Shimura varieties (pure or mixed), the method allows to dispense with their very technical definition once the axioms have been verified. Another obvious advantage is that certain statements can then be proven for all examples at once. As for disadvantages, the approach is far too general to grasp the arithmetic subtleties that have so far been crucial in proving deep results on unlikely intersections (lower bounds for the size of Galois orbits, height upper bounds, etc.). Nevertheless, it provides a streamlined approach to performing certain reduction steps. Other attempts to axiomatize basic facts about unlikely intersections have been made by Zilber \cite{Zilber16}, Ullmo \cite{Ullmo}, Pila \cite{PilaUnpubl}, and Eterovi\'{c} and Scanlon \cite{Eterovic_Scanlon_Preprint}.

We have already mentioned the article \cite{HP} by Habegger and Pila. This did not come out of the blue but, while of great importance, was just one of the last steps after almost two decades of progress on the topic. Indeed, one of the first attempts to go beyond problems of Manin-Mumford-Andr\'e-Oort type arguably dates back to the end of the 90s, when Bombieri, Masser and Zannier \cite{BMZ99} showed that a curve in $\mathbb{G}^n_m$ has finite intersection with the union of algebraic subgroups of codimension at least 2, provided that it is defined over the algebraic numbers and not contained in a translate of a proper algebraic subgroup of $\mathbb{G}^n_m$. Precursors of this result can be found in \cite{Schinzel89} and in the appendix to \cite{SchinzelBook} by Zannier. See also the conjecture on p. 223 of \cite{Zhang_1998}. Maurin \cite{Maurin} later showed that the hypothesis of not being contained in a translate of a proper algebraic subgroup of $\mathbb{G}^n_m$ in \cite{BMZ99} can be replaced by the weaker but necessary hypothesis of not being contained in a proper algebraic subgroup of $\mathbb{G}^n_m$.

In \cite{BMZ99}, the authors mentioned possible analogues of their result in the setting of abelian varieties and families thereof, and so, while Bombieri, Masser and Zannier continued their study of intersections of subvarieties of $\mathbb{G}^n_m$ with algebraic subgroups, other authors started considering the abelian analogues. The earliest work in this direction is from 2003, due to Viada \cite{Viada2003} as well as R\'emond and Viada \cite{RemVia}.

In the meanwhile, Zilber \cite{Zilber} independently considered problems of this kind from a completely different viewpoint and with different motivations coming from model theory. He formulated a general conjecture for subvarieties of arbitrary dimension of semiabelian varieties.

Related conjectures for powers of the multiplicative group were formulated by Bombieri, Masser and Zannier in \cite{BMZ06} and \cite{BMZ07}; see the Appendix to \cite{BMZ08} for a proof of the equivalence of the various conjectures in the case of powers of the multiplicative group.

Around the same time, in \cite{Pink} and \cite{PinkUnpubl}, Pink proposed very general conjectures in the context of mixed Shimura varieties. His most general conjecture contains all of the abovementioned results about $\mathbb{G}^n_m$ and abelian varieties. At the same time, it provides an analogue of Zilber's conjecture in the more general context of mixed Shimura varieties. Pink's formulation of the conjecture seems to be slightly weaker than Zilber's, but we will show in this article that they are often equivalent. Pink's conjecture also implies the ``modular analogue'' of Zilber's conjecture, which generalizes the Andr\'e-Oort conjecture in the same way as Zilber's conjecture generalizes the Manin-Mumford conjecture.

After the groundbreaking article \cite{PilaZannier} of Pila and Zannier, in which they developed and applied a new strategy relying on point counting to give an alternative proof of the Manin-Mumford conjecture, it gradually became clear that this new approach would lead to relevant progress towards the aforementioned conjectures. Indeed, in the last fifteen years the method has been applied quite successfully in basically all different contexts included in Pink's formulation, see the bibliographies of \cite{Zannier} and \cite{PilaUnpubl}.

For the present work, the article \cite{HP} of Habegger and Pila is of particular importance. Although their main aim was to show that o-minimality, together with functional transcendence results and Galois bounds that are still largely open as of today, is sufficient to obtain new cases of the Zilber-Pink conjecture for abelian varieties and products of modular curves, they also introduced new ideas (e.g., the defect condition) that we exploited in \cite{BD} to extend their result and prove the Zilber-Pink conjecture for curves in complex abelian varieties.

Going through the proofs in \cite{HP} and \cite{BD}, one realizes that many steps are purely formal and here we propose a categorical setting in which basic objects (e.g., special and weakly special subvarieties) can be defined and facts about them (e.g., the defect condition) can be proved formally.

We show that basically all the abovementioned examples, prominently (connected) mixed Shimura varieties, satisfy our axioms. On the other hand, unlikely intersections have also been studied in various other settings and for some of these, it is unclear whether they could be integrated into the framework of distinguished categories: We mention here (families of) Drinfeld modules (see \cite{GhiocaHsia}), the results for affine space over function fields in \cite{CGMM}, Klingler's mixed Hodge varieties in \cite{Klingler}, and Gao's enlarged mixed Shimura varieties in \cite{Gao19} as well as the growing field of unlikely intersections in arithmetic dynamics (see the survey \cite{ArDyn}). We further refer to Section 3 of \cite{GT} for various examples that show the difficulty of defining special subvarieties in the context of arithmetic dynamics. While a common method has been applied to studying unlikely intersections in our examples of distinguished categories, namely using the Pila-Zannier strategy to combine o-minimal point counting, Galois bounds, and functional transcendence results of Ax-Schanuel type (see \cite{Zannier}), the methods used in these other settings are sometimes very different. 

The paper is organized as follows. In Section \ref{sec:dc}, we define our main object of study, distinguished categories, and we see that many of the abovementioned examples form distinguished categories.

We devote Sections \ref{sec:msv}, \ref{sec:intersectionspecial}, and \ref{sec:axiomsformsv} to a detailed proof that the category of connected (mixed) Shimura varieties (of Kuga type) is distinguished. Unfortunately, slightly different definitions of connected mixed Shimura varieties appear in the literature. For this reason and since we have not been able to locate them in the literature, we decided to include proofs of some widely used and well-known facts about connected mixed Shimura varieties, for instance the fact that the irreducible components of an intersection of special subvarieties are special. We hope that this will make our text more accessible to readers who are not very familiar with this topic.

Many concepts from the field of unlikely intersections can be defined naturally in an arbitrary distinguished category: In Section \ref{sec:special}, we introduce the concept of (weakly) special subvarieties. Our definitions are inspired by Pink's in \cite{Pink}, with which they are formally identical for the special case of the distinguished category of connected mixed Shimura varieties, and they are equivalent to the usual definitions of (weakly) special subvarieties in the case of semiabelian varieties (see Remark \ref{rmk:equivalentdefinition}). In the same section, we prove basic facts about (weakly) special subvarieties (e.g., special subvarieties are weakly special, irreducible components of intersections of weakly special subvarieties are weakly special) that are well-known for commutative algebraic groups and for mixed Shimura varieties, although in the latter case we have not always been able to locate proofs in the literature. We also compare our approach with Ullmo's in \cite{Ullmo}.

The so-called defect condition, introduced in \cite{HP} and conjectured there to hold in any mixed Shimura variety, is often a useful technical tool when studying unlikely intersections: In Section \ref{sec:defcon}, we consider the concept of (weak) defect and prove the defect condition in an arbitrary distinguished category (special cases can be found in \cite{HP}, \cite{DawRen}, \cite{PilaTsimerman}, and \cite{Cass2}). The definition of defect is formally identical with the definition in \cite{PinkUnpubl}.

In Section \ref{sec:optimality}, we then consider the concept of (weak) optimality in an arbitrary distinguished category. The definition of optimality goes back to \cite{HP} while weak optimality for subvarieties of $\mathbb{G}^n_m$ already appears in \cite{Poizat} under the name of cd-maximality.

In Section \ref{sec:A5}, we formulate the axiom \ref{ax:5}, which informally speaking says that weakly optimal subvarieties for a given subvariety come in finitely many ``families". Given an extension $K \subset L$ of algebraically closed fields and a distinguished category over $K$, we can perform a base change to get a distinguished category over $L$. If \ref{ax:5} is satisfied in any base change of the original category with respect to an extension of algebraically closed fields of finite transcendence degree, we call the category very distinguished. In our main examples of very distinguished categories (abelian varieties, tori, semiabelian varieties, connected Shimura varieties, and connected mixed Shimura varieties of Kuga type over algebraically closed fields of finite transcendence degree over $\bar{\mathbb{Q}}$), the fact that they are very distinguished is usually deduced from a functional transcendence result of weak Ax-Schanuel type (called ``Weak Complex Ax" in \cite{HP}) together with the facts that the corresponding uniformization map is definable in an o-minimal structure and that a countable definable set is finite.

In Section \ref{sec:zpext}, we formulate a statement for arbitrary distinguished categories that corresponds to the Zilber-Pink conjecture. We then show in Theorem \ref{thm:fieldofdef} that the Zilber-Pink statement for the base change of a very distinguished category with respect to an extension of algebraically closed fields follows from the same statement for the original category.

This reduction of the transcendence degree of the ground field in the Zilber-Pink statement is similar to the main result of \cite{BD}, where we reduced the Zilber-Pink conjecture for abelian varieties from $\mathbb{C}$ to $\bar{\mathbb{Q}}$. Note however that we do not recover the main result of \cite{BD}: If the ambient abelian variety $A$ cannot be defined over $\bar{\mathbb{Q}}$, then the method we present here does not reduce the Zilber-Pink conjecture for $A$ to the Zilber-Pink conjecture for an abelian variety over $\bar{\mathbb{Q}}$. Our method however allows to reduce to the case where the subvariety is defined over the algebraic closure of the field of definition of the abelian variety. The analogous statement for powers of the multiplicative group has been proven by Bombieri, Masser and Zannier in \cite{BMZ08}. Our proof of Theorem \ref{thm:fieldofdef} takes its inspiration from the proof of Proposition 4.1 in \cite{BD}.

Recently, in \cite{PilaScanlon}, Pila and Scanlon proved an effective Zilber-Pink result over a transcendental field extension of $\mathbb{C}$ as well as a reduction of the transcendental to the algebraic case in this context, considering only varieties associated to certain quotient spaces, transcendental points, and strongly special subvarieties. In their reduction step, they have to assume a formulation of the conjecture over the ground field that is different from what they obtain over the field extension (but we will prove in this paper that it often follows from what they obtain over the field extension, see Section \ref{sec:zilberpink}). They suggest a method to remove this restriction over the ground field $\mathbb{C}$ without using our work. However, the dimension of the variety that they need to apply the conjecture to over the ground field can be bigger than the dimension of the variety they started with. This is a fundamental difference to our method, which yields more refined results.

Theorem \ref{thm:fieldofdef} also contains another statement that implies for example that the Zilber-Pink conjecture holds for any complex curve that cannot be defined over $\bar{\mathbb{Q}}$ inside the moduli space $\mathcal{A}_2$ of principally polarized abelian surfaces over $\bar{\mathbb{Q}}$. We denote the base change of a $K$-variety $V$ with respect to a field extension $K \subset L$ by $V_L$ (see Definition \ref{def:basechange}).

The following theorem is a special case of Corollary \ref{cor:uncondzp}.

\begin{thm}\label{thm:curveina2}
Let $K$ be an algebraically closed field. Let $C$ be a curve in $(\mathcal{A}_2)_K$ that is not the base change of a curve in $\mathcal{A}_2$ (so cannot be defined over $\bar{\mathbb{Q}}$) and is not contained in a proper special subvariety of $(\mathcal{A}_2)_K$. Then, $C$ contains at most finitely many points that lie on a special subvariety of $(\mathcal{A}_2)_K$ of codimension at least 2.
\end{thm}

So far, only partial results for curves that are defined over $\bar{\mathbb{Q}}$ inside $\mathcal{A}_2$ were known (see for example \cite{DawOrr1} and \cite{DawOrr2}). Our proof does not use any information about $\mathcal{A}_2$ apart from the fact that it is a connected Shimura variety of dimension $3$ over $\bar{\mathbb{Q}}$ for which the Andr\'e-Oort conjecture is known to hold. Thus, we also recover Pila's Theorem 1.4 in \cite{PilaFermat}, which says that the Zilber-Pink conjecture holds for any complex curve that cannot be defined over $\bar{\mathbb{Q}}$ inside the cube of the moduli space of elliptic curves.
\medskip

Let us now see an example. We consider the hyperelliptic curve of genus two defined by the equation
$$
y^2=x(x-1)(x-\pi t)(x-t^2)(x-t^4)
$$
over the field $K(t)$, where $K=\overline{\mathbb{Q}(\pi)} \subset \mathbb{C}$ and $t$ is an indeterminate.
Let $J_t$ be its Jacobian.
This gives a curve $C$ in $(\mathcal{A}_2)_K$ that cannot be defined over $\bar{\mathbb{Q}}$. There exists a finite set $F$ such that, for each $t_0 \in \mathbb{C}\setminus F$, there is a specialized hyperelliptic curve and a specialized Jacobian $J_{t_0}$ over $\mathbb{C}$.

By the Theorem in \cite{Masser99}, specializing $\pi$ to 1, we know that  $J_t$ only has trivial endomorphisms, even over an algebraic closure of $K(t)$. This implies (see Table 1 on p.~11 of \cite{DawOrr1}) that $C$ is not contained in any proper special subvariety of $(\mathcal{A}_2)_K$. Theorem \ref{thm:curveina2} then implies that the intersection of $C$ with the union of all special subvarieties of $(\mathcal{A}_2)_K$ that have codimension at least 2 is finite. This in turn implies that there are at most finitely many complex numbers $t_0$ such that at least one of the following holds:
\begin{enumerate}
	\item $J_{t_0}$ is simple and its endomorphism ring is a $\mathbb{Z}$-module of rank four (i.e., $J_{t_0}$ has quaternionic or complex multiplication),
		\item $J_{t_0}$ is isogenous to the square of an elliptic curve, or
		\item $J_{t_0}$ is isogenous to the product of an elliptic curve and a CM elliptic curve.
\end{enumerate}

In Sections \ref{sec:redoptsing} and \ref{sec:zilberpink}, we prove (again assuming \ref{ax:5} most of the time) that the Zilber-Pink statement is equivalent to several seemingly weaker statements. In particular, while the statement is concerned with optimal subvarieties of arbitrary dimension, we show in Theorem \ref{thm:redoptsing} that it suffices to prove it for optimal singletons. This generalizes Theorem 8.3 in \cite{DawRen} (for Shimura varieties) and Theorem 6.1 in \cite{BD} (for abelian varieties); the latter theorem is an easy consequence of the results in \cite{HP}.

Furthermore, the Zilber-Pink statement we commonly use in this article generalizes Zilber's formulation of the conjecture. However, we show in Theorem \ref{thm:zilberequivpink} that it is equivalent to the corresponding generalization of Pink's formulation of the conjecture. This generalizes Theorem 1.9 in \cite{BD}, where we showed (the non-trivial direction of) this equivalence in the case of abelian varieties.
In the toric case, Ullmo had pointed out already earlier that the two formulations are equivalent (see p.~320 of \cite{BMZ08}) while Zannier mentioned that this equivalence can be proved by imposing additional multiplicative relations on the positive-dimensional atypical intersections (see p.~34 of \cite{Zannier}).

In both Theorem \ref{thm:redoptsing} and Theorem \ref{thm:zilberequivpink}, we have to assume that \ref{ax:5} is satisfied. 
Statements similar to Theorem \ref{thm:redoptsing} and Theorem \ref{thm:zilberequivpink} for connected mixed Shimura varieties of Kuga type can be found in the work of Cassani \cite{Cass3}.

In Section \ref{sec:legendre}, we apply our results to prove that the Zilber-Pink conjecture holds for a complex curve in a fibered power of the Legendre family of elliptic curves, which is a connected mixed Shimura variety of Kuga type. This was previously known only if the curve can be defined over $\bar{\mathbb{Q}}$, in which case it follows from combining the results of \cite{RemVia}, \cite{Viada2008}, \cite{galateau2010}, \cite{BC}, and \cite{Superbarro}. We use these results together with our Theorem \ref{thm:fieldofdef} and the fact that connected mixed Shimura varieties of Kuga type over $\bar{\mathbb{Q}}$ form a very distinguished category to deduce the $\mathbb{C}$-case.

The theorem below is a consequence of Theorem \ref{thm:legendre}.

\begin{thm}
	Let $K$ be an algebraically closed field and let $\mathcal{E}^g$ be the $g$-th fibered power of the Legendre family of elliptic curves over $\bar{\mathbb{Q}}$. Let $C$ be a curve in $\mathcal{E}^g_K$ that is not contained in a proper special subvariety of $\mathcal{E}^g_K$. Then, $C$ contains at most finitely many points that lie on a special subvariety of $\mathcal{E}^g_K$ of codimension at least 2.
\end{thm}

In Section \ref{sec:semiabelian}, we show that the Zilber-Pink conjecture holds for a complex curve in a semiabelian variety, provided that the semiabelian variety can be defined over $\bar{\mathbb{Q}}$.

\begin{thm}
	Let $K$ be an algebraically closed field and let $G$ be a semiabelian variety over $\bar{\mathbb{Q}}$. Let $C$ be a curve in $G_K$ that is not contained in a proper algebraic subgroup of $G_K$. Then $C$ contains at most finitely many points that lie on an algebraic subgroup of $G_K$ of codimension at least 2.
\end{thm}

This follows from combining our Theorem \ref{thm:fieldofdef} and the fact that semiabelian varieties over $\bar{\mathbb{Q}}$ form a very distinguished category with recent work \cite{BarKS} (see Theorem \ref{thm:semiabelian}) of the first-named author, K\"uhne, and Schmidt.

\section{Distinguished categories}\label{sec:dc}

Let $K$ be an algebraically closed field of characteristic $0$. Consider a category $\mathfrak{C}$ with objects $\mathfrak{V}$ and morphisms $\mathfrak{M}$ together with a covariant functor $\mathcal{F}$ from $\mathfrak{C}$ to the category of varieties over $K$. Typically, this functor will correspond to forgetting some additional information.

We formulate the following axioms that this category may or may not satisfy:

\begin{enumerate}[label={(A\arabic*)}]
\item \label{ax:1}\emph{Direct Products} - If $X,Y \in \mathfrak{V}$, then there exists $Z \in \mathfrak{V}$ and there exist morphisms $\pi_X:Z \to X$, $\pi_Y:Z \to Y$ in $\mathfrak{M}$ such that the morphism $\mathcal{F}(Z) \to \mathcal{F}(X) \times_K \mathcal{F}(Y)$ induced by the morphisms $\mathcal{F}(\pi_X)$ and $\mathcal{F}(\pi_Y)$ is an isomorphism. We identify $\mathcal{F}(Z)$ with $\mathcal{F}(X) \times_K \mathcal{F}(Y)$. Furthermore, if $\phi: W \to X$ and $\psi: W \to Y$ are morphisms in $\mathfrak{M}$, then there exists a unique morphism $\chi: W \to Z$ in $\mathfrak{M}$ such that $\pi_X\circ \chi=\phi$ and $\pi_Y\circ \chi=\psi$. We write $X \times Y$ or $X \times_K Y$ for $Z$ and $(\phi,\psi)$ for $\chi$. 
\item\label{ax:2} \emph{Fibered Products} - If $\phi: X \to Z$ and $\psi: Y \to Z$ are morphisms in $\mathfrak{M}$, then there exists $n \in \mathbb{Z}$, $n \geq 0$, and there exist $X_1,\hdots,X_n \in \mathfrak{V}$ and morphisms $\phi_i: X_i \to X \times Y$ in $\mathfrak{M}$ ($i=1,\hdots,n$) such that $\bigcup_{i=1}^{n}{\mathcal{F}(\phi_i)(\mathcal{F}(X_i))} = (\mathcal{F}(X) \times_{\mathcal{F}(Z)} \mathcal{F}(Y))_{\red} \subset \mathcal{F}(X) \times_K \mathcal{F}(Y)$.
\item\label{ax:3} \emph{Final Object} - The category has a final object that is mapped to $\Spec K$ by $\mathcal{F}$.
\item \label{ax:4} \emph{Fiber Dimension} - If $\phi: X \to Y$ is a morphism in $\mathfrak{M}$, then $\mathcal{F}(\phi)(\mathcal{F}(X))$ is closed in $\mathcal{F}(Y)$ and there exist morphisms $\phi_1: W \to X$, $\phi_2: W \to Z$, and $\phi_3: Z \to Y$ in $\mathfrak{M}$ such that $\phi \circ \phi_1 = \phi_3 \circ \phi_2$, $\mathcal{F}(\phi_1)$ is finite and surjective, $\mathcal{F}(\phi_3)$ has finite fibers, $\mathcal{F}(\phi_2)$ is surjective, $\mathcal{F}(\phi_2)^{-1}(z)$ is irreducible for all $z \in \mathcal{F}(Z)$, and
$\dim_w \mathcal{F}(\phi_2)^{-1}(\mathcal{F}(\phi_2)(w))$ is a constant function for all $w \in \mathcal{F}(W)$.
\end{enumerate}

In other words, \ref{ax:4} says that the image of $\mathcal{F}(\phi)$ is closed and that $\phi$ factorizes as in the following diagram:
\begin{equation} \label{diag}
\begin{tikzcd}
W \arrow[d,"\substack{\text{finite},\\ \text{surjective}}","\phi_1"'] \arrow[r,"\substack{\text{surjective,} \\ \text{fibers are irreducible} \\ \text{of constant dimension}}","\phi_2"'] &[7em] Z \arrow[d,"\text{finite fibers}","\phi_3"'] \\
X \arrow[r,"\phi"] & Y,
\end{tikzcd}
\end{equation}
where the properties in the diagram are properties of the images of the morphisms under $\mathcal{F}$. Note that \ref{ax:4} implies that $\dim_x \mathcal{F}(\phi)^{-1}(\mathcal{F}(\phi)(x))$ is a constant function for $x \in (\mathcal{F}(X))(K)$ and therefore for $x\in \mathcal{F}(X)$ (see Appendix E of \cite{GW}). By the same Appendix, it suffices to show that $\mathcal{F}(\phi_2)$ is surjective and that all fibers of $\mathcal{F}(\phi_2)$ over closed points of $\mathcal{F}(Z)$ are irreducible of constant dimension in order to verify the properties of $\mathcal{F}(\phi_2)$ in \ref{ax:4}.

\begin{defn}
If $\mathfrak{C}$ satisfies \ref{ax:1} to \ref{ax:4}, we call it a \emph{distinguished category}. The elements of $\mathfrak{V}$ (which we will often identify with their images under $\mathcal{F}$) will then be called \emph{distinguished varieties} while the elements of $\mathfrak{M}$ (which we will often similarly identify) will be called \emph{distinguished morphisms}.
\end{defn}

The domain and codomain of a distinguished morphism are automatically distinguished varieties.

\begin{rmk}
We are going to see later that these axioms are all necessary in order to define fundamental concepts and prove basic facts. For the reader who is already familiar with the topic, we summarize here how they will be used:
\begin{enumerate}
	\item[\ref{ax:1}] is needed to introduce \ref{ax:2}.
	\item[\ref{ax:2}] is equivalent to the fact that irreducible components of intersections of special subvarieties are special.
	\item[\ref{ax:3}] implies that special subvarieties are weakly special.
	\item[\ref{ax:4}] together with \ref{ax:2} implies that irreducible components of (pre-)images of (weakly) special subvarieties under distinguished morphisms are (weakly) special and irreducible components of intersections of weakly special subvarieties are weakly special.
	The fact that the image is closed is necessary in order to be able to define special subvarieties at all. The constant fiber dimension is fundamental in the proof of the defect condition.
\end{enumerate}
\end{rmk}

We are going to use the following notation throughout the article.

\begin{defn}\label{def:basechange}
	If $K \subset L$ is an arbitrary field extension and $V$ is a finite union of varieties over $K$, then $V_L = V\times_{K} L$ is called the base change of $V$ to $L$. We use analogous notation for the base change of morphisms between varieties.
\end{defn}

 We now give some examples of distinguished categories. If $\mathcal{F}$ is not mentioned, then it is to be understood that $\mathfrak{C}$ is a subcategory of the category of varieties over $K$ and $\mathcal{F}$ is just the inclusion functor.

A trivial but instructive example is the following: Let $X$ be an arbitrary fixed variety over $K$. The category $\mathfrak{C}_{\mathrm{triv}}(X)$ with objects $\mathfrak{V} = \{X^n; n \in \mathbb{Z}, n \geq 0\}$ and morphisms $\mathfrak{M} =\{ \mathrm{maps } \ \ X^m \to X^n,(x_1,\hdots,x_m) \mapsto (x_{i_1},\hdots,x_{i_n})\}$ is distinguished.

The category of connected commutative algebraic groups gives an interesting example of a distinguished category which we denote by $\mathfrak{C}_{\mathrm{comm}}$. To be more precise, if we take $\mathfrak{V}$ to be the connected commutative algebraic groups over $K$, $\mathfrak{M}$ to consist of homomorphisms of algebraic groups composed with translations by torsion points, and $\mathcal{F}$ to be the forgetful functor, we can easily see that axioms \ref{ax:1} to \ref{ax:4} follow from well-known properties of algebraic groups and homomorphisms between them.

In particular, also in view of \cite{BD}, one can consider the full subcategories $\mathfrak{C}_{\mathrm{ab}}$ and $\mathfrak{C}_{\mathrm{semiab}}$ of $\mathfrak{C}_{\mathrm{comm}}$, consisting of (semi-)abelian varieties over $K$, or even just $\mathfrak{C}_{\mathrm{tor}}$ given by $\mathfrak{V} = \{\mathbb{G}^n_{m,K}; n \in \mathbb{Z}, n \geq 0\}$ or $\mathfrak{C}_{\mathrm{add}}$ given by $\mathfrak{V} = \{\mathbb{G}^n_{a,K}; n \in \mathbb{Z}, n \geq 0\}$.

In the next section we are going to see further examples of distinguished categories. We end this section by noting that, given two distinguished categories $\mathfrak{C}$ and $\mathfrak{C}'$ over an algebraically closed field $K$, we can form a product category $\mathfrak{C} \times \mathfrak{C}'$ that is again distinguished: Its objects are pairs of objects of $\mathfrak{C}$ and $\mathfrak{C}'$, its morphisms pairs of morphisms of $\mathfrak{C}$ and $\mathfrak{C}'$, and its functor is defined by applying to each element of the pair of objects or morphisms the functor of the corresponding distinguished category and taking the direct product of the two resulting varieties or the morphism between direct products induced by the two resulting morphisms of varieties respectively.

\section{Connected mixed Shimura varieties}\label{sec:msv}
  
In this section together with the next two ones, we are going to prove that connected mixed Shimura varieties and the subcategories consisting of the pure and the Kuga type ones fit in our framework and satisfy axioms \ref{ax:1} to \ref{ax:4}.
 
 Different definitions of a (connected) mixed Shimura datum appear in the literature. We use the following. We denote the Deligne torus $\Res_{\mathbb{C}/\mathbb{R}}\mathbb{G}_{m,\mathbb{C}}$ by $\mathbb{S}$ and we identify $\mathbb{S}(\mathbb{R})$ with $\mathbb{C}^{\ast}$. We fix once and for all a complex square root of $-1$ that we denote by $\sqrt{-1}$. For the definition of a rational mixed Hodge structure, its type, and its weight filtration, we refer to Chapter 1 of \cite{PinkDiss}. The definition of a Cartan involution can be found on pp. 274--275 of \cite{MilneISV}.
 
 \begin{defn}\label{defn:msv}
 	A \emph{connected mixed Shimura datum} is a pair $(P,X^+)$, where $P$ is a connected linear algebraic group over $\mathbb{Q}$ with unipotent radical $W$ and another algebraic subgroup $U \subset W$ that is normal in $P$ and uniquely determined by $X^+$ and condition \ref{ax:MSDc} below, and $X^+ \subset \Hom(\mathbb{S}_{\mathbb{C}},P_{\mathbb{C}})$ is a connected component of an orbit under conjugation (from the left) by the subgroup $P(\mathbb{R})  U(\mathbb{C}) \subset P(\mathbb{C})$ such that for some (or equivalently for all) $x \in X^+$ we have that
 	\begin{enumerate}[label=(MSD.\alph*)]
 		\item \label{ax:MSDa}the composite homomorphism $\mathbb{S}_{\mathbb{C}} \stackrel{x}{\to} P_{\mathbb{C}} \to (P/U)_{\mathbb{C}}$ is defined over $\mathbb{R}$,
 		\item  \label{ax:MSDb}the adjoint representation induces on $\Lie P$ a rational mixed Hodge structure of type
 		\[ \{(-1,1),(0,0),(1,-1),(-1,0),(0,-1),(-1,-1)\},\]
 		\item  \label{ax:MSDc}the weight filtration on $\Lie P$ is given by
 		\[ W_n(\Lie P) = \left\{
 		\begin{array}{ll}
 		\{0\} & \mbox{ if $n < -2$,}\\
 		\Lie U & \mbox{ if $n = -2$,}\\
 		\Lie W & \mbox{ if $n = -1$,}\\
 		\Lie P & \mbox{ if $n \geq 0$},
 		\end{array}
 		\right.\]
 		\item  \label{ax:MSDd}the conjugation by $x(\sqrt{-1})$ induces a Cartan involution on $(P/W)^{\ad}_{\mathbb{R}}$ and $(P/W)^{\ad}$ possesses no non-trivial $\mathbb{Q}$-factor $H$ such that $H(\mathbb{R})$
 		is compact, and
 		\item  \label{ax:MSDe}$P/P^{\der}$ is an almost-direct product of a $\mathbb{Q}$-split torus with a torus of compact type defined over $\mathbb{Q}$.
 	\end{enumerate}
 	
 	If $U=\{1\}$, then $(P,X^+)$ is called \emph{connected mixed Shimura datum of Kuga type}, while if $W=\{1\}$ it is called \emph{connected pure Shimura datum} or \emph{connected Shimura datum}.
 \end{defn}
 
 In Definition 2.4(1) in \cite{G17}, $X^+$ is taken to be a $P(\mathbb{R})^+  U(\mathbb{C})$-orbit instead of a connected component of a $P(\mathbb{R})  U(\mathbb{C})$-orbit, where $P(\mathbb{R})^+$ denotes the identity component of $P(\mathbb{R})$. Definition \ref{defn:msv} above is equivalent to Definition 2.4(1) in \cite{G17}. Definition 2.1 in \cite{Pink} is more restrictive than our definition here as Pink additionally demands that $P$ possesses no proper normal subgroup $P'$, defined over $\mathbb{Q}$, such that $x$ factors through $P'_{\mathbb{C}} \subset P_{\mathbb{C}}$. Note that the second part of \ref{ax:MSDd} seems to be missing in Definition 2.1 in \cite{Pink}, but follows from Pink's additional condition (vi) together with the rest. Definition 2.1 in \cite{G18} is equal to Definition 2.1 in \cite{Pink} in the case of connected mixed Shimura data of Kuga type. Definition 2.1 in \cite{PinkDiss} is a slightly more general analogue of our definition for mixed Shimura data instead of connected mixed Shimura data.
 
 We now let $P(\mathbb{R})_{+}$ be the stabilizer of $X^+$ in $P(\mathbb{R})$ (it contains $P(\mathbb{R})^{+}$), and we let $\Gamma$ be a congruence subgroup of $P(\mathbb{Q})$ that is contained in $P(\mathbb{Q})_{+} := P(\mathbb{Q}) \cap P(\mathbb{R})_{+}$. By abuse of notation, we will say in this situation that $\Gamma$ is a congruence subgroup of $P(\mathbb{Q})_{+}$. By Proposition 2.2 in \cite{RohlfsSchwermer}, there always exists such a congruence subgroup. It follows that, for every congruence subgroup of $P(\mathbb{Q})$, its intersection with $P(\mathbb{Q})_{+}$ is a congruence subgroup of $P(\mathbb{Q})$ that is contained in $P(\mathbb{Q})_{+}$.

 By \cite{PinkDiss}, 1.18(a), there is a canonical complex structure on $X^+$. There exists an algebraic variety $V$ over $\bar{\mathbb{Q}}$ such that $\Gamma\backslash X^{+}$ is canonically isomorphic to the analytification of $V_{\mathbb{C}}$ as a complex analytic space (see \cite{PinkDiss}, Proposition 3.3, Proposition 9.24, Definition 9.25, and Theorem 11.18).
 
 The variety $V$ is called a \emph{connected mixed Shimura variety}. If $(P,X^+)$ is of Kuga type or pure, then $V$ is called \emph{connected mixed Shimura variety of Kuga type} or \emph{connected (pure) Shimura variety} respectively.
 
 Condition \ref{ax:MSDe} is usually dropped in the definition of a connected (pure) Shimura datum/variety. Nevertheless, in this article, connected (pure) Shimura data will satisfy \ref{ax:MSDe}. We will see that this does not cause any problems thanks to Remarks \ref{rmk:quotientshimuradatum} and \ref{rmk:shimurasubdatum}. If $(P,X^+)$ satisfies \ref{ax:MSDa} to \ref{ax:MSDd} with $W = \{1\}$, then we can consider its quotient by the center of $P$ (see Remark \ref{rmk:quotientshimuradatum}) and this quotient will satisfy \ref{ax:MSDe}, so this condition is not a big restriction in applications to unlikely intersections.
 
 \begin{defn}
 Let $(P,X^+)$ and $(P',X'^+)$ be connected mixed Shimura data. If we have a homomorphism $\phi: P \to P'$ of algebraic groups over $\mathbb{Q}$ which induces a map $X^+ \to X'^+$ through $x \mapsto \phi_{\mathbb{C}} \circ x$, we say that $\phi$ induces a \emph{Shimura morphism} $(P,X^+)\to(P',X'^+)$ of connected mixed Shimura data.
 \end{defn}

Suppose we are given a Shimura morphism $(P,X^+)\to(P',X'^+)$ induced by $\phi: P \to P'$ and congruence subgroups $\Gamma \subset P(\mathbb{Q})_{+}$ and $\Gamma' \subset P'(\mathbb{Q})_{+}$ with $\phi(\Gamma) \subset \Gamma'$. Then, this induces a holomorphic map $\Gamma\backslash X^+ \to \Gamma' \backslash X'^+$ between complex analytic spaces that in turn gives a morphism between the corresponding connected mixed Shimura varieties (see \cite{PinkDiss}, 3.4, Proposition 9.24, Proposition 11.10, and Theorem 11.18).
 
 We are now in place to define a distinguished category $\mathfrak{C}_{\mathrm{mSv}}$ over $\bar{\mathbb{Q}}$, where
 \begin{itemize}
 	\item $\mathfrak{V}_{\mathrm{mSv}}$ consists of triples $(P,X^{+},\Gamma)$ where $(P,X^+)$ is a connected mixed Shimura datum and $\Gamma$ is a congruence subgroup of $P(\mathbb{Q})_{+}$;
 	\item $\mathfrak{M}_{\mathrm{mSv}}$ consists of arrows $ (P,X^+,\Gamma) \to (P',X'^+,\Gamma')$ given by a Shimura morphism $(P,X^+)\to(P',X'^+)$ induced by a homomorphism of algebraic groups $\phi: P \to P'$ that satisfies $\phi(\Gamma) \subset \Gamma'$ and composition of arrows is given by composition of Shimura morphisms;
 	\item the functor $\mathcal{F}_{\mathrm{mSv}}$ sends an element $(P,X^{+},\Gamma)$ of $\mathfrak{V}_{\mathrm{mSv}}$ to the corresponding connected mixed Shimura variety $V$ over $\bar{\mathbb{Q}}$. Moreover, it sends an element $(P,X^+,\Gamma) \to (P',X'^+,\Gamma')$ of $\mathfrak{M}_{\mathrm{mSv}}$ to the induced morphism of connected mixed Shimura varieties.
 \end{itemize}
   
By restricting the above definition to connected mixed Shimura data of Kuga type or connected pure Shimura data, we can also define the full subcategories $\mathfrak{C}_{\mathrm{mSvK}}$ and $\mathfrak{C}_{\mathrm{pSv}}$ of $\mathfrak{C}_{\mathrm{mSv}}$. By abuse of notation, we also call the elements of $\mathfrak{M}_{\mathrm{mSv}}$, $\mathfrak{M}_{\mathrm{mSvK}}$, and $\mathfrak{M}_{\mathrm{pSv}}$ as well as their images under the respective functor Shimura morphisms.

We would like to remark that we cannot just take $\mathfrak{C}_{\mathrm{mSv}}$ to be a subcategory of the category of varieties over $\bar{\mathbb{Q}}$ and take for $\mathcal{F}_{\mathrm{mSv}}$ the inclusion functor because of the following: There exists a trivial connected Shimura datum that is associated to the trivial reductive group and admits Shimura morphisms from any other connected Shimura datum. Its associated connected Shimura variety consists only of a point. There are however also non-trivial connected Shimura data that are associated to certain (positive-dimensional) tori and do not admit Shimura morphisms from all other connected Shimura data, and in particular not from the trivial connected Shimura datum, but whose associated connected Shimura varieties nevertheless consist only of a point.

On the other hand, these non-trivial connected Shimura data can sometimes also admit Shimura morphisms to connected Shimura data with positive-dimensional associated connected Shimura varieties (leading to special points on these), which the trivial connected Shimura datum cannot. This means that the existence or non-existence of Shimura morphisms between connected Shimura varieties depends not only on the variety structure, but also on the associated connected Shimura datum and congruence subgroup (and non-isomorphic pairs of connected Shimura data and congruence subgroups might yield connected Shimura varieties that are isomorphic as varieties). All of this holds also for connected mixed Shimura varieties or connected mixed Shimura varieties of Kuga type instead of connected Shimura varieties.

In the next two sections we are going to prove the following proposition:

\begin{prop}\label{prop:axiomsformsv}
The categories $\mathfrak{C}_{\mathrm{mSv}}$, $\mathfrak{C}_{\mathrm{mSvK}}$, and $\mathfrak{C}_{\mathrm{pSv}}$ satisfy the axioms \ref{ax:1} to \ref{ax:4}.
\end{prop}

Proposition \ref{prop:axiomsformsv} will allow us to forget the very technical definition of a connected (mixed) Shimura variety (of Kuga type) for a large part of this article and work only with the axioms \ref{ax:1} to \ref{ax:4}.

We conclude this section with some definitions and a proposition concerning Shimura morphisms.

\begin{defn}\label{def:shimuramorphisms} Let $\phi$ be a Shimura morphism $(P,X^+,\Gamma)\to (Q,Y^+,\Delta)$. By abuse of notation, we also denote the corresponding homomorphism of algebraic groups $P\to Q$ by $\phi$. We say that $\mathcal{F}_{\mathrm{mSv}}(\phi):\mathcal{F}_{\mathrm{mSv}}(P,X^+,\Gamma)\to \mathcal{F}_{\mathrm{mSv}}(Q,Y^+,\Delta)$ is a
	\begin{enumerate}
		\item \emph{Shimura submersion} if the image of $\phi $ contains $Q^{\der}$,
		\item \emph{quotient Shimura morphism} if $\phi: P \to Q$ is surjective and in this case we call $(Q,Y^+)$ a \emph{quotient Shimura datum} of $(P,X^+)$,
		\item \emph{Shimura immersion} if the identity component of $\ker \phi$ is a torus,
		\item \emph{Shimura embedding} if $\phi: P \to Q$ is injective and in this case we call $(P,X^+)$ a \emph{Shimura subdatum} of $(Q,Y^+)$, and
		\item \emph{Shimura covering} if it is a Shimura submersion and immersion.
	\end{enumerate}
\end{defn}

These definitions concern morphisms between connected mixed Shimura varieties but actually depend on the underlying homomorphisms of algebraic groups (and only on these). With a slight abuse of language, we also call Shimura submersions, quotient Shimura morphisms, Shimura immersions, embeddings, and coverings the corresponding Shimura morphisms between connected mixed Shimura data as well as the corresponding elements of $\mathfrak{M}_{\textrm{mSv}}$.

\begin{rmk}\label{rmk:quotientshimuradatum}
In Proposition 2.9 in \cite{PinkDiss}, the quotient of a mixed Shimura datum $(P,X)$ by a normal algebraic subgroup $K$ of $P$ is constructed; it satisfies a certain natural universal property. Thanks to Corollary 2.12 in \cite{PinkDiss}, this construction also yields quotients $(P,X^+)/K$ of connected mixed Shimura data $(P,X^+)$ by such subgroups $K$ that are again connected mixed Shimura data. One can check that the abovementioned slight difference between our definition and Pink's definition in \cite{PinkDiss} is preserved under taking such quotients. Furthermore, one can check that a Shimura morphism $(P,X^+) \to (Q,Y^+)$ is a quotient Shimura morphism if and only if it induces an isomorphism between $(P,X^+)/K$ and $(Q,Y^+)$, where $K = \ker(P \to Q)$.
\end{rmk}

We recall some results from \cite{Pink} that we will use when verifying that \ref{ax:4} holds for our three categories $\mathfrak{C}_{\mathrm{mSv}}$, $\mathfrak{C}_{\mathrm{mSvK}}$, and $\mathfrak{C}_{\mathrm{pSv}}$ and in Section \ref{sec:A5}. 

\begin{prop}[\cite{Pink}, Facts 2.6]\label{prop:pinkfacts} The following hold:
	\begin{enumerate}
		\item Shimura submersions are surjective.
		\item Shimura immersions are finite.
	\end{enumerate}
\end{prop}

As no proof of these facts is provided in \cite{Pink} and the definition of a connected mixed Shimura datum in \cite{Pink} is slightly different from our definition here, we give a proof. It relies on the proof of the proposition in the pure case that appears in \cite{OrrThesis}.

\begin{proof}
	Let $\phi: P \to Q$ be a homomorphism of connected linear algebraic groups over $\mathbb{Q}$ that induces a Shimura morphism of connected mixed Shimura data $(P,X^+) \to (Q,Y^+)$ and thereby induces a Shimura morphism $\mathcal{F}_{\textrm{mSv}}(P,X^+,\Gamma) \to \mathcal{F}_{\textrm{mSv}}(Q,Y^+,\Delta)$. We first note that surjectivity and finiteness may be checked after base change to $\mathbb{C}$. This is guaranteed by Propositions 14.48 and 14.51 in \cite{GW} (the morphism $\Spec \mathbb{C} \to \Spec \bar{\mathbb{Q}}$ is faithfully flat and quasi-compact, see \cite{GW}, Definition 10.1 and right after Definition 14.7).  One can check the surjectivity over $\mathbb{C}$ on $\mathbb{C}$-points thanks to Theorem 10.70 in \cite{GW}.
	
	We start with (1). Suppose that $Q^{\mathrm{der}} \subset \phi(P)$, which means that the Shimura morphism $\mathcal{F}_{\textrm{mSv}}(P,X^+,\Gamma) \to \mathcal{F}_{\textrm{mSv}}(Q,Y^+,\Delta)$ induced by $\phi$ is a Shimura submersion. By Proposition 2.9 in \cite{G17}, $\phi$ induces a Shimura morphism of connected pure Shimura data
	\[(G_P,X_G^+) := (P,X^+)/W_P \to (Q,Y^+)/W_Q =: (G_Q,Y_G^+),\]
	where $W_P$ and $W_Q$ are the unipotent radicals of $P$ and $Q$ respectively and $(P,X^+)/W_P$ and $(Q,Y^+)/W_Q$ are the induced quotient Shimura data. The derived subgroup $G_Q^{\mathrm{der}}$ is contained in the image of the induced homomorphism $G_P \to G_Q$ and so this Shimura morphism is also a Shimura submersion. Let $U_Q \subset W_Q$ and $U_P \subset W_P$ be the subgroups from the definition of a connected mixed Shimura datum. By Proposition 2.9 in \cite{G17}, we have $\phi(U_P) \subset U_Q$.
	
	By \ref{ax:MSDe} we have that the quotient $Q/Q^{\der}$ is of multiplicative type. Moreover, the image of $W_Q$ in such a quotient is unipotent and therefore trivial by \cite{MilneAG}, Proposition 14.16.	
	Therefore, we have $W_Q \subset Q^{\der}\subset\phi(P)$. 
	
	Now we claim that $\phi|_{U_P}: U_P \to U_Q$ is surjective. If this is the case, then it follows from Proposition 5.1 in \cite{MilneISV} that $W_Q(\mathbb{R})U_Q(\mathbb{C}) \subset\phi(P(\mathbb{R})_{+}U_P(\mathbb{C}))$.
	Moreover, by the proof of Theorem 2.4(3) in \cite{OrrThesis}, the map $X_G^+ \to Y_G^+$ is surjective and $X^+\to X_G^+$ is surjective by Proposition 5.1 in \cite{MilneISV}.
	Then \cite{PinkDiss}, 2.18, implies that the map $X^+ \to Y^+$ is surjective and therefore the induced map $\Gamma \backslash X^+ \to \Delta \backslash Y^+$ is surjective as well, from which (1) follows.

	We are left with proving that $\phi|_{U_P}: U_P \to U_Q$ is surjective. We argue at the level of Lie algebras. The torus $\mathbb{S}_{\mathbb{C}}$ acts on $\Lie P$ through the adjoint representation. We have $\mathbb{S}_\mathbb{C}$-invariant subspaces $\Lie \ker \phi \subset \Lie  \phi^{-1}(U_Q)$ and this gives an $\mathbb{S}_\mathbb{C}$-invariant subspace $\Lie \phi^{-1}(U_Q)/\Lie \ker \phi$ of $\Lie P/\Lie\ker \phi$. Now $\Lie \phi^{-1}(U_Q)/\Lie \ker \phi$ is isomorphic to $\Lie(U_Q \cap \phi(P))$ by Proposition 1.63 in \cite{MilneAG}, so the action of $\mathbb{S}_\mathbb{C}$ on it is pure of type $\{(-1,-1)\}$ (as defined in \cite{PinkDiss}, 1.3, p.~9). Hence, $\Lie \phi^{-1}(U_Q)/\Lie \ker \phi$ must be contained in the $(-1,-1)$-subspace of $\Lie P/\Lie \ker \phi$, which is equal to $\Lie (\ker \phi)U_P/\Lie \ker \phi$. It follows that $\Lie \phi^{-1}(U_Q)$ is contained in $\Lie (\ker \phi)U_P$. 
For the identity component $\phi^{-1}(U_Q)^0$, we have
	\[ \Lie(\phi^{-1}(U_Q)^0 \cap (\ker \phi)U_P) = \Lie \phi^{-1}(U_Q)^0 \cap \Lie(\ker \phi)U_P= \Lie\phi^{-1}(U_Q)^0,\]
	 see \cite{MilneAG}, Chapter 10, p. 190. By \cite{MilneAG}, Proposition 10.15, $\phi^{-1}(U_Q)^0$ is contained in $(\ker \phi)U_P$. But $\phi(\phi^{-1}(U_Q))$ is connected (being unipotent), so $\phi(\phi^{-1}(U_Q)) = \phi(\phi^{-1}(U_Q)^0)$ and hence $\phi^{-1}(U_Q) = (\ker \phi) \phi^{-1}(U_Q)^0\subset (\ker \phi)U_P$.
	Since $U_Q \subset W_Q \subset \phi(P)$, this implies that $\phi|_{U_P}: U_P \to U_Q$ is surjective and we are done with (1).
	
	\medskip
	We turn to (2), using the same notation as in (1), but now $\phi$ induces a Shimura immersion, i.e., the identity component of $\ker \phi$ is a torus. We want to show that the map $\Gamma\backslash X^+ \to \Delta\backslash Y^+$ of complex analytic spaces is finite, i.e., proper and separated with finite fibers (note that separatedness is automatic). By GAGA (see \cite{SGA1}, Expos\'{e} XII, Proposition 3.2(vi)), this implies that the corresponding morphism of algebraic varieties over $\mathbb{C}$ is also finite.
	
	By Proposition 2.9 in \cite{G17}, we have $\phi(W_P) \subset W_Q$. The kernel of $\phi|_{W_P}$ is a unipotent algebraic group whose identity component is a torus. Hence it is trivial, so $\phi|_{W_P}$ is injective. Again, there is an induced Shimura morphism $(G_P,X_G^+) \to (G_Q,Y_G^+)$. The kernel of the induced homomorphism $G_P \to G_Q$ is equal to $\phi^{-1}(W_Q)/W_P$. Since $\phi|_{W_P}$ is injective, the restriction of $P \to G_P$ to $\ker \phi$ is an isomorphism onto its image $H\subset \phi^{-1}(W_Q)/W_P$. We have that $(\phi^{-1}(W_Q)/W_P)/H \simeq (W_Q \cap \phi(P))/\phi(W_P)$ is unipotent and normal in $G_P/H$. It follows from Theorem 22.42 in \cite{MilneAG} that $G_P/H$ is reductive. Therefore, we must have $\phi^{-1}(W_Q)/W_P = H \simeq \ker \phi$. So the induced Shimura morphism $(G_P,X_G^+) \to (G_Q,Y_G^+)$ is a Shimura immersion as well. Together with Theorem 2.4(2) in \cite{OrrThesis}, the injectivity of $\phi|_{W_P}$, and \cite{PinkDiss}, 2.18, this implies that the map $\Gamma \backslash X^+ \to \Delta \backslash Y^+$ has countable and therefore finite fibers. Note that, by Lemma 2.2 in \cite{OrrThesis} (which is formulated for reductive groups, but holds with the same proof for arbitrary linear algebraic groups) one can find congruence subgroups of $G_P(\mathbb{Q})_+$ and $G_Q(\mathbb{Q})_+$ such that we can apply Theorem 2.4(2) in \cite{OrrThesis}.
	
	A continuous map with compact -- e.g., finite -- fibers is proper if and only if it is closed. It therefore remains to show that the continuous map $\Gamma \backslash X^+ \to \Delta \backslash Y^+$ is closed.
		
	We choose a Levi decomposition $P=G_P\ltimes W_P$ (see Theorem 2.3 in \cite{PlatonovRapinchuk}). Let $i_P: G_P \to P$ denote the induced homomorphism of algebraic groups and let $\pi_P: P \to G_P$ denote the canonical quotient homomorphism. Fix $x_G \in X_G^+$. It follows from the surjectivity of $X^+ \to X^+_G$ that there exists $\tilde{x} \in X^+$ such that $(\pi_P)_{\mathbb{C}} \circ \tilde{x} = x_G$. By Proposition 2.17 in \cite{PinkDiss}, $i_P$ induces a morphism between the associated mixed Shimura data. Proposition 2.18 in \cite{PinkDiss} then implies that $x = (i_P)_{\mathbb{C}} \circ x_G \in W_P(\mathbb{R})U_P(\mathbb{C})\tilde{x} \subset X^+$. If $\omega: \mathbb{G}_{m,\mathbb{R}} \to \mathbb{S}$ is defined by $t\in \mathbb{R}^{\ast}\mapsto t \in \mathbb{C}^{\ast} = \mathbb{S}(\mathbb{R})$, then the composition $x \circ \omega_{\mathbb{C}} = (i_P)_{\mathbb{C}} \circ x_G\circ \omega_{\mathbb{C}}$ is the base change of a homomorphism $x_\omega: \mathbb{G}_{m,\mathbb{Q}}\to P$ thanks to Remark 2.2(i) in \cite{G17} (note that $W_P \subset P^{\der}$ and so the center of $G_P$ is isogenous to $P/P^{\der} \simeq G_P/(G_P)^{\der}$).
	
	Let $Z_P(x_\omega)$ denote the centralizer of $x_\omega(\mathbb{G}_{m,\mathbb{Q}})$ in $P$. The proof of Lemma 1.8 in \cite{PinkDiss} shows that there is a Levi decomposition $P = G_P \ltimes W_P$ such that $G_P \ltimes \{0\} = Z_P(x_\omega)$ and we will assume that we have chosen this as the Levi decomposition of $P$. If $Z(x)\subset (Z_P(x_\omega))(\mathbb{R})^+$ denotes the intersection with $P(\mathbb{R})^+U_P(\mathbb{C})$ of the centralizer of $x(\mathbb{S}_{\mathbb{C}})$ in $P_{\mathbb{C}}$ and $Z(x_G)$ denotes the intersection with $G_P(\mathbb{R})^+$ of the centralizer of $x_G(\mathbb{S}_{\mathbb{C}})$ in $(G_P)_{\mathbb{C}}$, then this Levi decomposition induces a homeomorphism
\[ X^+ \simeq P(\mathbb{R})^+U_P(\mathbb{C})/Z(x) \to \left(G_P(\mathbb{R})^+/Z(x_G)\right) \times W_P(\mathbb{R})U_P(\mathbb{C}) \simeq X_G^+ \times W_P(\mathbb{R})U_P(\mathbb{C}).\]
The homeomorphism is $P(\mathbb{R})^+U_P(\mathbb{C})$-equivariant, where $P(\mathbb{R})^+U_P(\mathbb{C})$ operates on $W_P(\mathbb{R})U_P(\mathbb{C})$ by multiplication from the left and its operation on $X_G^+$ is the canonical one.

	Using $y = \phi_{\mathbb{C}} \circ x$, we get in the same way a Levi decomposition of $Q$ and a homeomorphism $Y^+ \to Y_G^+ \times W_Q(\mathbb{R})U_Q(\mathbb{C})$. Furthermore, we have $\phi(G_P \ltimes \{0\}) \subset G_Q \ltimes\{0\}$ for these Levi decompositions and the map $X_G^+ \times W_P(\mathbb{R})U_P(\mathbb{C}) \to Y_G^+ \times W_Q(\mathbb{R})U_Q(\mathbb{C})$ induced by these homeomorphisms is the canonical one.
	
	After maybe replacing $\Gamma$ and $\Delta$ by smaller congruence subgroups, using again (a more general version of) Lemma 2.2 in \cite{OrrThesis} in the process, we can assume without loss of generality that $\Gamma = \Gamma_G \ltimes \Gamma_W$, where $\Gamma_G$ is a neat congruence subgroup of $G_P(\mathbb{Q})_+$ (as defined in \cite{PinkDiss}, 0.5) and $\Gamma_W \subset W_P(\mathbb{Q})$ is a $\Gamma_G$-invariant congruence subgroup, and analogously $\Delta = \Delta_G \ltimes \Delta_W$.
	
	By Theorem 2.4(2) in \cite{OrrThesis}, the map $\Gamma_G \backslash X_G^+ \to \Delta_G \backslash Y_G^+$ is proper. It follows that the same holds for the map
	\[(\Gamma_G \backslash X_G^+) \times_{\Delta_G \backslash Y_G^+} (\Delta\backslash Y^+) \to \Delta \backslash Y^+.\]
	
	Consider now the map
	\[ \Gamma \backslash X^+ \to (\Gamma_G \backslash X_G^+) \times_{\Delta_G \backslash Y_G^+} (\Delta\backslash Y^+).\]
	Thanks to the homeomorphisms constructed above as well as Proposition 3.3(b) in \cite{PinkDiss} with its proof together with the neatness of $\Gamma_G$ and $\Delta_G$, we can cover $\Gamma_G \backslash X_G^+$ with open sets $\mathcal{U}$ over which $\Gamma \backslash X^+$ is homeomorphic to $\mathcal{U} \times (\Gamma_W \backslash W_P(\mathbb{R})U_P(\mathbb{C}))$ while $(\Gamma_G \backslash X_G^+) \times_{\Delta_G \backslash Y_G^+} (\Delta\backslash Y^+)$ is homeomorphic to $\mathcal{U} \times (\Delta_W \backslash W_Q(\mathbb{R})U_Q(\mathbb{C}))$ (and the map between them is the canonical one). It then suffices to show that the map
	\[\Gamma_W \backslash W_P(\mathbb{R})U_P(\mathbb{C}) \to \Delta_W \backslash W_Q(\mathbb{R})U_Q(\mathbb{C})\]
	is proper.
	
	Set $V_P = W_P/U_P$ and $V_Q = W_Q/U_Q$. By \cite{PinkDiss}, 2.15, $U_P$ and $V_P$ are abelian and $W_P$ is isomorphic to $U_P \times_{\mathbb{Q}} V_P$ as a $\mathbb{Q}$-variety in such a way that the induced morphisms $U_P \to U_P \times_{\mathbb{Q}} V_P \to V_P$ are the canonical ones; analogous statements hold for $W_Q$. Arguing as above for $\phi^{-1}(U_Q) \subset (\ker \phi)U_P$ in case (1) and using Proposition 2.9 in \cite{G17}, we find that $\phi^{-1}(U_Q) \cap W_P = U_P$. The exact sequence $U_P \to W_P \to V_P$ induces an exact sequence of congruence subgroups $\Gamma_U \to \Gamma_W \to \Gamma_V$. Analogously, we get an exact sequence of congruence subgroups $\Delta_U \to \Delta_W \to \Delta_V$.
	
	As $\phi^{-1}(U_Q) \cap W_P = U_P$, the induced homomorphism $\phi_V: V_P \to V_Q$ is injective. Since any basis of $\Delta_V \cap \phi_V(V_P(\mathbb{Q}))$ can be completed to a basis of $\Delta_V$, the map $\phi_V^{-1}(\Delta_V)\backslash V_P(\mathbb{R}) \to \Delta_V \backslash V_Q(\mathbb{R})$ is a closed embedding. Furthermore, the map $\Gamma_V\backslash V_P(\mathbb{R}) \to \phi_V^{-1}(\Delta_V)\backslash V_P(\mathbb{R})$ is proper since $\Gamma_V$ has finite index in $\phi_V^{-1}(\Delta_V)$. Hence, the map $\Gamma_V\backslash V_P(\mathbb{R}) \to \Delta_V \backslash V_Q(\mathbb{R})$ is proper and therefore the map
	\[(\Gamma_V\backslash V_P(\mathbb{R})) \times_{\Delta_V \backslash V_Q(\mathbb{R})} (\Delta_W \backslash W_Q(\mathbb{R})U_Q(\mathbb{C})) \to \Delta_W \backslash W_Q(\mathbb{R})U_Q(\mathbb{C})\]
	is proper as well.
	
	Using that $W_P(\mathbb{R})U_P(\mathbb{C})$ and $W_Q(\mathbb{R})U_Q(\mathbb{C})$ are homeomorphic to $V_P(\mathbb{R}) \times U_P(\mathbb{C})$ and $V_Q(\mathbb{R}) \times U_Q(\mathbb{C})$ respectively and that the various quotient homomorphisms are covering maps in order to argue as above, we are reduced to showing that the map $\Gamma_U \backslash U_P(\mathbb{C}) \to \Delta_U \backslash U_Q(\mathbb{C})$ is proper. But this map is the composition of $\Gamma_U \backslash U_P(\mathbb{C}) \to \phi^{-1}(\Delta_U) \backslash U_P(\mathbb{C})$, which is proper since $\Gamma_U$ has finite index in $\phi^{-1}(\Delta_U)$, and the closed embedding $\phi^{-1}(\Delta_U) \backslash U_P(\mathbb{C}) \to \Delta_U \backslash U_Q(\mathbb{C})$ (recall that $\phi|_{W_P}$ is injective), hence proper.
\end{proof}

\begin{rmk}\label{rmk:imageofu}
Proposition 2.9 in \cite{G17} and the argument with Lie algebras in the proof of Proposition \ref{prop:pinkfacts}(1) show that $\phi(P) \cap U_Q = \phi(\phi^{-1}(U_Q)) = \phi(U_P)$ for any Shimura morphism $(P,X^+) \to (Q,Y^+)$ induced by a homomorphism of algebraic groups $\phi: P \to Q$, where $U_P \subset P$ and $U_Q \subset Q$ denote the respective subgroups from the definition of a connected mixed Shimura datum. Combined with Proposition 2.9 in \cite{PinkDiss} and its proof, this shows that being pure or of Kuga type is preserved under taking quotients. We will use these facts throughout the article.
\end{rmk}

\section{Intersections of special subvarieties of connected mixed Shimura varieties}\label{sec:intersectionspecial}

In this section and in the next one, we indicate by $\mathfrak{V}$, $\mathfrak{M}$, and $\mathcal{F}$ respectively the class of objects, the class of morphisms, and the functor corresponding to any of the three categories we are considering in Proposition \ref{prop:axiomsformsv}. 

In order to prove \ref{ax:2} we need to show that, in our context, irreducible components of intersections of special subvarieties are special. This is a well-known, widely used fact but we were unable to find a detailed proof of it in the literature and our definition of a connected mixed Shimura datum slightly differs from the ones in \cite{PinkDiss} and \cite{Pink}. For these reasons we include a proof here.

Before stating the theorem, we note that, for any $\phi\in \mathfrak{M}$, the image of $\mathcal{F}(\phi)$ is closed because of Remark 5.5 in \cite{G17} and Proposition 4.2 in \cite{Pink} (which also holds with our definition).
We can then call the image of a Shimura morphism a \emph{special subvariety} (see Definition \ref{defn:special} below and Definition 5.9(1) in \cite{G17}).

\begin{thm}\label{thm:intersectionspecial}
	Irreducible components of intersections of special subvarieties of connected (mixed) Shimura varieties (of Kuga type) are special.
\end{thm}

We recall that $\mathbb{S}$ denotes the Deligne torus.
In order to prove the above theorem we need several auxiliary facts.

\begin{lem}\label{lem:intersectionfundamentallemma}
	Let $(P,X^+)$ be a connected mixed Shimura datum and let $H \subset P$ be a connected algebraic subgroup. Let $U$ be the normal unipotent subgroup of $P$ from the definition of a connected mixed Shimura datum. Set $X_H = \{ x \in X^+;\mbox{ } x(\mathbb{S}_{\mathbb{C}}) \subset H_{\mathbb{C}}\}$. Then $X_H$ is a finite (possibly empty) union of $H(\mathbb{R})^+(U \cap H)(\mathbb{C})$-orbits.
\end{lem}

\begin{proof}
	We can assume without loss of generality that $X_H \neq \emptyset$.
	
	Let $u: H \to H/(U\cap H)$ denote the quotient map, where we regard $H/(U \cap H)$ as an algebraic subgroup of $P/U$. Then, $u_{\mathbb{C}} \circ x$ can be defined over $\mathbb{R}$ for each $x \in X_H$. Let $W$ denote the unipotent radical of $P$ and let $\pi: P \to P/W$ denote the quotient map. Set $(G,X_G^+) = (P,X^+)/W$ and $G_H = H/(W \cap H) \subset G$. The elements of $X_G^+$ are defined over $\mathbb{R}$, so we view $X_G^+$ as a subset of $\Hom(\mathbb{S},G_{\mathbb{R}})$.
	
	Let $X_G$ be the full $G(\mathbb{R})$-conjugacy class of some element of $X_G^+$.
	
	We first show \underline{Claim 1}: $X_{G_H} = \{ x \in X_G;\mbox{ }x(\mathbb{S}) \subset (G_H)_{\mathbb{R}}\}$ is a finite union of $G_H(\mathbb{R})$-conjugacy classes.
	
	For this, we reproduce an argument by Moonen in \cite{MoonenLetter}.
	
	We start the proof of Claim 1 by proving \underline{Claim 2}: $X_H \neq \emptyset$ implies that $G_H$ is reductive. It suffices to show that the image $H'$ of $G_H$ in $G^{\ad}$ is reductive since the center $Z(G)$ of $G$ is of multiplicative type by Corollary 17.62(a) in \cite{MilneAG}. Let $x \in X_H$. By \ref{ax:MSDd}, conjugation by $x(\sqrt{-1})$ induces a Cartan involution $\theta$ on $G^{\ad}_{\mathbb{R}}$, which yields an inner form $G^{\ad,\theta}_{\mathbb{R}}$ of $G^{\ad}_{\mathbb{R}}$ such that $G^{\ad,\theta}_{\mathbb{R}}(\mathbb{R})$ is compact. Since $x \in X_H$, this involution restricts to an involution of $H'_{\mathbb{R}}$, which we also denote by $\theta$. It yields an inner form $H'^{,\theta}_{\mathbb{R}} \subset G^{\ad,\theta}_{\mathbb{R}}$ of $H'_{\mathbb{R}}$. We deduce that $H'^{,\theta}_{\mathbb{R}}(\mathbb{R})$ is compact as well, so $H'_{\mathbb{R}}$ has a compact inner form. It follows from Proposition 14.32 in \cite{MilneAG} that $H'$ must be reductive and therefore $G_H$ is reductive. Thus, we have established Claim 2.
	
	We introduce the following notation: Given an algebraic group $Q$ over $\mathbb{R}$ and a maximal torus $T \subset Q$, we write $RW_{Q}(T) = N_Q(T)(\mathbb{R})/Z_Q(T)(\mathbb{R})$ for the associated ``real Weyl group", where $N_Q(T)$ and $Z_Q(T)$ denote the normalizer and the centralizer of $T$ in $Q$ respectively.
	
	By Corollary 3 on p. 320 of \cite{PlatonovRapinchuk}, there exists a finite set $\{S_1, \hdots, S_k\}$ of representatives for the conjugacy classes of maximal tori in $(G_H)_{\mathbb{R}}$ under conjugation by $G_H(\mathbb{R})$. Given a $G_H(\mathbb{R})$-conjugacy class $X_{\alpha}\subset X_{G_H}$, we can find $i \in \{1,\hdots,k\}$ such that some element of $X_{\alpha}$ factors through $S_i$. For each $G_H(\mathbb{R})$-conjugacy class $X_{\alpha} \subset X_{G_H}$, we choose an index $i = i(\alpha)$ for which this holds. We also choose an element $x \in X_{\alpha}$ such that $x(\mathbb{S}) \subset S_i$.
	
Set $S'_i = \ad_{\mathbb{R}}(S_i)$, where $\ad: G \to G^{\ad}$, and write $x' = \ad_{\mathbb{R}} \circ x$. By Proposition 17.20 in \cite{MilneAG}, $S'_i$ is a maximal torus of $H'_{\mathbb{R}}$. We claim that the class of $x'$ in $RW_{H'_\mathbb{R}}(S'_i)\backslash \Hom(\mathbb{S},S'_i)$ depends only on $X_{\alpha}$ and the choice of $i$, but not on $x$. To see this, suppose we have another element $y \in X_{\alpha}$ such that $y$ factors through $S_i$ too. Choose $h \in G_H(\mathbb{R})$ such that $y = hx$. Let $M \subset (G_H)_{\mathbb{R}}$ be the centralizer of $x(\mathbb{S})$, which is a reductive subgroup of $(G_H)_{\mathbb{R}}$ by Corollary 17.59 in \cite{MilneAG}. Let $M'$ be the image of $M$ in ${H'_\mathbb{R}} \subset G^{\ad}_{\mathbb{R}}$. Since $S'_i$ is a maximal torus of $H'_{\mathbb{R}}$, it is also a maximal torus of $M'$. Since $h^{-1}S_ih$ is contained in the centralizer of $h^{-1}y = x$, we have that $T := \ad_{\mathbb{R}}(h)^{-1}S'_i\ad_{\mathbb{R}}(h)$ is also a maximal torus of $M'$.
	
	On the other hand, $M'(\mathbb{R})$ is compact because the Cartan involution of $G^{\ad}_{\mathbb{R}}$ that is induced by $x(\sqrt{-1})$ restricts to the identity on $M'$. Therefore, $S'_i(\mathbb{R})$ and $T(\mathbb{R})$ are conjugate in $M'(\mathbb{R})$ by Theorem 15 on p.~250 of \cite{OV}. By the classification of real tori up to isogeny, $S'_i(\mathbb{R})$ is Zariski dense in $S'_i$. Thus there exists $m \in M'(\mathbb{R})$ such that $\ad_{\mathbb{R}}(h)^{-1}S'_i\ad_{\mathbb{R}}(h) = m^{-1}S'_im$. Then $z = \ad_{\mathbb{R}}(h)m^{-1}$ is an element of $N_{H'_\mathbb{R}}(S'_i)(\mathbb{R})$, and because $\ad_{\mathbb{R}} \circ y = \ad_{\mathbb{R}} \circ hx = zx'$, we see that $\ad_{\mathbb{R}} \circ y$ gives the same class in $RW_{H'_\mathbb{R}}(S'_i)\backslash \Hom(\mathbb{S},S'_i)$ as $x'$. So we get a well-defined class $\cl(X_{\alpha}) \in RW_{H'_\mathbb{R}}(S'_i) \backslash \Hom(\mathbb{S},S'_i)$.
	
	Choose now a maximal torus $T_i \subset G_{\mathbb{R}}$ containing $S_i$, and write $T^{\ad}_i \subset G^{\ad}_{\mathbb{R}}$ for its image under $\ad_{\mathbb{R}}$. Applying the arguments of the previous paragraphs with $G$ in place of $G_H$, $X_G$ in place of $X_{G_H}$, $T_i$ in place of $S_i$, and $T^{\ad}_i$ in place of $S'_i$, i.e., for $H = P$, we find that the $G(\mathbb{R})$-conjugacy class $X_G$ gives rise to a well-defined element of $RW_{G^{\ad}_\mathbb{R}}(T^{\ad}_i) \backslash \Hom(\mathbb{S},T^{\ad}_i)$. Since $RW_{G^{\ad}_\mathbb{R}}(T^{\ad}_i)$ is finite by Corollary 17.39(a) in \cite{MilneAG}, this implies that, given an index $i$, there are only finitely many elements in $RW_{H'_\mathbb{R}}(S'_i)\backslash \Hom(\mathbb{S},S'_i)$ that can occur as $\cl(X_{\alpha})$.
	
	Suppose we have $G_H(\mathbb{R})$-conjugacy classes $X_{\alpha}$ and $X_{\beta}$ in $X_{G_H}$ such that $i(\alpha) = i(\beta)$ and $\cl(X_{\alpha}) = \cl(X_{\beta})$. We want to show that $X_{G_H}$ is a finite union of $G_H(\mathbb{R})$-conjugacy classes. Choose $x \in X_{\alpha}$ and $y \in X_{\beta}$ such that $x$ and $y$ both factor through $S_i$. Since the group $RW_{H'_\mathbb{R}}(S'_i)$ is finite, it suffices to show that $\ad_{\mathbb{R}} \circ x = \ad_{\mathbb{R}} \circ y$ implies $X_\alpha = X_\beta$. If $\ad_{\mathbb{R}} \circ x = \ad_{\mathbb{R}} \circ y$, then $x$ and $y$ differ by a homomorphism $\chi: \mathbb{S} \to Z(G)_{\mathbb{R}} \cap S_i$, i.e., $y_{\mathbb{C}}(s) = x_{\mathbb{C}}(s) \cdot \chi_{\mathbb{C}}(s)$ for all $s \in \mathbb{S}(\mathbb{C})$. Now we look at the abelianization map $\ab: G \to G/G^{\der}$; then $\ab_{\mathbb{R}} \circ \xi$ is independent of the choice of $\xi \in X_G$. Hence we find that the image of $\chi$ is contained in $(Z(G) \cap G^{\der})_{\mathbb{R}}$. But this is a finite group by Proposition 19.21(b) in \cite{MilneAG} and $\mathbb{S}$ is connected, so it follows that $\chi$ is trivial and therefore $X_{\alpha} = X_{\beta}$.
	
	It follows that $X_{G_H}$ is a finite union of $G_H(\mathbb{R})$-conjugacy classes and so Claim 1 holds.
	
	We deduce that there exists a finite subset $C_H \subset X_H$ such that, for any $x \in X_H$, the composition $\pi_{\mathbb{C}} \circ x$ is conjugate to $\pi_{\mathbb{C}} \circ x_0$ through $G_H(\mathbb{R})$ for some $x_0 \in C_H$. After maybe enlarging $C_H$, we can even assume that $\pi_{\mathbb{C}} \circ x$ and $\pi_{\mathbb{C}} \circ x_0$ are conjugate through the identity component $G_H(\mathbb{R})^+$.
	
	We want to show that every $x \in X_H$ lies in the $H(\mathbb{R})^+(U \cap H)(\mathbb{C})$-orbit of some $x_0 \in C_H$. By the above together with Proposition 5.1 in \cite{MilneISV}, we can assume that $\pi_{\mathbb{C}} \circ x = \pi_{\mathbb{C}} \circ x_0$ for some $x_0 \in C_H$ after replacing $x$ by an $H(\mathbb{R})^+$-conjugate.
	
	Let $\tilde{\pi}: P/U \to G$ be the canonical homomorphism and set $T = (\pi_{\mathbb{C}} \circ x_0)(\mathbb{S}_{\mathbb{C}})$. Both $(\tilde{\pi}_{\mathbb{C}})|_{(u_{\mathbb{C}} \circ x)(\mathbb{S}_{\mathbb{C}})}$ and $(\tilde{\pi}_{\mathbb{C}})|_{(u_{\mathbb{C}}\circ x_0)(\mathbb{S}_{\mathbb{C}})}$ are isomorphisms onto $T$. They each give a Levi decomposition
	\[ \left(\tilde{\pi}_{\mathbb{C}}\right)|_{(H/(U \cap H))_{\mathbb{C}}}^{-1}(T) \simeq T \ltimes ((W \cap H)/(U \cap H))_{\mathbb{C}}.\]
	Since everything is defined over $\mathbb{R}$ here, both Levi decompositions are defined over $\mathbb{R}$. By Theorem 2.3 in \cite{PlatonovRapinchuk}, they are conjugate through an element of $((W \cap H)/(U \cap H))(\mathbb{R})$. By Proposition 5.1 in \cite{MilneISV} and Proposition 14.32 in \cite{MilneAG}, we can assume that $u_{\mathbb{C}} \circ x = u_{\mathbb{C}} \circ x_0$ after replacing $x$ by an $H(\mathbb{R})^+$-conjugate. Applying a similar argument to $u: H \to H/(U \cap H)$ over $\mathbb{C}$ finally shows that $x \in X_H$ lies in the $H(\mathbb{R})^+(U \cap H)(\mathbb{C})$-orbit of $x_0 \in C_H$ as desired.
\end{proof}

\begin{lem}\label{lem:intersectionfundamentallemma2}
	Let $(P,X^+)$ be a connected mixed Shimura datum and let $H \subset P$ be a connected algebraic subgroup. Let $U$ be the normal unipotent subgroup of $P$ from the definition of a connected mixed Shimura datum. Let $X_H^+ \subset X^+$ be an $H(\mathbb{R})^+(U \cap H)(\mathbb{C})$-orbit and suppose that $H$ is the Mumford-Tate group of an element $x \in X_H^+$, i.e., the smallest algebraic subgroup of $P$ such that $x(\mathbb{S}_{\mathbb{C}}) \subset H_{\mathbb{C}}$. Then, $(H,X_H^+)$ is a Shimura subdatum of $(P,X^+)$.
\end{lem}

\begin{proof}
	We have to verify that $(H,X_H^+)$ satisfies the axioms \ref{ax:MSDa} to \ref{ax:MSDe}.
	
	The normal unipotent subgroup of $H$ from the definition of a connected mixed Shimura datum will of course be $U \cap H$. \ref{ax:MSDa} is then immediate.
	
	The composition of $x$ with the adjoint representation of $H_{\mathbb{C}}$ induces a representation of $\mathbb{S}_{\mathbb{C}}$ on $\Lie H_{\mathbb{C}}$ and thus a decomposition of $\Lie H_{\mathbb{C}}$ into character eigenspaces. This representation is a subrepresentation of the representation of $\mathbb{S}_{\mathbb{C}}$ on $\Lie P_{\mathbb{C}}$ induced by the composition of $x$ with the adjoint representation of $P_{\mathbb{C}}$. Let $W$ denote the unipotent radical of $P$. Since the vector subspaces $\Lie(U \cap H)$ and $\Lie(W \cap H)$ of $\Lie H$ are defined over $\mathbb{Q}$ and the composition of $x$ with $H_{\mathbb{C}} \to (H/(U \cap H))_{\mathbb{C}}$ is defined over $\mathbb{R}$, it follows from \ref{ax:MSDb} and \ref{ax:MSDc} for $(P,X^+)$ that the decomposition of $\Lie H_{\mathbb{C}}$ yields a rational mixed Hodge structure on $\Lie H$ with $W_{-2}(\Lie H) = \Lie(U \cap H)$ and $W_{-1}(\Lie H) = \Lie(W \cap H)$. So \ref{ax:MSDb} follows.
	
	By the proof of Lemma \ref{lem:intersectionfundamentallemma}, $G_H = H/(W \cap H)$ is reductive, so $W \cap H$ is the unipotent radical of $H$ and \ref{ax:MSDc} follows as well. 
	
	We now prove the first part of \ref{ax:MSDd}. Set $G = P/W$. Let $w: H \to G_H \subset G$ and $\ad: G \to G^{\ad}$ denote the respective quotient homomorphisms, and let $Z(G)$ denote the center of $G$. Proposition 1.20 in \cite{MilneISV} shows that \ref{ax:MSDd} for $(P,X^+)$ implies that the adjoint representation $(\Lie G^{\ad})_{\mathbb{R}}$ of $G^{\ad}_{\mathbb{R}}$ carries an $(\ad_{\mathbb{C}} \circ w_{\mathbb{C}} \circ x)(\sqrt{-1})$-polarization as defined in \cite{MilneISV}. Let $\ad_H$ denote the subrepresentation $\ad(G_H) \to \GL(\Lie (\ad(G_H)))$ of the adjoint representation $G^{\ad} \to \GL(\Lie G^{\ad})$. It is equal to the adjoint representation of $\ad(G_H)$ and factors through $(G_H)^{\ad}$.
	
	Since $G_H$ is reductive, it follows from Theorem 22.42 in \cite{MilneAG} that also $\ad(G_H) = G_H/(Z(G) \cap G_H)$ is reductive. Let $Z(\ad(G_H))$ denote the center of $\ad(G_H)$. By Proposition 14.23 in \cite{MilneAG}, $(\ker \ad_H)/Z(\ad(G_H))$ is unipotent (so in particular connected). But $\ad(G_H)/Z(\ad(G_H))$ is semisimple, so $(\ker \ad_H)/Z(\ad(G_H))$ is reductive by Corollary 21.53 in \cite{MilneAG}. It follows that $\ker \ad_H = Z(\ad(G_H))$, so we get a faithful representation of $\ad(G_H)/Z(\ad(G_H))$. There are canonical homomorphisms $$\ad(G_H) \to (G_H)^{\ad} \to \ad(G_H)/Z(\ad(G_H)). $$ 
	By Corollary 17.62(e) in \cite{MilneAG}, $(G_H)^{\ad}$ has trivial center, so we get an isomorphism $(G_H)^{\ad} \simeq \ad(G_H)/Z(\ad(G_H))$ and hence a faithful representation of $((G_H)^{\ad})_{\mathbb{R}}$. Proposition 1.20 in \cite{MilneISV}, applied the other way around, now shows that $x(\sqrt{-1})$ induces a Cartan involution on $((G_H)^{\ad})_{\mathbb{R}}$. We have established the first part of \ref{ax:MSDd}.
	
	We now go on to prove the second part of \ref{ax:MSDd}. Suppose that the group $(G_H)^{\ad}$ possesses a non-trivial $\mathbb{Q}$-factor $H'$ such that $H'(\mathbb{R})$ is compact. Let $\rho$ denote the adjoint representation $(G_H)^{\ad}\to \GL(\Lie (G_H)^{\ad})$. Since $H'$ is normal in $(G_H)^{\ad}$, the subspace $\Lie H' \subset \Lie (G_H)^{\ad}$ is $(G_H)^{\ad}$-invariant. Therefore, $\rho$ induces a representation $\rho': (G_H)^{\ad} \to \GL(\Lie H')$. Composing $x$ with first the quotient map $H_{\mathbb{C}} \to ((G_H)^{\ad})_{\mathbb{C}}$ and then $\rho'_{\mathbb{C}}$ and noting that the resulting homomorphism is defined over $\mathbb{R}$ by \ref{ax:MSDa}, we get a representation $\rho_{\mathbb{S}}: \mathbb{S} \to \GL(\Lie H')_{\mathbb{R}}$. Now conjugation by $x(\sqrt{-1})$ induces a Cartan involution on $((G_H)^{\ad})_{\mathbb{R}}$ and therefore on $H'_{\mathbb{R}}$. Since the identity is a Cartan involution on $H'_{\mathbb{R}}$, Corollary 4.3 in Chapter I of \cite{SatakeBook} implies that the image of $x(\sqrt{-1})$ in $(G_H)^{\ad}(\mathbb{R})$ centralizes $H'_{\mathbb{R}}$. We deduce that $\sqrt{-1} \in (\ker \rho_{\mathbb{S}})(\mathbb{R})$. As the pure Hodge structure on $\Lie (G_H)^{\ad}$ induced by $x$ is of type $\{(0,0),(1,-1),(-1,1)\}$, it follows from $\sqrt{-1} \in (\ker \rho_{\mathbb{S}})(\mathbb{R})$ that $\ker \rho_{\mathbb{S}} = \mathbb{S}$. Since $H$ is the Mumford-Tate group of $x$, this implies that the representation $\rho': (G_H)^{\ad} \to \GL(\Lie H')$ must be trivial. But the adjoint representation $H' \to \GL(\Lie H')$ of the semisimple algebraic group $H' \subset (G_H)^{\ad}$ factors through $\rho'$ and arguing as above shows that its kernel is the center of $H'$, which is trivial by Corollary 17.62(e) and Theorem 21.51 in \cite{MilneAG} since $H'$ is a $\mathbb{Q}$-factor of an adjoint reductive group. So we contradict the non-triviality of $H'$ and therefore there is no such factor $H'$ and the second part of \ref{ax:MSDd} follows.
	
	We are left with showing that \ref{ax:MSDe} holds. Set $W_H = W \cap H$, the unipotent radical of $H$. The algebraic group $H$ acts trivially on $W_H/(W_H \cap H^{\der})$ by conjugation, but $$\Lie W_H/(W_H \cap H^{\der}) \simeq \Lie W_H/\Lie(W_H \cap H^{\der})$$ and so it follows from \ref{ax:MSDc}, already proved, that $W_H/(W_H \cap H^{\der}) = \{0\}$, i.e., $W_H \subset H^{\der}$. Therefore, $H/H^{\der} \simeq G_H/(G_H)^{\der}$, which is isogenous to the identity component $Z(G_H)^0$ of the center of $G_H$.
	
	It therefore suffices to show that $Z(G_H)^0$ is a torus as in \ref{ax:MSDe}. Recall that $w: H \to G_H$ denotes the quotient map. By Lemma 13.3, Proposition 5.1, and Theorem 5.4 in \cite{MilneISV}, there exists $y \in X_H^+$ and a maximal torus $T_H$ in $G_H$ such that $z = w_{\mathbb{C}} \circ y$ factors through $(T_H)_{\mathbb{C}}$. We have $Z(G_H)^0 \subset T_H$ by the maximality of $T_H$. We can find a maximal torus $T$ in $G$ that contains $T_H$ and $Z(G)^0$. Since the category of $\mathbb{Q}$-tori up to isogeny is semi-simple (see Proposition 4.1 in \cite{Brion}), it suffices to show that $T$ is a torus as in \ref{ax:MSDe}.
	
	Recall that $G = G^{\der} Z(G)^{0}$. Therefore, $T = T^{\der}  Z(G)^0$, where $T^{\der}$ by abuse of notation denotes the identity component $(T \cap G^{\der})^0$ and is a maximal torus in $G^{\der}$. Since $Z(G)^0$ is isogenous to $G/G^{\der}$, \ref{ax:MSDe} for $(P,X^+)$ shows that it suffices to prove that $T^{\der}$ is a torus as in \ref{ax:MSDe}. Now $T^{\der}(\mathbb{R})$ fixes $z$ and so $T^{\der}(\mathbb{R})$ is contained in the stabilizer $K_z$ of $z$ in $G^{\der}(\mathbb{R})$.
	
	The stabilizer $K_z$ is a closed subset of $K_1(\mathbb{R})$, where $K_1$ denotes the algebraic subgroup of $G^{\der}_{\mathbb{R}}$ of elements commuting with $z(\sqrt{-1})$. Let $K_2$ denote the algebraic subgroup of $G^{\ad}_{\mathbb{R}}$ of elements commuting with the image of $z(\sqrt{-1})$ under the homomorphism $G_{\mathbb{R}} \to G^{\ad}_{\mathbb{R}}$. By \ref{ax:MSDd}, conjugation by $z(\sqrt{-1})$ induces a Cartan involution on $G^{\ad}_{\mathbb{R}}$, which implies that $K_2(\mathbb{R})$ is compact. Clearly $K_1$ is mapped into $K_2$ by the homomorphism $G^{\der}_{\mathbb{R}} \to G^{\ad}_{\mathbb{R}}$. As the topological covering $G^{\der}(\mathbb{C}) \to G^{\ad}(\mathbb{C})$ has finite fibers, this implies that $K_1(\mathbb{R})$ is compact. Hence, $K_z$ is a compact closed subgroup of $G^{\der}(\mathbb{R})$ and so $T^{\der}(\mathbb{R})$ is compact. This completes the proof.
\end{proof}

\begin{rmk}\label{rmk:shimurasubdatum}
By adapting the end of the proof of Lemma \ref{lem:intersectionfundamentallemma2}, we can also prove the following: Let $(P,X^+)$ be a connected mixed Shimura datum and let $H \subset P$ be a connected algebraic subgroup. Let $U$ be the normal unipotent subgroup of $P$ from the definition of a connected mixed Shimura datum and let $X_H^+ \subset X^+ \cap \mathrm{Hom}(\mathbb{S}_{\mathbb{C}},H_{\mathbb{C}})$ be an $H(\mathbb{R})^+(U \cap H)(\mathbb{C})$-orbit. If $W = \{1\}$ and $(H,X_H^+)$ satisfies \ref{ax:MSDa} to \ref{ax:MSDd}, then it also satisfies \ref{ax:MSDe}. This can also be proved if $W \neq \{1\}$, see Example 2.10 in \cite{PinkDiss}.
\end{rmk}

\begin{lem}\label{lem:intersectionfundamentallemma3}
	Let $(P,X^+)$ be a connected mixed Shimura datum and let $H \subset P$ be a connected algebraic subgroup. Let $U$ be the normal unipotent subgroup of $P$ from the definition of a connected mixed Shimura datum. Let $X_H^+ \subset X^+$ be an $H(\mathbb{R})^+(U \cap H)(\mathbb{C})$-orbit such that one (equivalently: any) $x \in X_H^+$ factors through $H_{\mathbb{C}}$. Then there exist finitely many $x_i \in X_H^+$ ($i = 1,\hdots,n$) such that $X_H^+$ is the disjoint union of the $M_{x_i}(\mathbb{R})^+(U\cap M_{x_i})(\mathbb{C})x_i$, where $M_{x_i} \subset H$ denotes the Mumford-Tate group of $x_i$ as defined in Lemma \ref{lem:intersectionfundamentallemma2}.
\end{lem}

\begin{proof}
	We first show that $H(\mathbb{R})^+(U \cap H)(\mathbb{C})$ is a connected real analytic manifold: Let $W$ denote the unipotent radical of $P$. It follows from the proof of Lemma \ref{lem:intersectionfundamentallemma} that $W_H := W \cap H$ is the unipotent radical of $H$. Set $G_H = H/W_H$ and $U_H = U \cap H$, and fix a Levi decomposition $H = G_H \ltimes W_H$. We get a homeomorphism
	\[ H(\mathbb{R})^+U_H(\mathbb{C}) \to G_H(\mathbb{R})^+ \times W_H(\mathbb{R})U_H(\mathbb{C}).\]
	Set $V_H = W_H/U_H$ and decompose $\Lie W_H = V \oplus \Lie U_H$ (as a vector space) such that the restriction of the differential $V \to \Lie V_H$ is an isomorphism. Since the exponential map of a unipotent group is an isomorphism of schemes by Proposition 14.32 in \cite{MilneAG}, we get an isomorphism of schemes
	\[ W_H \to \Lie W_H \to \Lie V_H \times_{\mathbb{Q}} \Lie U_H \to V_H \times_{\mathbb{Q}} U_H.\]
	By the commutative diagram before Proposition 14.32 on p.~289 in \cite{MilneAG}, this yields a homeomorphism $W_H(\mathbb{R})U_H(\mathbb{C}) \to V_H(\mathbb{R}) \times U_H(\mathbb{C})$. This gives a structure of a connected real analytic manifold on $H(\mathbb{R})^+U_H(\mathbb{C})$.
	
	Fix now some $x_0 \in X_H^+$. For each $x \in X_H^+$, the set of $h\in H(\mathbb{R})^+U_H(\mathbb{C})$ such that $hx_0$ factors through the Mumford-Tate group of $x$ is a closed real analytic subset of $H(\mathbb{R})^+U_H(\mathbb{C})$. Note that it suffices to check whether $(hx_0)(2)$ and $(hx_0)(\alpha)$ lie in the Mumford-Tate group of $x$ for some fixed $\alpha\in S^1 \subset \mathbb{C}^{\ast} = \mathbb{S}(\mathbb{R})$ of infinite order since $2$ and $\alpha$ then generate a Zariski dense subgroup of $\mathbb{S}$.
	
	Now, a connected real analytic manifold cannot be covered by countably many proper closed real analytic subsets by the Baire category theorem and the fact that any proper closed real analytic subset has empty interior by the identity theorem for real analytic functions. Since the set of possible Mumford-Tate groups of elements of $X^+_H$ is countable, there must be some $x_1 \in X_H^+$ such that all $x \in X_H^+$ factor through the Mumford-Tate group $M_{x_1} \subset H$ of $x_1$. It then follows from Lemma \ref{lem:intersectionfundamentallemma} that $X_H^+$ is a finite union of $M_{x_1}(\mathbb{R})^+(U \cap M_{x_1})(\mathbb{C})$-orbits, among them $M_{x_1}(\mathbb{R})^+(U \cap M_{x_1})(\mathbb{C})x_1$.
	
	If any such orbit, call it $\Omega$, does not contain any element whose Mumford-Tate group is $M_{x_1}$, then we iterate the previous step to find that it contains an element $x_2$ such that all elements of $\Omega$ factor through the Mumford-Tate group $M_{x_2}$ of $x_2$. Furthermore, we have $M_{x_2}\subsetneq M_{x_1}$ and so $\dim M_{x_1} > \dim M_{x_2}$ since Mumford-Tate groups are connected. The orbit $\Omega$ is a finite union of $M_{x_2}(\mathbb{R})^+(U \cap M_{x_2})(\mathbb{C})$-orbits by Lemma \ref{lem:intersectionfundamentallemma}. Continuing in this way, we eventually prove the lemma as the dimension of the Mumford-Tate group cannot drop below $0$.
\end{proof}

We are now able to prove Theorem \ref{thm:intersectionspecial}.

\begin{proof}[Proof of Theorem \ref{thm:intersectionspecial}]
	Let $S_1$ and $S_2$ be two special subvarieties of a connected (mixed) Shimura variety (of Kuga type) $S = \mathcal{F}(P,X^+,\Gamma)$. 
	
	Let $i \in \{1,2\}$. By definition, we have $S_i = \mathcal{F}(\phi_i)(\mathcal{F}(H_i,Y_i^+,\Gamma_i))$ for a Shimura morphism $\phi_i: (H_i,Y_i^+,\Gamma_i) \to (P,X^+,\Gamma)$. Let $K_i$ denote the kernel of the homomorphism $H_i \to P$. By Proposition 2.9 in \cite{PinkDiss}, there is a quotient Shimura datum $(H_i,Y_i^{+})/K_i$ and the Shimura morphism $(H_i,Y_i^{+}) \to (P,X^{+})$ factors as $(H_i,Y_i^{+}) \to (H_i,Y_i^{+})/K_i \to (P,X^{+})$. Thanks to Lemma 2.2 in \cite{OrrThesis} (which is formulated for reductive groups, but holds with the same proof for arbitrary linear algebraic groups), we can find a congruence subgroup $\Gamma'_i$ of $(H_i/K_i)(\mathbb{Q})_{+}$ that gives a factorization of $\phi_i$ as
	\[ (H_i,Y_i^{+},\Gamma_i) \to ((H_i,Y_i^{+})/K_i,\Gamma'_i) \to (P,X^{+}, \Gamma).\]
	Furthermore, being pure or of Kuga type is preserved under taking quotients by Remark \ref{rmk:imageofu}. This shows that we can assume without loss of generality that $(H_i,Y_i^+)$ is a Shimura subdatum of $(P,X^+)$.
	
	There is a canonical identification $S(\mathbb{C}) = \Gamma \backslash X^+$. We denote the map $X^+ \to \Gamma \backslash X^+$ by ``$\unif$''. Since $(H_i,Y_i^+)$ is a Shimura subdatum of $(P,X^+)$, we can identify $Y_i^+$ with its image in $X^+$ and it follows that $S_i(\mathbb{C}) = \unif(Y_i^+)$ after this identification ($i = 1,2$). We have
	\[(S_1 \cap S_2)(\mathbb{C}) = \unif(\unif^{-1}(S_1(\mathbb{C})) \cap \unif^{-1}(S_2(\mathbb{C}))) = \bigcup_{\gamma \in \Gamma}{\unif(Y_1^+ \cap \gamma Y_2^+)}.\]
	Note that $(\gamma H_i \gamma^{-1},\gamma Y_i^+)$ is a Shimura subdatum of $(P,X^+)$ for any $\gamma \in \Gamma$ and any $i \in \{1,2\}$.
	
	Let $U \subset P$ be the normal unipotent subgroup from the definition of a connected mixed Shimura datum. For $\gamma \in \Gamma$, let $H_{\gamma}$ denote the identity component of $H_1 \cap \gamma H_2 \gamma^{-1}$. We have $Y_1^+ \cap \gamma Y_2^+ \subset \{ x \in X^+; x(\mathbb{S}_{\mathbb{C}}) \subset (H_{\gamma})_{\mathbb{C}}\}$. It follows from Remark \ref{rmk:imageofu} applied to $(\gamma H_i \gamma^{-1},\gamma Y_i^+) \to (P,X^+)$ that $U \cap \gamma H_i \gamma^{-1}$ is the normal unipotent subgroup from the definition of a connected mixed Shimura datum for $(\gamma H_i \gamma^{-1},\gamma Y_i^+)$ for $i = 1,2$, $\gamma \in \Gamma$. Therefore, $Y_1^+ \cap \gamma Y_2^+$ is a union of $H_{\gamma}(\mathbb{R})^+(U \cap H_{\gamma})(\mathbb{C})$-orbits. By Lemma \ref{lem:intersectionfundamentallemma}, it is a finite such union.
	
	Combining Lemmas \ref{lem:intersectionfundamentallemma2} and \ref{lem:intersectionfundamentallemma3} shows that each such orbit is a finite union of sets $Y'^+$ such that $(H', Y'^+)$ is a Shimura subdatum of $(P,X^+)$ for a connected algebraic subgroup $H'$ of $P$. If $(P,X^+)$ is pure or of Kuga type, then $(H',Y'^+)$ is also pure or of Kuga type respectively by Proposition 2.9 in \cite{G17}.
		
	This implies that $(S_1 \cap S_2)(\mathbb{C})$ is a countable union of sets $S'(\mathbb{C})$ for special subvarieties $S' \subset S$. Since the $\mathbb{C}$-points of an algebraic variety over $\mathbb{C}$ cannot be covered by the $\mathbb{C}$-points of countably many proper subvarieties, each component of $S_1 \cap S_2$ must be special.
\end{proof}

\section{Proof of Proposition \ref{prop:axiomsformsv}}\label{sec:axiomsformsv}

As one might expect, axioms \ref{ax:1} and \ref{ax:3} are the easiest to verify. 

Axiom \ref{ax:1} holds with $(P,X^+,\Gamma) \times (P',X'^+,\Gamma') = (P\times_{\mathbb{Q}}P',X^+ \times X'^+,\Gamma \times \Gamma')$, where $X^+\times X'^+$ consists of the homomorphisms $(\sigma,\sigma'): \mathbb{S}_{\mathbb{C}} \to P_{\mathbb{C}} \times_{\mathbb{C}} P'_{\mathbb{C}}$ for $\sigma \in X^+$ and $\sigma' \in X'^+$ (cf. \cite{PinkDiss}, Definition 2.5).

Axiom \ref{ax:3} is satisfied with $P$ the trivial group, $X^+ = \Hom(\mathbb{S}_{\mathbb{C}},P_{\mathbb{C}})$, and $\Gamma = P(\mathbb{Q})$.

Let us now turn to \ref{ax:2}. We recall that a special subvariety is the image of a Shimura morphism.

Let $\phi: (P,X^+,\Gamma) \to (Q,Y^+,\Delta)$ and $\psi:  (P',X'^+,\Gamma')\to (Q,Y^+,\Delta)$ be Shimura morphisms and set $V=\mathcal{F}(P,X^+,\Gamma) $, $V'=\mathcal{F}(P',X'^+,\Gamma') $, and $W=\mathcal{F}(Q,Y^+,\Delta) $. We want to prove that $(V \times_{W} V')_{\red} \subset V \times_{\bar{\mathbb{Q}}} V'$ is a finite union of special subvarieties.

We let $U=V \times_{\bar{\mathbb{Q}}} V'\times_{\bar{\mathbb{Q}}} W \times_{\bar{\mathbb{Q}}} W$ and let $G \subset U$ be the graph of the Shimura morphism $\mathcal{F}(\phi)\times \mathcal{F}(\psi):  V \times_{\bar{\mathbb{Q}}} V'\to W \times_{\bar{\mathbb{Q}}} W$. We note that $U $ is a connected (mixed) Shimura variety (of Kuga type) and $G$ is a special subvariety of $U$ because it is the image of the Shimura morphism $(\id_{V \times_{\bar{\mathbb{Q}}} V'}, \mathcal{F}(\phi)\times \mathcal{F}(\psi)):V \times_{\bar{\mathbb{Q}}} V'\to U$. 

Let $D$ be the diagonal in $W \times_{\bar{\mathbb{Q}}} W$. It is special since it is the image of the Shimura morphism $(\id_W,\id_W):W \to W \times_{\bar{\mathbb{Q}}} W$. Hence, $D':=V \times_{\bar{\mathbb{Q}}} V'\times_{\bar{\mathbb{Q}}} D$ is a special subvariety of $U$.

We can apply Theorem \ref{thm:intersectionspecial} to $G$ and $D'$ and obtain that the irreducible components of $G\cap D'$ are special subvarieties of $U$.

Finally, if $\pi : U\to V \times_{\bar{\mathbb{Q}}} V'$ is the projection, which is a Shimura morphism, we have that $(V \times_{W} V')_{\red} $ is the image of $G\cap D'$ via $\pi$ and therefore a finite union of special subvarieties. We are done with \ref{ax:2}.

To prove \ref{ax:4} we need several further auxiliary lemmas.
 We start by showing that, under certain conditions, a surjective homomorphism of algebraic groups gives a Shimura covering.

\begin{lem}\label{lem:centext}
Suppose that $(Q,Y^{+})$ is a connected mixed Shimura datum and $\tilde{Q}$ is an algebraic group over $\mathbb{Q}$. Suppose that there is a surjective homomorphism $\phi: \tilde{Q} \to Q$ whose kernel is a central torus in $\tilde{Q}$ that is an almost-direct product of a $\mathbb{Q}$-split torus and a torus of compact type that is defined over $\mathbb{Q}$.

If some $y \in Y^+$ factors through $\phi_{\mathbb{C}}$, then there exist a connected mixed Shimura datum $(\tilde{Q},\tilde{Y}^{+})$ and a Shimura covering $(\tilde{Q},\tilde{Y}^{+}) \to (Q,Y^{+})$ associated to $\phi$. Moreover, if $(Q,Y^{+})$ is a connected pure Shimura datum or a connected mixed Shimura datum of Kuga type, then the same holds for $(\tilde{Q},\tilde{Y}^{+})$.
\end{lem}

\begin{proof}
First of all, $\tilde{Q}$ is a connected linear algebraic group by Proposition 8.1 in \cite{MilneAG}.

Let $y \in Y^+\subset \Hom(\mathbb{S}_{\mathbb{C}},Q_{\mathbb{C}})$ such that there exists $\tilde{y} \in \Hom(\mathbb{S}_{\mathbb{C}},\tilde{Q}_{\mathbb{C}})$ with $y = \phi_{\mathbb{C}} \circ \tilde{y}$.

Let $W$ and $\tilde{W}$ denote the unipotent radicals of $Q$ and $\tilde{Q}$ respectively. Let $U \subset W$ be the algebraic subgroup, normal in $Q$, from the definition of a connected mixed Shimura datum and set
\[\tilde{U} = \left(\phi|_{\tilde{W}}\right)^{-1}(U) \subset\tilde{W}.\]

This will be the algebraic subgroup of $\tilde{W}$ from the definition of a connected mixed Shimura datum. We start by verifying that $\phi|_{\tilde{W}}$ is an isomorphism onto $W$ and that $\tilde{U}$ is normal in $\tilde{Q}$.

Since $\phi$ is surjective, we have that $\phi(\tilde{W})$ is normal in $Q$. By Corollary 14.7 in \cite{MilneAG}, $\phi(\tilde{W})$ is also unipotent and hence $\phi(\tilde{W}) \subset W$. By Corollary 14.17 in \cite{MilneAG} we have that $\tilde{W} \cap \ker\phi$ is trivial and so $\phi|_{\tilde{W}}: \tilde{W}\to Q$ is injective.

 We have exact sequences
\[ 1 \to \ker\phi \to \phi^{-1}(W) \to W \to 1\]
and
\[1 \to \ker\phi \to \phi^{-1}(U) \to U \to 1.\]
The algebraic group $\phi^{-1}(W)$ is connected by Proposition 8.1 in \cite{MilneAG} and it admits a central subnormal series $\phi^{-1}(W) \supset\phi^{-1}(U) \supset \ker\phi \supset \{1\}$ because $U$ is central in $W$ and $W/U$ is commutative by \cite{PinkDiss}, 2.15. Hence $\phi^{-1}(W)$ is nilpotent as defined in Definition 6.34 in \cite{MilneAG}. By Corollary 16.48 in \cite{MilneAG}, $\phi^{-1}(W)$ is isomorphic to the product of a unipotent algebraic group and a torus. The first exact sequence above
 shows that $\phi^{-1}(W) \simeq (\ker \phi) \times_{\mathbb{Q}} W$. Therefore, $\tilde{Q}$ contains a normal unipotent subgroup that is isomorphic to $W$. Hence $\dim \tilde{W} \geq \dim W$ and the injective homomorphism $\phi|_{\tilde{W}}: \tilde{W}\to W$ must be an isomorphism and the same for $\phi|_{\tilde{U}}: \tilde{U} \to U$.

Any conjugate of $\tilde{U}$ in $\tilde{Q}$ is contained in $\phi^{-1}(U)$ and unipotent. It follows as above from the second exact sequence
that $\phi^{-1}(U) \simeq (\ker \phi) \times_{\mathbb{Q}} U \simeq (\ker \phi) \times_{\mathbb{Q}} \tilde{U}$ and hence $\tilde{U}$ must be normal in $\tilde{Q}$.

We might be tempted to take $\tilde{Y}^+$ to be the connected component of the orbit of $\tilde{y}$ under conjugation by $\tilde{Q}(\mathbb{R})\tilde{U}(\mathbb{C})$ that contains $\tilde{y}$, but this might not satisfy \ref{ax:MSDa}. To remedy this, we will multiply $\tilde{y} \in \Hom(\mathbb{S}_{\mathbb{C}},\tilde{Q}_{\mathbb{C}})$ by a suitable homomorphism $\mathbb{S}_{\mathbb{C}} \to (\ker\phi)_{\mathbb{C}}$ that we construct in the following. Let $\chi$ denote the composition of $y\in \Hom(\mathbb{S}_{\mathbb{C}},Q_{\mathbb{C}})$ with $Q_{\mathbb{C}} \to (Q/U)_{\mathbb{C}}$. Let $T$ denote the real torus in $(Q/U)_{\mathbb{R}}$ such that $T_{\mathbb{C}}$ is equal to $\chi(\mathbb{S}_{\mathbb{C}})$. From now on, we view $\chi$ as a homomorphism $\mathbb{S}\to T$. Let $\psi: \tilde{Q}/\tilde{U} \to Q/U$ be the homomorphism induced by $\phi$ and set $\tilde{T} = \psi_{\mathbb{R}}^{-1}(T)$. We have an exact sequence
\[ 1\to \ker\phi_{\mathbb{R}} \simeq \ker \psi_{\mathbb{R}} \to \tilde{T} \to T \to 1,\]
where the isomorphism between the two kernels is induced by the restriction of $\tilde{Q} \to \tilde{Q}/\tilde{U}$ to $\ker \phi$. By Proposition 8.1 and Theorem 15.39 in \cite{MilneAG}, $\tilde{T}$ is a torus as well.

In the next paragraphs, we will construct $\tilde{\chi}: \mathbb{S} \to \tilde{T}$ such that $\chi$ is equal to $\tilde{\chi}$ composed with $\tilde{T}\to T$. Let $\chi'$ denote the composition of $\tilde{y}$ with $\tilde{Q}_{\mathbb{C}} \to (\tilde{Q}/\tilde{U})_{\mathbb{C}}$; its image is contained in $\tilde{T}_{\mathbb{C}}$. Then $\tilde{\chi}_{\mathbb{C}}\chi'^{-1}$ canonically yields a homomorphism $\mathbb{S}_{\mathbb{C}} \to \ker \phi_{\mathbb{C}}$. After replacing $\tilde{y}$ by its product with that homomorphism, we can define $\tilde{Y}^+$ to be the connected component of the orbit of $\tilde{y}$ under conjugation by $\tilde{Q}(\mathbb{R})\tilde{U}(\mathbb{C})$ that contains $\tilde{y}$. Condition \ref{ax:MSDa} is then automatically satisfied.

It remains to find a lift $\tilde{\chi}$ of $\chi$ to $\tilde{T}$. Passing to the character modules and using Theorem 12.23 in \cite{MilneAG}, we get an exact sequence
\[ 0 \to X^{\ast}(T) \to X^{\ast}(\tilde{T}) \to X^{\ast}(\ker \psi_{\mathbb{R}}) \to 0\]
of $\Gal(\mathbb{C}/\mathbb{R})$-modules and a $\Gal(\mathbb{C}/\mathbb{R})$-equivariant homomorphism $\chi^{\ast} = (\chi^{\ast}_1,\chi^{\ast}_2): X^{\ast}(T) \to X^{\ast}(\mathbb{S}) \simeq \mathbb{Z}^2$ that we want to extend to a $\Gal(\mathbb{C}/\mathbb{R})$-equivariant homomorphism $X^{\ast}(\tilde{T}) \to \mathbb{Z}^2$ (where $\Gal(\mathbb{C}/\mathbb{R})$ acts on $\mathbb{Z}^2$ by swapping the factors). We can find a free subgroup $F \subset X^{\ast}(\tilde{T})$ (isomorphic to $X^{\ast}(\ker \psi_{\mathbb{R}})$ as an abelian group) such that $X^{\ast}(\tilde{T}) \simeq X^{\ast}(T)\oplus F$ as abelian groups. Let $p_1: X^{\ast}(\tilde{T}) \to X^{\ast}(T)$ and $p_2: X^{\ast}(\tilde{T}) \to F$ denote the two projections; by abuse of notation, we will also view them as homomorphisms from $X^{\ast}(\tilde{T})$ to itself. If $\tau$ denotes the non-trivial element of $\Gal(\mathbb{C}/\mathbb{R})$, which we also view as an automorphism of $X^{\ast}(\tilde{T})$, then we will show that an extension of $\chi^{\ast}$ is given by
\begin{equation}\label{eq:lift}
v\in X^{\ast}(\tilde{T}) \mapsto ((\chi^{\ast}_1 \circ p_1)(v),(\chi^{\ast}_2 \circ p_1)(v)+(\chi^{\ast}_1 \circ p_1 \circ \tau \circ p_2)(v)) \in \mathbb{Z}^2.
\end{equation}

Clearly, this is a homomorphism of abelian groups that extends $\chi^{\ast}$. We still have to verify that it is $\Gal(\mathbb{C}/\mathbb{R})$-equivariant.

Let $v\in X^{\ast}(\tilde{T})$. Since $\tau(X^{\ast}(T)) = X^{\ast}(T)$, $\chi^{\ast}$ is $\Gal(\mathbb{C}/\mathbb{R})$-equivariant, and $v = p_1(v)+p_2(v)$, we have
\[ (\chi^{\ast}_1\circ p_1)(\tau(v)) = \chi^{\ast}_1(\tau(p_1(v))+p_1(\tau(p_2(v)))) = (\chi^{\ast}_2 \circ p_1)(v) + (\chi^{\ast}_1 \circ p_1 \circ \tau \circ p_2)(v). \]
Furthermore, we have that
\begin{align*}
(\chi^{\ast}_2 \circ p_1)(\tau(v))+(\chi^{\ast}_1 \circ p_1 \circ \tau \circ p_2)(\tau(v)) =\\
(\chi^{\ast}_2 \circ p_1)(\tau(p_1(v)))+(\chi^{\ast}_2 \circ p_1)(\tau(p_2(v)))+(\chi^{\ast}_1 \circ p_1 \circ \tau \circ p_2)(\tau(v)) =\\
(\chi^{\ast}_1 \circ p_1)(v)+(\chi^{\ast}_2 \circ p_1)(\tau(p_2(v)))+(\chi^{\ast}_1 \circ p_1 \circ \tau \circ p_2)(\tau(v)) = \\
(\chi^{\ast}_1 \circ p_1)(v)+\chi^{\ast}_1((\tau \circ p_1\circ \tau \circ p_2)(v)+(p_1 \circ \tau \circ p_2 \circ \tau \circ p_2)(v)).
\end{align*}
We can now use that
\begin{align*}
\tau \circ p_1\circ \tau \circ p_2 + p_1 \circ \tau \circ p_2 \circ \tau \circ p_2 = \\
p_1 \circ \tau \circ p_1\circ \tau \circ p_2 + p_1 \circ \tau \circ p_2 \circ \tau \circ p_2 = \\
p_1 \circ \tau \circ \tau \circ p_2 = p_1 \circ p_2 = 0.
\end{align*}
So \eqref{eq:lift} is really the desired extension and we proved that condition \ref{ax:MSDa} is satisfied.

The composition of $\tilde{y}$ with the adjoint representation yields a decomposition $\Lie \tilde{Q}_{\mathbb{C}} = \bigoplus_{(m,n) \in \mathbb{Z}^2}{L^{m,n}}$ into character eigenspaces for the induced representation of $\mathbb{S}_{\mathbb{C}}$. As $\tilde{Q}$ is a central extension of $Q$ and we have an $\mathbb{S}_{\mathbb{C}}$-equivariant isomorphism $\Lie Q \simeq(\Lie \tilde{Q})/(\Lie \ker \phi)$, the numbers $\dim L^{m,n}$ are equal to the Hodge numbers $h^{m,n}$ (as defined in \cite{PinkDiss}, 1.1) of the rational mixed Hodge structure induced by $y$ on $\Lie Q$ apart from $\dim L^{0,0}$, which equals $h^{0,0} + \dim \ker \phi$.

Recall that $\tilde{U}$ is normal in $\tilde{Q}$ and $\phi|_{\tilde{U}}$ is an isomorphism onto $U$. This yields an $\mathbb{S}_{\mathbb{C}}$-equivariant isomorphism $\Lie \tilde{U}_{\mathbb{C}} \simeq \Lie U_{\mathbb{C}}$. Since $\dim L^{-1,-1} = h^{-1,-1}$ and $\Lie U = W_{-2}(\Lie Q)$, we must have $\Lie \tilde{U}_{\mathbb{C}} = L^{-1,-1}$. An analogous argument shows that $\Lie \tilde{W}_{\mathbb{C}} = L^{-1,-1} \oplus L^{0,-1} \oplus L^{-1,0}$. Finally, since the composition of $\tilde{y}$ with $\tilde{Q}_{\mathbb{C}} \to (\tilde{Q}/\tilde{U})_{\mathbb{C}} $ is defined over $\mathbb{R}$, we have that complex conjugation acts on $\Lie (\tilde{Q}/\tilde{U})_{\mathbb{C}} \simeq L^{0,-1} \oplus L^{-1,0} \oplus L^{1,-1} \oplus L^{0,0} \oplus L^{-1,1}$ by sending $L^{m,n}$ to $L^{n,m}$.

It follows that conditions \ref{ax:MSDb} and \ref{ax:MSDc} are satisfied.

As $\ker \phi$ is central in $\tilde{Q}$ and $\tilde{Q}/\tilde{W} \to Q/W$ is surjective, there are surjective homomorphisms
\[Q/W \simeq \tilde{Q}/((\ker \phi)\tilde{W}) \to (\tilde{Q}/\tilde{W})^{\ad} \to (Q/W)^{\ad}.\]
As the adjoint group of a reductive group has trivial center by Corollary 17.62(e) in \cite{MilneAG}, this shows that $(\tilde{Q}/\tilde{W})^{\ad}$ and $(Q/W)^{\ad}$ are canonically isomorphic. Condition \ref{ax:MSDd} follows.

We set $\tilde{G} = \tilde{Q}/\tilde{W}$ and $G = Q/W$ and fix Levi decompositions $\tilde{Q} = \tilde{G} \ltimes \tilde{W}$ and $Q = G \ltimes W$. By condition \ref{ax:MSDe} for $(Q,Y^+)$, we have $Q^{\der} = G^{\der} \ltimes W$. Furthermore, we have $\tilde{Q}^{\der} = \tilde{G}^{\der} \ltimes \tilde{V}$ for some algebraic subgroup $\tilde{V}\subset \tilde{W}$. As $\phi$ is surjective, so is $\phi|_{\tilde{Q}^{\der}}: \tilde{Q}^{\der}\to Q^{\der}$.

Arguing as in the proof that $\phi|_{\tilde{W}}: \tilde{W}\to W$ is an isomorphism with $\tilde{Q}^{\der}$, $Q^{\der}$, and the identity component $(\phi|_{\tilde{Q}^{\der}})^{-1}(W)^0$ of $(\phi|_{\tilde{Q}^{\der}})^{-1}(W)$ in place of $\tilde{Q}$, $Q$, and $\phi^{-1}(W)$ shows that $\phi|_{\tilde{V}}: \tilde{V} \to W$ is an isomorphism. It follows that $\tilde{V} = \tilde{W}$ and $\tilde{Q}^{\der} = \tilde{G}^{\der} \ltimes \tilde{W}$. We deduce that $\tilde{Q}/\tilde{Q}^{\der} \simeq \tilde{G}/\tilde{G}^{\der}$ and $Q/Q^{\der} \simeq G/G^{\der}$.

Now $\tilde{G}/\tilde{G}^{\der}$ is isogenous to the identity component $Z(\tilde{G})^{0}$ of the center $Z(\tilde{G})$ of $\tilde{G}$ and $G/G^{\der}$ is isogenous to $Z(G)^{0}$, see Example 19.25 in \cite{MilneAG}. Furthermore, $\phi$ induces a surjective homomorphism $\bar{\phi}: \tilde{G}\to G$, whose kernel is a central torus isomorphic to $\ker \phi$, and we have an exact sequence
\[ 1 \to \ker \bar{\phi} \to Z(\tilde{G})^{0} \to \bar{\phi}(Z(\tilde{G})^{0}) \to 1.\]
Since $\ker \bar{\phi}$ as well as $\bar{\phi}(Z(\tilde{G})^{0}) \subset Z(G)^{0}$ are almost-direct products of $\mathbb{Q}$-split tori with tori of compact type that are defined over $\mathbb{Q}$ and the category of $\mathbb{Q}$-tori up to isogeny is semi-simple (see Proposition 4.1 in \cite{Brion}), this establishes condition \ref{ax:MSDe} for $(\tilde{Q},\tilde{Y}^+)$.

Hence, $(\tilde{Q},\tilde{Y}^+)$ is a connected mixed Shimura datum and $\phi$ induces a Shimura morphism $(\tilde{Q},\tilde{Y}^+) \to (Q,Y^+)$. Since $\phi$ is surjective and its kernel is a torus, this Shimura morphism is a Shimura covering. Finally, if $W$ is trivial, so is $\tilde{W}$, and the same holds for $U$ and $\tilde{U}$. The lemma follows.
\end{proof}

We are going to use the following lemma to construct the right vertical arrow in the diagram \eqref{diag}.

\begin{lem}\label{lem:milneshih}
Suppose that $(Q,Y^{+})$ is a connected mixed Shimura datum. Then there is a connected mixed Shimura datum $(\tilde{Q},\tilde{Y}^{+})$ and a Shimura covering $(\tilde{Q},\tilde{Y}^{+}) \to (Q,Y^{+})$ such that the associated homomorphism $\tilde{Q} \to Q$ is surjective and its kernel is a central torus of compact type. Furthermore, the derived subgroup $\tilde{Q}^{\der}$ of $\tilde{Q}$ is essentially simply connected (as defined in \cite{LubVen}).

If $(Q,Y^{+})$ is a connected pure Shimura datum or a connected mixed Shimura datum of Kuga type, then the same holds for $(\tilde{Q},\tilde{Y}^{+})$.
\end{lem}

\begin{proof}
Let $W$ be the unipotent radical of $Q$ and set $G = Q/W$. There is an induced connected pure Shimura datum $(G,X^{+}) = (Q,Y^{+})/W$ (see \cite{PinkDiss}, Proposition 2.9). We want to use the construction in Proposition 3.1 and Application 3.4 in Chapter V of \cite{MilneShih} to get a connected pure Shimura datum $(\tilde{G},\tilde{X}^{+})$ and a Shimura morphism $(\tilde{G},\tilde{X}^{+}) \to (G,X^{+})$ such that $\tilde{G}^{\der}$ is simply connected (as defined in Definition 18.5 in \cite{MilneAG}), the induced homomorphism $\tilde{G} \to G$ is surjective, and its kernel is a central torus in $\tilde{G}$.

However, the construction in \cite{MilneShih} does not necessarily preserve condition \ref{ax:MSDe} in the definition of a connected mixed Shimura datum as the kernel of $\tilde{G} \to G$ might have a non-$\mathbb{Q}$-split factor of non-compact type (up to isogeny). The construction in \cite{MilneShih} therefore has to be modified.

The following argument was kindly provided to the authors by Chris Daw: By Lemma 13.3 and Theorem 5.4 in \cite{MilneISV}, there exists $x \in X^+$ and a maximal $\mathbb{Q}$-torus $T$ in $G$ such that $x$ factors through $T_{\mathbb{C}}$. Let $Z$ denote the center of $G$ and $Z^0$ its identity component. Recall that $G = G^{\der}  Z^0$ and $Z^0 \subset T$ by the maximality of $T$. Therefore, $T = T^{\der}  Z^0$, where $T^{\der}$ by abuse of notation denotes the identity component $(T \cap G^{\der})^0$ and is a maximal $\mathbb{Q}$-torus in $G^{\der}$. By Corollary 17.84 in \cite{MilneAG}, we must have $T^{\der} =  T \cap G^{\der}$. Since $Z^0 \subset T$, it follows that
\begin{equation}\label{eq:centercontained}
T^{\der} \cap Z^0 = G^{\der} \cap Z^0.
\end{equation}

We now show that $T^{\der}(\mathbb{R})$ is compact in the same way as done at the end of the proof of Lemma \ref{lem:intersectionfundamentallemma2}: To start with, $T^{\der}(\mathbb{R})$ fixes $x$ and so $T^{\der}(\mathbb{R})$ is contained in the stabilizer $K_x$ of $x$ in $G^{\der}(\mathbb{R})$. The stabilizer $K_x$ is a closed subset of $K_1(\mathbb{R})$, where $K_1$ denotes the algebraic subgroup of $G^{\der}_{\mathbb{R}}$ of elements commuting with $x(\sqrt{-1})$. Let $K_2$ denote the algebraic subgroup of $G^{\ad}_{\mathbb{R}}$ of elements commuting with the image of $x(\sqrt{-1})$ under the homomorphism $G_{\mathbb{R}} \to G^{\ad}_{\mathbb{R}}$. By \ref{ax:MSDd}, conjugation by $x(\sqrt{-1})$ induces a Cartan involution on $G^{\ad}_{\mathbb{R}}$, which implies that $K_2(\mathbb{R})$ is compact (see the definition in \cite{MilneISV}, pp. 274--275). Clearly $K_1$ is mapped into $K_2$ by the homomorphism $G^{\der}_{\mathbb{R}} \to G^{\ad}_{\mathbb{R}}$. As the topological covering $G^{\der}(\mathbb{C}) \to G^{\ad}(\mathbb{C})$ has finite fibers, this implies that $K_1(\mathbb{R})$ is compact. Hence, $K_x$ is a compact closed subgroup of $G^{\der}(\mathbb{R})$ and so $T^{\der}(\mathbb{R})$ is compact.

Now let $\pi: G' \to G^{\der}$ denote the universal covering of $G^{\der}$ as defined in Definition 18.7 in \cite{MilneAG} (see, e.g., Theorem 18.25 and Remark 18.27 in \cite{MilneAG} for the existence) so that $G'$ is a simply connected semisimple algebraic group. The kernel of $\pi$ is finite and of multiplicative type. By Theorem 15.39 in \cite{MilneAG}, $\pi^{-1}(T^{\der})$ is of multiplicative type as well. Let $T'$ denote the identity component $\pi^{-1}(T^{\der})^0$. Then $T'$ is a maximal $\mathbb{Q}$-torus in $G'$ of compact type. Again by Corollary 17.84 in \cite{MilneAG}, we have
\begin{equation}\label{eq:preimage}
T' = \pi^{-1}(T^{\der}).
\end{equation}
Since the homomorphism $\pi$ has finite kernel, the induced map between cocharacter modules
\[ X_{\ast}(T') \to X_{\ast}(T), \quad \chi \mapsto \pi_{\bar{\mathbb{Q}}} \circ \chi\]
is injective, and so we consider $X_{\ast}(T')$ as a submodule of $X_{\ast}(T)$.

In this proof we several times make use of the fact that cocharacter modules of tori are torsion-free as $\mathbb{G}_m$ is connected.

Now we consider the canonical maps $\ad: G\to G^{\ad}$ and $\nu: G \to C$, where $C = G/G^{\der}$, and we let $T^{\ad}$ by abuse of notation denote the maximal $\mathbb{Q}$-torus $\ad(T)$ of $G^{\ad}$. We obtain an isogeny $T \to T^{\ad}\times_{\mathbb{Q}} C$ (with kernel contained in $Z \cap G^{\der}$) and hence another inclusion
\[ X_{\ast}(T) \subset X_{\ast}(T^{\ad}) \oplus X_{\ast}(C), \quad \chi \mapsto (\ad_{\bar{\mathbb{Q}}} \circ \chi, \nu_{\bar{\mathbb{Q}}} \circ \chi).\]
By construction, $X_{\ast}(T')$ is contained in $X_{\ast}(T^{\ad}) = X_{\ast}(T^{\ad}) \oplus \{0\}$. Hence, we can form the following Cartesian square of Galois modules:

\begin{center}
\begin{tikzcd}
P_0 \arrow{r} \arrow{d} & F = X_{\ast}(T^{\ad}) \oplus X_{\ast}(C) \arrow{d} \\   
M = X_{\ast}(T)/X_{\ast}(T') \arrow[hookrightarrow]{r} &  (X_{\ast}(T^{\ad})/X_{\ast}(T')) \oplus X_{\ast}(C).
\end{tikzcd}
\end{center}

Furthermore, since $F$ is torsion-free and the lower horizontal arrow is injective, $P_0$ is also torsion-free. Since the right vertical arrow is surjective, also the left vertical arrow is surjective.

Now we proceed exactly as in \cite{MilneShih}. Let $P_1$ be the kernel of the left vertical arrow, and let $P_2$ be the fiber product of $P_0$ and $X_{\ast}(T)$ over $M$ (the map $X_{\ast}(T) \to M$ being the projection to the quotient). We obtain the following diagram, having exact rows and columns.

\begin{equation}\label{eq:bigdiagram}
\begin{tikzcd}
& & 0 \arrow{d} & 0 \arrow{d} & \\   
& & X_{\ast}(T') \arrow[equal]{r} \arrow{d} & X_{\ast}(T')\arrow{d} & \\
0 \arrow{r} & P_1 \arrow{r}\arrow[equal]{d} & P_2 \arrow{r} \arrow{d} & X_{\ast}(T) \arrow{r} \arrow{d} & 0 \\
0 \arrow{r} & P_1 \arrow{r} & P_0 \arrow{r}  \arrow{d} & M \arrow{r} \arrow{d} & 0 \\
& & 0 & 0 &
\end{tikzcd}
\end{equation}

First, note that $P_2$ is torsion-free, as it is a submodule of $P_0 \oplus X_{\ast}(T)$. Second, observe that $P_1$ is isomorphic to $X_{\ast}(T')$. To see this, suppose that $(0,f) \in P_1\subset M\oplus F$. Then $f = (t,0)$ such that $t \in X_{\ast}(T')$. Therefore, $P_1$ is contained in $\{0\} \oplus X_{\ast}(T') \oplus \{0\}$, and the reverse inclusion is immediate. Thus, $P_1$ is also torsion-free. Hence, the $\mathbb{Z}$-dual of the upper horizontal sequence is also exact. The category equivalence $X^{\ast}$ (see \cite{MilneAG}, Theorem 12.23) then yields a $\mathbb{Q}$-torus $\tilde{T}$ and an exact sequence of $\mathbb{Q}$-tori
\begin{equation}\label{eq:exseq}
1 \to T' \to \tilde{T} \to T \to 1
\end{equation}
such that the $\mathbb{Z}$-dual of the upper horizontal sequence corresponds to the induced exact sequence of character modules of these $\mathbb{Q}$-tori.

Pulling back \eqref{eq:exseq} along $Z^0 \subset T$ yields another exact sequence of tori
\[ 1 \to T' \to \tilde{Z} \stackrel{\eta}{\to} Z^0 \to 1.\]
The group $G$ is isomorphic to $(G' \times_{\mathbb{Q}} Z^0)/E$, where the finite group $E = \pi^{-1}(G^{\der} \cap Z^0)\subset G'$ embeds centrally into the product via $g \in E(\bar{\mathbb{Q}}) \mapsto (g,\pi(g)^{-1}) \in G'(\bar{\mathbb{Q}}) \times Z^0(\bar{\mathbb{Q}})$. The embedding is central since $E$ is mapped into the center of $G^{\der}$ by the homomorphism $\pi$, which has finite kernel.

The left column of \eqref{eq:bigdiagram} yields a homomorphism $T' \hookrightarrow \tilde{T}$ such that the composition $T' \to \tilde{T} \to T$ coincides with the morphism $\pi|_{T'}$, that we recall has finite kernel. Furthermore, we have $\tilde{Z} \subset \tilde{T}$ equal to the pre-image of $Z^0 \subset T$. Since $\pi(T') = T^{\der}$, $T = T^{\der}  Z^0$, and $T^{\der} \cap Z^0$ is finite, we get an isogeny $T' \times_{\mathbb{Q}} \tilde{Z} \to \tilde{T}$. Let $\tilde{E}$ be its kernel. It projects bijectively onto its image in either factor.

If $(g,z) \in \tilde{E}(\bar{\mathbb{Q}}) \subset G'(\bar{\mathbb{Q}}) \times \tilde{Z}(\bar{\mathbb{Q}})$, then we have $g \in E(\bar{\mathbb{Q}})$, hence $\tilde{E}$ is central in $G' \times_{\mathbb{Q}} \tilde{Z}$. We define $\tilde{G} = (G' \times_{\mathbb{Q}} \tilde{Z})/\tilde{E}$. The homomorphism $\id_{G'} \times \eta: G' \times_{\mathbb{Q}} \tilde{Z} \to G' \times_{\mathbb{Q}} Z^0$ induces a surjective homomorphism $\rho: \tilde{G} \to G$. By construction, there is an isomorphism $\tilde{G}^{\der} \simeq G'$.

Let $H \subset G' \times_{\mathbb{Q}} \tilde{Z}$ be the algebraic subgroup defined by
\[ H(\bar{\mathbb{Q}}) = \{ (g,z) \in G'(\bar{\mathbb{Q}}) \times \tilde{Z}(\bar{\mathbb{Q}}); g \in E(\bar{\mathbb{Q}}), \eta(z) = \pi(g)^{-1}\},\]
then $\tilde{E} \subset H$ and $\ker \rho = H/\tilde{E}$. Consider the image $E'$ of $\tilde{E}$ under projection to $T' \subset G'$. Thanks to \eqref{eq:centercontained} and \eqref{eq:preimage}, we have
\[E' = (\pi|_{T'})^{-1}(T^{\der} \cap Z^0) = \pi^{-1}(T^{\der} \cap Z^0) = \pi^{-1}(G^{\der} \cap Z^0) = E.\]
Since the map $\tilde{E} \to E'$ is bijective, this shows that $\ker \rho$ is isomorphic to $\ker \eta$, i.e., to $T'$, and central in $\tilde{G}$.

By construction, $\tilde{T} \simeq (T' \times_{\mathbb{Q}} \tilde{Z})/\tilde{E}$ embeds into $\tilde{G}$ as a maximal torus with $\rho(\tilde{T}) = T$ and $\ker \rho|_{\tilde{T}} \simeq \ker \rho \simeq T'$. Now $x \in \Hom(\mathbb{S}_{\mathbb{C}},G_{\mathbb{C}})$ is induced by $x \in \Hom(\mathbb{S}_{\mathbb{C}},T_{\mathbb{C}})$. Thanks to the exact sequence $1 \to \ker \rho \to \tilde{T} \to T \to 1$ of tori, the latter lifts to $\tilde{x} \in \Hom(S_{\mathbb{C}},\tilde{T}_{\mathbb{C}})$ and hence to $\tilde{x} \in \Hom(\mathbb{S}_{\mathbb{C}},\tilde{G}_{\mathbb{C}})$. We then define $\tilde{Q} = \tilde{G} \times_G Q \simeq \tilde{G} \ltimes W$. The kernel of the surjective homomorphism $\tilde{Q} \to Q$ is isomorphic to $\ker \rho$, which is isomorphic to the $\mathbb{Q}$-torus $T'$ of compact type, and is central in $\tilde{Q}$. If $y \in Y^+$ maps to $x \in X^+$, then $y$ lifts to $\tilde{y} \in \Hom(\mathbb{S}_{\mathbb{C}},\tilde{Q}_{\mathbb{C}})$. Most of Lemma \ref{lem:milneshih} now follows from Lemma \ref{lem:centext}. That $\tilde{Q}^{\der}$ is essentially simply connected follows from the isomorphism $\tilde{Q}^{\der} \simeq \tilde{G}^{\der} \ltimes W$ and the simply connectedness of $\tilde{G}^{\der} \simeq G'$.
\end{proof}

The following lemma gives almost all we need to deduce \ref{ax:4}.

\begin{lem}\label{lem:axiomfourmsv}
	
	Let $\phi:(P,X^{+},\Gamma) \to (Q,Y^{+}, \Delta)\in \mathfrak{M}$. There exist morphisms $ \phi_1:(\tilde{P},\tilde{X}^{+},\tilde{\Gamma}) \to (P,X^{+},\Gamma)$, $\phi_2: (\tilde{P},\tilde{X}^{+},\tilde{\Gamma}) \to(\tilde{Q},\tilde{Y}^{+},\tilde{\Delta}) $, $\phi_3:(\tilde{Q},\tilde{Y}^{+},\tilde{\Delta}) \to (Q,Y^{+}, \Delta)\in \mathfrak{M}$ such that $\phi \circ \phi_1 = \phi_3 \circ \phi_2$ and $\mathcal{F}(\phi_1)$ is a Shimura covering, $\mathcal{F}(\phi_2)$ is a Shimura submersion, and $\mathcal{F}(\phi_3)$ is a Shimura immersion. Moreover, $\phi_2:\tilde{P}\to \tilde{Q}$ is surjective, its kernel is connected, and $\phi_2(\tilde{\Gamma}) = \tilde{\Delta}$.
\end{lem}

\begin{proof}
Let $K$ be the identity component of $\ker \phi$. By Proposition 2.9 in \cite{PinkDiss}, there is a quotient Shimura datum $(P,X^{+})/K$ and the Shimura morphism $(P,X^{+}) \to (Q,Y^{+})$ factors as $(P,X^{+}) \to (P,X^{+})/K \to (Q,Y^{+})$, where the second Shimura morphism is a Shimura immersion as the homomorphism $P/K \to Q$ has finite kernel. Thanks to Lemma 2.2 in \cite{OrrThesis} (which is formulated for reductive groups, but holds with the same proof for arbitrary linear algebraic groups), we can find a congruence subgroup $\Gamma_K$ of $(P/K)(\mathbb{Q})_{+}$ that gives a factorization of $\phi$ as
\[ (P,X^{+},\Gamma) \to ((P,X^{+})/K,\Gamma_K) \to (Q,Y^{+}, \Delta).\]
Here, we use that the intersection of a congruence subgroup of $(P/K)(\mathbb{Q})$ with $(P/K)(\mathbb{Q})_{+}$ is again a congruence subgroup of $(P/K)(\mathbb{Q})$ (see Section \ref{sec:msv} of this article and Proposition 2.2 in \cite{RohlfsSchwermer}).

It suffices to prove the lemma for $(Q,Y^+) = (P,X^{+})/K$. If $(P,X^+)$ is pure or of Kuga type, then it follows from Remark \ref{rmk:imageofu} that the same holds for $(P,X^{+})/K$. We then have by construction that the Shimura morphism $(P,X^{+}) \to (Q,Y^{+})$ is a Shimura submersion, $\phi: P \to Q$ is surjective, and the kernel of $\phi$ is connected. 

From Lemma \ref{lem:milneshih}, we get a connected mixed Shimura datum $(\tilde{Q},\tilde{Y}^+)$ and a Shimura covering $(\tilde{Q},\tilde{Y}^+) \to (Q,Y^+)$ with surjective associated homomorphism $\phi_3:\tilde{Q} \to Q$ and $\ker \phi_3$ a central torus of compact type. Furthermore, $\tilde{Q}^{\der}$ is essentially simply connected (as defined in \cite{LubVen}). We define $\tilde{P} = \tilde{Q} \times_Q P$. We automatically get two homomorphisms $\phi_1:\tilde{P} \to P$ and $\phi_2: \tilde{P} \to \tilde{Q}$ such that $\phi \circ \phi_1 = \phi_3 \circ \phi_2$.

Note that the maps $\tilde{Y}^+ \to Y^+$ and $X^+ \to Y^+$ are surjective because of Proposition 5.1 in \cite{MilneISV}. This implies that, given $y \in Y^+$, we can choose pre-images of $y$ in $\tilde{Y}^+$ and $X^+$ and thus find $\tilde{x} \in\Hom(\mathbb{S}_{\mathbb{C}},\tilde{P}_{\mathbb{C}})$ such that $y =\phi_{\mathbb{C}} \circ (\phi_1)_{\mathbb{C}} \circ\tilde{x}=(\phi_3)_{\mathbb{C}}\circ (\phi_2)_{\mathbb{C}} \circ \tilde{x}$. 

Now, $\ker \phi_1$ and $\ker \phi_3$ are isomorphic. Furthermore, $\phi_1$ is surjective and $\ker \phi_1$ is central in $\tilde{P}$. Hence, $\phi_1$ satisfies the hypotheses of Lemma \ref{lem:centext}. We then get a connected mixed Shimura datum $(\tilde{P},\tilde{X}^+)$ and a Shimura covering $(\tilde{P},\tilde{X}^+) \to (P,X^+)$.

For $H \in \{\tilde{P},P,\tilde{Q},Q\}$, let $W_H$ denote the unipotent radical of $H$, let $U_H$ denote the unipotent subgroup from the definition of a connected mixed Shimura datum, and let $u_H: H \to H/U_H$ denote the corresponding quotient map. By construction, both $(u_P \circ \phi_1)_{\mathbb{C}} \circ \tilde{x}$ and $(u_{\tilde{Q}}\circ \phi_2)_{\mathbb{C}} \circ \tilde{x}$ are defined over $\mathbb{R}$. It follows from Remark \ref{rmk:imageofu}, from the fact that $W_{\tilde{P}} \simeq W_{\tilde{Q}} \times_{W_Q} W_P$, and from the proof of Lemma \ref{lem:centext} that we can assume that $U_{\tilde{P}} \simeq W_{\tilde{Q}} \times_{W_Q} U_P \simeq U_{\tilde{Q}} \times_{U_Q} U_P$ and that $U_{\tilde{Q}} \to U_Q$ is an isomorphism. We deduce that $\tilde{P}/U_{\tilde{P}} \simeq (\tilde{Q}/U_{\tilde{Q}}) \times_{Q/U_Q} (P/U_P)$. This implies that $(u_{\tilde{P}})_{\mathbb{C}} \circ \tilde{x}$ is also defined over $\mathbb{R}$. This means that in the proof of Lemma \ref{lem:centext} we do not have to replace $\tilde{x}$ for \ref{ax:MSDa} to be satisfied and so we can assume that $\tilde{x} \in \tilde{X}^{+}$. We therefore also obtain a Shimura morphism $(\tilde{P},\tilde{X}^{+}) \to (\tilde{Q},\tilde{Y}^{+})$ such that the obvious diagram commutes.

Both $(\tilde{P},\tilde{X}^+)$ and $(\tilde{Q},\tilde{Y}^+)$ are pure or of Kuga type if $(P,X^+)$ and $(Q,Y^+)$ both are.

Furthermore, $\phi_2$ is surjective and its kernel is isomorphic to the kernel of $\phi$, hence connected. In particular, $(\tilde{P},\tilde{X}^+) \to (\tilde{Q},\tilde{Y}^+)$ is a Shimura submersion.

We still have to find the congruence subgroups $\tilde{\Gamma}$ and $\tilde{\Delta}$.

We fix a Levi decomposition $\tilde{Q} = \tilde{G} \ltimes W_{\tilde{Q}}$. Then we have $\tilde{Q}^{\der} = \tilde{G}^{\der} \ltimes W_{\tilde{Q}}$ because of \ref{ax:MSDe}. Now $(\tilde{P}/\tilde{P}^{\der})(\mathbb{Q}) \subset (\tilde{P}/\tilde{P}^{\der})(\mathbb{A}_f)$ is discrete by \ref{ax:MSDe} and Theorem 5.26 in \cite{MilneISV}. By choosing a sufficiently small congruence subgroup $\tilde{\Gamma}$ of $\tilde{P}(\mathbb{Q})_{+}$ inside the pre-image of $\Gamma$ under $\phi_1$, we can hence suppose that $\tilde{\Gamma} \subset \tilde{P}^{\der}(\mathbb{Q})$ and therefore $\tilde{\Delta} := \phi_2(\tilde{\Gamma}) \subset \tilde{Q}^{\der}(\mathbb{Q})$. We now apply Proposition 0.1(ii) in \cite{LubVen} and Theorem 4.1 in \cite{PlatonovRapinchuk} to the induced (surjective) homomorphism $\tilde{P}^{\der} \to \tilde{Q}^{\der}$ to deduce that $\tilde{\Delta}$ is a congruence subgroup of $\tilde{Q}^{\der}(\mathbb{Q})$, using that $\tilde{Q}^{\der}$ is essentially simply connected (as defined in \cite{LubVen}) by construction. Furthermore, $\tilde{\Delta}$ is an arithmetic subgroup of $\tilde{Q}(\mathbb{Q})$ by Theorem 4.1 in \cite{PlatonovRapinchuk}. We have $\tilde{\Delta} \subset \tilde{Q}(\mathbb{Q})_{+}$ since $\tilde{\Gamma} \subset \tilde{P}(\mathbb{Q})_{+}$.

Since $\tilde{Q}^{\der} \subset \tilde{Q}$ is a closed algebraic subgroup, it follows that $\tilde{Q}^{\der}(\mathbb{Q}) \subset \tilde{Q}(\mathbb{Q})$ is closed in the congruence topology. Since the congruence topology on $\tilde{Q}^{\der}(\mathbb{Q})$ is induced by the congruence topology on $\tilde{Q}(\mathbb{Q})$, this implies that $\tilde{\Delta}$ is also closed with respect to the congruence topology in $\tilde{Q}(\mathbb{Q})$, so a congruence subgroup of $\tilde{Q}(\mathbb{Q})$. 
\end{proof}

We are now ready to verify \ref{ax:4}. Let $(P,X^+,\Gamma)$ and $(Q,Y^+,\Delta)$ be two elements of $\mathfrak{V}$ and let $\phi \in\mathfrak{M}$ be a morphism between them, induced by $\phi: P \to Q$. As already mentioned, the image of $\mathcal{F}(\phi)$ is closed by Remark 5.5 in \cite{G17}.

We now apply Lemma \ref{lem:axiomfourmsv} to our situation. This gives distinguished morphisms $\phi_1, \phi_2, \phi_3$ with $\phi \circ \phi_1 = \phi_3 \circ \phi_2$ and such that
 $\mathcal{F}(\phi_1)$ is a Shimura covering, $\mathcal{F}(\phi_2)$ is a Shimura submersion, and $\mathcal{F}(\phi_3)$ is a Shimura immersion. By Proposition \ref{prop:pinkfacts}, the morphism $\mathcal{F}(\phi_1)$ is finite and surjective, $\mathcal{F}(\phi_2)$ is surjective, and $\mathcal{F}(\phi_3)$ has finite fibers.

 It remains to check that all fibers of the morphism $\mathcal{F}(\phi_2)$ are irreducible of the same dimension. This is equivalent to showing that the same holds for the corresponding morphism of complex analytic spaces $\tilde{\Gamma} \backslash \tilde{X}^+ \to \tilde{\Delta} \backslash \tilde{Y}^+$. By Lemma \ref{lem:axiomfourmsv}, the kernel of $\tilde{P} \to \tilde{Q}$ is connected. This implies that the fibers of the induced map $\tilde{X}^+ \to \tilde{Y}^+$ are irreducible as complex analytic spaces (cf. Remark 5.3(2) in \cite{G17}). Since all fibers of $\tilde{X}^+ \to \tilde{Y}^+$ are isomorphic as complex analytic spaces and $\phi_2(\tilde{\Gamma}) = \tilde{\Delta}$, \ref{ax:4} now follows. 

\section{Special and weakly special subvarieties}\label{sec:special}

In this section as well as in the two following ones, we fix an algebraically closed field $K$ of characteristic 0 and a distinguished category $\mathfrak{C}$ with objects $\mathfrak{V}$, morphisms $\mathfrak{M}$, and functor $\mathcal{F}$ to the category of algebraic varieties over $K$.
From now on, we will drop the distinction between objects and morphisms in $\mathfrak{C}$ and their images under $\mathcal{F}$ in order to avoid excessive notational baggage. The reader is advised to check that all our constructions exist and make sense in $\mathfrak{C}$ when necessary, and not only in the category of varieties over $K$.

We start by noticing that, by \ref{ax:4}, images of distinguished morphisms are closed subvarieties. Therefore, it makes sense to give the following definition.

\begin{defn}\label{defn:special}
A subvariety of a distinguished variety $X$ is called a \emph{special subvariety} of $X$ if it is the image of a distinguished morphism.
\end{defn}

We will define weakly special subvarieties to be irreducible components of images of pre-images of $K$-rational points via distinguished morphisms. With the following lemma, we see that such images are actually finite unions of subvarieties.

\begin{lem}\label{lem:specialclosed}
 If $\psi: Z \to Y$ and $\phi: Z \to X$ are distinguished morphisms and $y \in Y(K)$, then $\phi(\psi^{-1}(y))$ is closed in $X$.
\end{lem}

\begin{proof}
We consider the distinguished morphism $(\psi,\phi): Z \to Y \times_K X$ and the special subvariety $S = (\psi,\phi)(Z) \subset Y \times_K X$. We have $\phi(\psi^{-1}(y)) = q^{-1}(S)$, where $q$ is the (not necessarily distinguished) morphism $X \to Y \times_K X$ induced by the constant map $X \to \{y\}$ and the identity on $X$. Therefore $\phi(\psi^{-1}(y))$ is closed.
\end{proof}

\begin{defn}\label{defn:weaklyspecial}
A subvariety $W$ of a distinguished variety $X$ is called \emph{weakly special} if there exist distinguished morphisms $\psi: Z \to Y$, $\phi: Z \to X$ and a point $y \in Y(K)$ such that $W$ is an irreducible component of $\phi(\psi^{-1}(y))$.
\end{defn}

\begin{rmk}\label{rmk:equivalentdefinition}
	Note that in the category $\mathfrak{C}_{\mathrm{mSv}}$ our definitions of special and weakly special subvarieties coincide with those in \cite{Pink} (but recall that the definition of a connected mixed Shimura datum there is different from ours) and in \cite{G17}.
	
	In the category $\mathfrak{C}_{\mathrm{pSv}}$, they also coincide with the more usual definition as in \cite{OrrThesis}, where one does not assume \ref{ax:MSDe}. One can show this using Remarks \ref{rmk:quotientshimuradatum} and \ref{rmk:shimurasubdatum} as well as the diagram in the proof of Proposition 2.6 at the bottom of p. 48 of \cite{OrrThesis}.
	
	Moreover, in the category of connected commutative algebraic groups $\mathfrak{C}_{\mathrm{comm}}$, one can easily check that weakly special subvarieties of a connected commutative algebraic group $X$ are cosets $xY$, where $Y$ is a connected algebraic subgroup of $X$ and $x\in X(K)$. The weakly special subvariety $xY$ is special if and only if it contains a torsion point, i.e., if $x$ can be chosen to have finite order.
\end{rmk}

The reader might not be very surprised by the following statement, which follows from \ref{ax:3}.

\begin{lem}
Special subvarieties are weakly special.
\end{lem}

We are going to see below two fundamental lemmas saying that images via distinguished morphisms, components of pre-images via distinguished morphisms, and components of intersections of (weakly) special subvarieties are (weakly) special.
First we need the following lemma that is also going to be useful later.

\begin{lem}\label{lem:wsnocomp}
Let $W$ be a weakly special subvariety of a distinguished variety $X$. Then, there exist distinguished morphisms $\psi: Z \to Y$, $\phi: Z \to X$ and a point $y \in Y(K)$ such that $W = \phi(\psi^{-1}(y))$.
\end{lem}

\begin{proof}
By definition, there exist distinguished morphisms $\psi': Z' \to Y'$, $\phi': Z' \to X$ and a point $y' \in Y'(K)$ such that $W$ is an irreducible component of $\phi'(\psi'^{-1}(y'))$.

We apply \ref{ax:4} to get distinguished morphisms $\psi'_1: U' \to Z'$, $\psi'_2: U' \to V'$, and $\psi'_3: V' \to Y'$ such that $\psi' \circ \psi'_1 = \psi'_3 \circ \psi'_2$, $\psi'_1$ is surjective, $\psi'_3$ has finite fibers, $\psi_2'$ is surjective, and $(\psi_2')^{-1}(v')$ is irreducible for all $v' \in V'(K)$.

It follows that
\[ \psi'^{-1}(y') = \psi'_1((\psi'_3 \circ \psi'_2)^{-1}(y')) = \bigcup_{v' \in (\psi'_3)^{-1}(y')}{\psi'_1((\psi'_2)^{-1}(v'))},\]
and hence
\[ \phi'(\psi'^{-1}(y')) = \bigcup_{v' \in (\psi'_3)^{-1}(y')}{(\phi' \circ \psi'_1)((\psi'_2)^{-1}(v'))}.\]

Since $\psi'_3$ has finite fibers and $(\psi_2')^{-1}(v')$ is irreducible for all $v' \in V'(K)$, there are at most finitely many $v' \in (\psi'_3)^{-1}(y')$ and each $(\phi' \circ \psi'_1)((\psi'_2)^{-1}(v'))$ is a subvariety of $X$ by Lemma \ref{lem:specialclosed}. We deduce that
\[ W = (\phi' \circ \psi'_1)((\psi'_2)^{-1}(v'))\]
for some $v' \in (\psi'_3)^{-1}(y')$ and the lemma follows.
\end{proof}

\begin{lem}\label{lem:(pre)image}
Images and irreducible components of pre-images of (weakly) special subvarieties under distinguished morphisms are (weakly) special.
\end{lem}

\begin{proof}
For images, this follows directly from Lemma \ref{lem:wsnocomp} and the fact that the composition of distinguished morphisms is distinguished.

We now turn to pre-images: Let $\psi: Y \to Z$, $\phi: U \to X$, and $\chi: Y \to X$ be distinguished morphisms. Thanks to Lemma \ref{lem:wsnocomp}, it suffices to prove that every irreducible component of $\phi^{-1}(\chi(Y))$ is special and that every irreducible component of $\phi^{-1}(\chi(\psi^{-1}(z)))$ is weakly special for every choice of $z \in Z(K)$.

Now, $\phi^{-1}(\chi(Y))$ equals the projection of $(U \times_X Y)_{\red} \subset U \times_K Y$ onto $U$. Therefore, all its irreducible components are special by \ref{ax:2}.

For the weakly special case, we note that $\phi^{-1}(\chi(\psi^{-1}(z)))$ equals the projection of $\Psi^{-1}(z) \cap (U \times_X Y)_{\red} \subset U \times_K Y$ onto $U$, where $\Psi: U \times_K Y \to Z$ is the composition of the canonical projection with $\psi$ (and as such a distinguished morphism).

By \ref{ax:2}, there exist distinguished varieties $X_1,\hdots,X_n$ and distinguished morphisms $\phi_i: X_i \to U \times_K Y$ ($i=1,\hdots,n$) such that $\bigcup_{i=1}^{n}{\phi_i(X_i)} = (U \times_X Y)_{\red} \subset U \times_K Y$. It follows that every irreducible component of
\[ \Psi^{-1}(z) \cap (U \times_X Y)_{\red} = \bigcup_{i=1}^{n}{\phi_i((\Psi \circ \phi_i)^{-1}(z))}\]
is weakly special. The same follows for every irreducible component of $\phi^{-1}(\chi(\psi^{-1}(z)))$.
\end{proof}

\begin{lem}\label{lem:intersection}
Every irreducible component of an intersection of two (weakly) special subvarieties is (weakly) special.
\end{lem}

\begin{proof}
Let $X$ be a distinguished variety and $S, S'$ two special subvarieties. There exists a distinguished morphism $\phi: Y \to X$ such that $S = \phi(Y)$. By Lemma \ref{lem:(pre)image}, every irreducible component of $\phi^{-1}(S')$ is special. Applying Lemma \ref{lem:(pre)image} again shows that the same holds for every irreducible component of $\phi(\phi^{-1}(S')) = S \cap S'$.

Let now $W,W'$ be two weakly special subvarieties of $X$. By Lemma \ref{lem:wsnocomp}, there exist distinguished morphisms $\phi: Y \to X$ and $\psi: Y \to Z$ as well as a point $z \in Z(K)$ such that $W = \phi(\psi^{-1}(z))$. We have $W \cap W' = \phi(\phi^{-1}(W') \cap \psi^{-1}(z))$. By Lemma \ref{lem:(pre)image}, it suffices to show that every irreducible component of $\phi^{-1}(W') \cap \psi^{-1}(z) = (\phi,\psi)^{-1}(W' \times_K \{z\})$ is weakly special. Applying Lemma \ref{lem:(pre)image} again, we see that it suffices to show that $W' \times_K \{z\} \subset X \times_K Z$ is weakly special.

By Lemma \ref{lem:wsnocomp}, there exist distinguished morphisms $\phi': Y' \to X$ and $\psi': Y' \to Z'$ and a point $z' \in Z'(K)$ such that $W' = \phi'(\psi'^{-1}(z'))$. We then have that $W' \times_K \{z\} = (\phi' \times \id_{Z})((\psi' \times \id_{Z})^{-1}(z',z))$ is weakly special.
\end{proof}

\begin{rmk}
While \ref{ax:2} might seem slightly unnatural, it is unavoidable (given \ref{ax:1}) if one wants Lemmas \ref{lem:(pre)image} and \ref{lem:intersection} to hold (cf. the proof of \ref{ax:2} in Section \ref{sec:axiomsformsv} above).
\end{rmk}

We conclude this section by comparing our approach with the concept of a special structure on a positive-dimensional quasi-projective algebraic variety $X$ over $\mathbb{C}$ as defined by Ullmo in D\'efinition 2.1 in \cite{Ullmo}. Such a special structure is a collection of sets $\Sigma_i(X^n)$ ($n \in \mathbb{N} = \{1,2,\hdots\}$, $i = 0,\hdots,n\dim X$) such that
\begin{enumerate}[label=(\alph*)]
\item the elements of $\Sigma_i(X^n)$ are subvarieties of $X^n$ of dimension $i$ (called special subvarieties of $X^n$ of dimension $i$);
\item $X^n \in \Sigma_{n \dim X}(X^n)$;
\item if $A \in \Sigma_i(X^n)$ and $B \in \Sigma_j(X^n)$, then every irreducible component of $A \cap B$ belongs to some $\Sigma_k(X^n)$;
\item $\Sigma(X^n) := \bigcup_{i=0}^{n\dim X}{\Sigma_i(X^n)}$ is countable;
\item every $Z \in \Sigma(X^n)$ contains a Zariski dense set of points that lie in $\Sigma_0(X^n)$;
\item $\Sigma_0(X^n) = \Sigma_0(X)^n$;
\item for every $n$-tuple of integers $(a_1,\hdots,a_n) \in \{0,\hdots,\dim X\}^n$, we have
\[ \prod_{i=1}^{n}{\Sigma_{a_i}(X)} \subset \Sigma_{a_1+\cdots+a_n}(X^n),\]
where the inclusion map is the natural one;
\item the set $\Sigma_{\dim X}(X^2)$ contains infinitely many subvarieties of $X^2$ on which the projection to either factor restricts to a surjective and finite morphism;
\item for every $Y \in \Sigma(X) \backslash \Sigma_0(X)$, the sets
\[ \Sigma_i(Y^n) := \{ S \in \Sigma_i(X^n); S \subset Y^n\} \quad (n \in \mathbb{N},\mbox{ } i = 0,\hdots,n\dim Y) \]
satisfy Properties (f) to (h) as well (with $Y$ instead of $X$).
\end{enumerate}

If we take $K = \mathbb{C}$ and we take $X$ to be an arbitrary distinguished variety in our setting, then all $X^n$ ($n \in \mathbb{N}$) are distinguished varieties by \ref{ax:1} and we can define $\Sigma_i(X^n)$ to be the set of special subvarieties of $X^n$ (in our sense) of dimension $i$ for $n \in \mathbb{N}$, $i = 0,\hdots,n\dim X$. Properties (a) and (b) are then immediate. Property (c) follows from Lemma \ref{lem:intersection}. Properties (f) and (g) follow directly from \ref{ax:1}. Properties (e) and (h) are not satisfied in general; $\mathfrak{C}_{\mathrm{triv}}(Y)$ provides a counterexample if $\dim Y > 0$ (with $X = Y$). Property (d) is also not satisfied in general; $\mathfrak{C}_{\mathrm{add}}$ provides a counterexample (with $X = \mathbb{G}_{a,\mathbb{C}}$) since $\mathbb{G}^2_{a,\mathbb{C}}$ contains uncountably many distinct lines through the origin.

Conversely, given a special structure on some $X$ as defined by Ullmo, it is unclear how to obtain a distinguished category in our sense. So the two approaches have some overlap, which is to be expected as they attempt to formalize the same phenomenon, but they are not equivalent.

\section{Defect and defect condition}\label{sec:defcon}

Lemma \ref{lem:intersection} shows that the following definition makes sense:

\begin{defn}
	For a subvariety $V$ of a distinguished variety $X$, let $\langle V \rangle$ denote the smallest special subvariety of $X$ that contains $V$ and let $\langle V \rangle_{\ws}$ denote the smallest weakly special subvariety of $X$ that contains $V$. We call $\delta(V) = \dim \langle V \rangle - \dim V$ the \emph{defect} of $V$ and $\delta_{\ws}(V) = \dim \langle V \rangle_{\ws} - \dim V$ the \emph{weak defect} of $V$.
\end{defn}

As already mentioned in the introduction, the concept of defect is not new as it appears already in \cite{PinkUnpubl}. In \cite{HP}, the authors introduce as geodesic defect what we call here weak defect.

\begin{thm}[Defect Condition]\label{thm:defectcondition}
	Let $X$ be a distinguished variety and $V' \subset V \subset X$ a sequence of nested subvarieties. Then
	\[ \dim \langle V \rangle - \dim \langle V \rangle_{\ws} \leq \dim \langle V' \rangle - \dim \langle V' \rangle_{\ws},\]
	or, equivalently,
	$$
	\delta(V) - \delta_{\ws}(V)\leq \delta(V') - \delta_{\ws}(V').
	$$
\end{thm}

Thanks to Proposition \ref{prop:axiomsformsv}, this proves Conjecture 4.4 in \cite{HP}, at least for an ambient connected mixed Shimura variety or an ambient semiabelian variety. Note that, in the case of connected pure Shimura varieties, this is essentially Proposition 4.4 in \cite{DawRen}. Other special cases of Theorem \ref{thm:defectcondition} were established in \cite{HP} itself as well as in \cite{PilaTsimerman}. Recent work of Cassani \cite{Cass2} deals with the defect condition for connected mixed Shimura varieties.

Before proving Theorem \ref{thm:defectcondition}, we collect some preliminary results that are going to be useful here and in the next section.

\begin{thm}[Fiber Dimension Theorem]\label{thm:fdt}
Let $f: V \to V'$ be a dominant morphism of varieties over an algebraically closed field. The following hold:
\begin{enumerate}
\item For every $v' \in V'$, every irreducible component of $f^{-1}(v')$ has dimension at least $\dim V - \dim V'$.
\item There exists $U \subset f(V)$ open and dense in $V'$ such that for every $v'\in U$, we have $\dim f^{-1}(v') = \dim V - \dim V'$.
\end{enumerate}
\end{thm}

\begin{proof}
This follows from Corollary 14.116 and Remark 14.117 in \cite{GW} and \cite{GWErr}.
\end{proof}

The following two lemmas both assert that a certain function is constant on all (possibly non-closed) points of a given variety. By constructibility (see Appendix E of \cite{GW}), it suffices to show that the respective function is constant on the closed points of the variety and we will prove the lemmas in this modified form.

\begin{lem}\label{lem:dimspecial}
Let $\phi: X \to Y$ be a distinguished morphism and $S \subset X$ a special subvariety. Then the function $\dim_s \phi|_S^{-1}(\phi(s))$ is constant on $S$.
\end{lem}

\begin{proof}
We have $S = \psi(Z)$ for a distinguished morphism $\psi: Z \to X$. Let $F$ be an irreducible component of $\phi|_S^{-1}(y)$ for some $y \in (\phi(S))(K)$. Note that $\phi(S) = (\phi \circ \psi)(Z)$ is a subvariety of $Y$ by \ref{ax:4}. By the Fiber Dimension Theorem, we have $\dim F \geq \dim S - \dim \phi(S)$; we want to show that $\dim F = \dim S - \dim \phi(S)$.

All irreducible components of $(\phi \circ \psi)^{-1}(y) = (\phi|_S \circ \psi)^{-1}(y)$ are weakly special of dimension
\[ \dim Z - \dim (\phi \circ \psi)(Z) = \dim Z - \dim \phi(S)\]
by \ref{ax:4} and by the Fiber Dimension Theorem. By Lemma \ref{lem:(pre)image}, their images under $\psi$ are closed in $X$.

We consider the set $\mathcal{I}$ of all such irreducible components that are mapped onto $F$ by $\psi$. Since $S = \psi(Z)$, we have $\mathcal{I} \neq \emptyset$.

We now assume that $\dim F > \dim S - \dim \phi(S)$ and aim for a contradiction. By the Fiber Dimension Theorem applied to $\psi|_V: V \to F$ for each $V \in \mathcal{I}$, there exists a point $x \in F(K)$ that is not contained in the image of any irreducible component of $(\phi \circ \psi)^{-1}(y)$ not belonging to $\mathcal{I}$ such that
\[ \dim \psi|_V^{-1}(x) = \dim V - \dim F < \dim V - \dim S + \dim \phi(S) = \dim Z - \dim S\]
for all $V \in \mathcal{I}$. Since $x$ is not contained in the image of any irreducible component of $(\phi \circ \psi)^{-1}(y)$ not belonging to $\mathcal{I}$, we have
\[ \psi^{-1}(x) = \bigcup_{V \in \mathcal{I}}{\psi|_V^{-1}(x)},\]
and therefore $\dim \psi^{-1}(x) < \dim Z - \dim S = \dim Z - \dim \psi(Z)$, contradicting the Fiber Dimension Theorem. The lemma follows.
\end{proof}

\begin{lem}\label{lem:dimweaklyspecial}
Let $\phi: X \to Y$ be a distinguished morphism and let $W \subset X$ be a weakly special subvariety. Then, the function $\dim_w \phi|_W^{-1}(\phi(w))$ is constant on $W$.
\end{lem}

\begin{proof}
By Lemma \ref{lem:wsnocomp}, we can write $W = \psi(\chi^{-1}(z))$ for distinguished morphisms $\psi: U \to X$, $\chi: U \to Z$ and $z \in Z(K)$. Let $\Delta$ be the diagonal in $U \times_K U$.

Let $F$ be an irreducible component of $\phi|_W^{-1}(y)$ for some $y \in (\phi(W))(K)$. Then $F \times_K \{z\}$ is an irreducible component of
\[ F' = (\phi^{-1}(y) \cap \psi(\chi^{-1}(z))) \times_K \{z\} = (\phi \times \id_Z)|_{(\psi \times \chi)(\Delta)}^{-1}(y,z),\]
where $\phi \times \id_Z: X \times_K Z \to Y \times_K Z$ is a distinguished morphism.

By Lemma \ref{lem:(pre)image}, $(\psi \times \chi)(\Delta)$ is special since $\Delta$ is special as it is the image of the distinguished morphism $(\id_U,\id_U)$. It follows from Lemma \ref{lem:dimspecial} that every irreducible component of $F'$ has the same dimension, independently of the choice of $y$. The lemma follows.
\end{proof}

\begin{lem}\label{lem:dimconstant}
Let $f: V \to V'$ be a dominant morphism of varieties over an algebraically closed field such that $\dim_v f^{-1}(f(v))$ is constant on $V$. Let $V_1' \subset V'$ be a subvariety of $V'$ and suppose that $V_1$ is an irreducible component of $f^{-1}(V_1')$ that dominates $V_1'$. Then $\dim V - \dim V' = \dim V_1 - \dim V_1'$.
\end{lem}

\begin{proof}
The set of irreducible components of $f^{-1}(V_1')$ that dominate $V_1'$ is in bijection with the set of irreducible components of the generic fiber of $f|_{f^{-1}(V_1')}: f^{-1}(V_1') \to V_1'$, see (2.1.8) in \cite{EGA_I}. By hypothesis and the Fiber Dimension Theorem, the irreducible components of this generic fiber all have the same dimension $\dim V - \dim V'$. It follows from the Fiber Dimension Theorem that all irreducible components of $f^{-1}(V_1')$ that dominate $V_1'$, so in particular $V_1$, are of dimension $\dim V_1' + \dim V - \dim V'$.
\end{proof}

\begin{proof}[Proof of Theorem \ref{thm:defectcondition}]
Note that $\langle V \rangle = \langle \langle V \rangle_{\ws}\rangle$. Therefore it is enough to show that, if $W' \subset W \subset X$ is a sequence of nested subvarieties, with $W,W'$ weakly special, then
\[ \dim \langle W \rangle - \dim W \leq \dim \langle W' \rangle - \dim W'.\]

We can assume without loss of generality that $W'$ is an irreducible component of $\langle W' \rangle \cap W$.

We will choose the notation such that $\tilde{S}'$ denotes a special subvariety while all letters derived from $W$ denote weakly special subvarieties. They will be defined by taking irreducible components of intersections or pre-images under distinguished morphisms of (weakly) special subvarieties and we tacitly use Lemmas \ref{lem:(pre)image} and \ref{lem:intersection} to see that they are (weakly) special.

By Lemma \ref{lem:wsnocomp}, there exist distinguished morphisms $\phi: X' \to X$, $\psi: X' \to Y'$ and a point $y' \in Y'(K)$ such that $W = \phi(\psi^{-1}(y'))$.

For reference, the following diagram shows most of the weakly special subvarieties that will appear.
\begin{equation*}
\begin{tikzcd}
\{y'\} \arrow[r,symbol=\subset] & \psi(\tilde{W}) & \arrow[l,symbol=\subset] \psi(\tilde{W}')\\
\hat{W} \arrow[r,symbol=\subset] \arrow[u] \arrow[d] & \tilde{W}  \arrow[u] \arrow[ld] & \arrow[l,symbol=\subset] \arrow[ld] \tilde{W}' = \overline{W} \arrow[u] \\
W & \arrow[l,symbol=\subset] W' &
\end{tikzcd}
\end{equation*}
Here, the varieties in the top row are contained in $Y'$, those in the middle row in $X'$, and those in the bottom row in $X$. Furthermore, all maps are surjective and induced by either $\phi$ (middle row to bottom row) or $\psi$ (middle row to top row).

Let $\hat{W}$ be an irreducible component of $\psi^{-1}(y')$ such that $\phi(\hat{W}) = W$.
As $W \subset \phi(\langle \hat{W} \rangle)$ and $\phi(\langle \hat{W} \rangle)$ is special by Lemma \ref{lem:(pre)image}, we must have $\langle W \rangle \subset \phi(\langle \hat{W} \rangle)$. On the other hand, $\langle \hat{W} \rangle$ is contained in $\phi^{-1}(\langle W \rangle )$ by Lemma \ref{lem:(pre)image}, so it follows that $\phi(\langle \hat{W} \rangle ) = \langle W \rangle $.

Let $\tilde{W}$ be an irreducible component of $\phi|_{\langle \hat{W} \rangle}^{-1}(W) = \phi^{-1}(W) \cap \langle \hat{W} \rangle$ that contains $\hat{W}$; we deduce that $\phi(\tilde{W}) = W$. By Lemmas \ref{lem:dimspecial} and \ref{lem:dimconstant}, applied to $\phi|_{\langle \hat{W} \rangle}$, $W$, and $\tilde{W}$, we have
\begin{equation}\label{eq:zero}
\dim \tilde{W} - \dim W = \dim\langle \hat{W} \rangle -\dim \langle W \rangle,
\end{equation}
or equivalently
\begin{equation}\label{eq:three}
\dim \langle W \rangle - \dim W = \dim \langle \hat{W} \rangle - \dim \tilde{W}.
\end{equation}

Applying Lemmas \ref{lem:dimspecial} and \ref{lem:dimweaklyspecial} together with the Fiber Dimension Theorem to $\psi|_{\langle \hat{W} \rangle}$ and $\psi|_{\tilde{W}}$ respectively and noting that $\langle \hat{W} \rangle$ as well as $\tilde{W}$ contain an irreducible component of $\psi^{-1}(y')$, we find that
$$
\dim \langle \hat{W} \rangle - \dim \psi(\langle \hat{W} \rangle)=\dim  \psi^{-1}(y') = \dim \tilde{W} - \dim \psi(\tilde{W}),
$$
which implies
\begin{equation}\label{eq:four}
\dim \langle \hat{W} \rangle - \dim \tilde{W} = \dim \psi(\langle \hat{W} \rangle) - \dim \psi(\tilde{W}).
\end{equation}

There exists an irreducible component $\overline{W}$ of $\phi|_{\tilde{W}}^{-1}(W')$ that contains an irreducible component of $\phi|_{\hat{W}}^{-1}(W')$ that surjects onto $W'$. Let $\tilde{S}'$ be an irreducible component of $\phi|_{\langle \hat{W} \rangle}^{-1}(\langle W' \rangle)=\phi^{-1}(\langle W' \rangle)\cap \langle \hat{W} \rangle$ that contains $\overline{W}$. Since $W' \subset \phi(\tilde{S}') \subset \langle W' \rangle $ and $\phi(\tilde{S}')$ is special by Lemma \ref{lem:(pre)image}, we must have $\phi(\tilde{S}') = \langle W' \rangle $.

It follows from Lemmas \ref{lem:dimspecial} and \ref{lem:dimconstant}, applied to $\phi|_{\langle \hat{W} \rangle}$, $\langle W' \rangle$, and $\tilde{S}'$, that
\begin{equation}\label{eq:one}
\dim \tilde{S}' - \dim \langle W' \rangle = \dim\langle \hat{W} \rangle - \dim \langle W \rangle.
\end{equation}

Let $\tilde{W}'$ be an irreducible component of $\tilde{S}' \cap \tilde{W}$ that contains $\overline{W}$. Since $W'$ is an irreducible component of $\langle W' \rangle \cap W$, $\phi(\overline{W}) =W'$, and $\phi(\tilde{W}') \subset \phi(\tilde{S}') \cap \phi(\tilde{W}) = \langle W' \rangle \cap W$, we must then have that $\phi(\tilde{W}') = W'$. We deduce that $\tilde{W}' \subset \phi|_{\tilde{W}}^{-1}(W')$ and so $\tilde{W}' = \overline{W}$ and $\tilde{W}'$ is an irreducible component of $\phi|_{\tilde{W}}^{-1}(W')$ that surjects onto $W'$. It follows from Lemmas \ref{lem:dimweaklyspecial} and \ref{lem:dimconstant}, applied to $\phi|_{\tilde{W}}$, $W'$, and $\tilde{W}'$, that
\[ \dim \tilde{W}' - \dim W' = \dim \tilde{W} - \dim W,\]
and therefore, by \eqref{eq:zero},
\begin{equation}\label{eq:two}
\dim \tilde{W}' - \dim W' = \dim\langle \hat{W} \rangle - \dim \langle W\rangle.
\end{equation}
Combining \eqref{eq:one} and \eqref{eq:two}, we find that
\begin{equation}\label{eq:five}
\dim \tilde{S}' - \dim \tilde{W}' = \dim \langle W' \rangle - \dim W'.
\end{equation}

We have $\hat{W} \cap \tilde{S}' \neq \emptyset$ and therefore $y' \in \psi(\tilde{S}')$. We deduce that $\hat{W} \subset \psi^{-1}(\psi(\tilde{S}'))$ and so $\langle \hat{W} \rangle \subset \psi^{-1}(\psi(\tilde{S}'))$ by Lemma \ref{lem:(pre)image}. Since $\tilde{S}' \subset \langle \hat{W} \rangle$, it follows that we must have $\psi(\tilde{S}') = \psi(\langle \hat{W} \rangle)$. As $\tilde{W}' \subset \tilde{W}$, we deduce that
\[ \dim \psi(\langle \hat{W} \rangle) - \dim \psi(\tilde{W}) \leq \dim \psi(\tilde{S}') - \dim \psi(\tilde{W}'),\]
which by \eqref{eq:three} and \eqref{eq:four} implies that 
\begin{equation}\label{eq:six}
\dim \langle W\rangle - \dim W \leq \dim \psi(\tilde{S}') - \dim \psi(\tilde{W}').
\end{equation}

Thanks to Lemmas \ref{lem:dimspecial} and \ref{lem:dimweaklyspecial}, applied to $\psi|_{\tilde{S}'}$ and $\psi|_{\tilde{W}'}$ respectively, together with the Fiber Dimension Theorem and $\tilde{W}' \subset \tilde{S}'$, it follows that
\[ \dim \tilde{W}' - \dim \psi(\tilde{W}') \leq \dim \tilde{S}' - \dim \psi(\tilde{S}'),\]
and hence
\[ \dim \psi(\tilde{S}') - \dim \psi(\tilde{W}') \leq \dim \tilde{S}' - \dim \tilde{W}'.\]
This, in combination with \eqref{eq:five} and \eqref{eq:six}, yields the theorem.
\end{proof}

\section{Optimality}\label{sec:optimality}

In this section, we introduce (weak) optimality. Like the defect $\delta$, this is not a new concept. In \cite{HP}, weak optimality is called geodesic-optimality while in \cite{Poizat} it is called cd-maximality.

\begin{defn}
Let $V$ be a subvariety of a distinguished variety $X$. A subvariety $W \subset V$ is called \emph{optimal for $V$ in $X$} if $\delta(U) > \delta(W)$ for every subvariety $U$ of $X$ such that $W \subsetneq U \subset V$. It is called \emph{weakly optimal for $V$ in $X$} if it satisfies the same property with $\delta_{\ws}$ in place of $\delta$.
\end{defn}

As in \cite{HP}, we can use the defect condition to prove that optimal subvarieties are also weakly optimal.

\begin{prop}\label{prop:oimplieswo}
Let $V$ be a subvariety of a distinguished variety $X$. If $W \subset V$ is optimal for $V$ in $X$, then it is weakly optimal for $V$ in $X$.
\end{prop}

\begin{proof}
This can be deduced in a purely formal way from Theorem \ref{thm:defectcondition} as in the proof of Proposition 4.5 in \cite{HP}. For the reader's convenience, we reproduce the proof here: Let $U$ be a subvariety of $X$ such that $W \subset U \subset V$, $\delta_{\ws}(U) \leq \delta_{\ws}(W)$, and $U$ is weakly optimal for $V$ in $X$. Then it follows from Theorem \ref{thm:defectcondition}, applied to $W \subset U$, that $\delta(U) - \delta_{\ws}(U) \leq \delta(W) - \delta_{\ws}(W)$ and so
\[ \delta(U) = \delta_{\ws}(U) + \delta(U) - \delta_{\ws}(U) \leq \delta_{\ws}(W) + \delta(W) - \delta_{\ws}(W) = \delta(W).\]
Since $W$ is optimal for $V$ in $X$, we deduce that $U = W$ and so $W$ is weakly optimal for $V$ in $X$.
\end{proof}

Part (1) of the following lemma is a weaker analogue of Lemma 2.6(1) in \cite{BD} (see Lemma \ref{lem:zilberpinkfiniteembedding} for the analogue of Lemma 2.6(2) in \cite{BD}).

\begin{lem}\label{lem:imageisoptimal}
Let $\psi: X \rightarrow X'$ be a distinguished morphism. Let $W \subset V \subset X$ be a sequence of nested subvarieties. Let $V'$ be the closure of $\psi(V)$ and let $W'$ be the closure of $\psi(W)$. Then, the following hold:

\begin{enumerate}
\item Suppose that $V_0 \subset V$ is an open and dense subset of $V$ such that $\psi(V_0)$ is an open and dense subset of $V'$ and $\psi|_{V_0}: V_0 \to \psi(V_0)$ is smooth. If $W$ is optimal for $V$ in $X$, $W$ intersects $V_0$, and $\langle W \rangle$ contains an irreducible component of a fiber of $\psi$, then $W'$ has defect at most $\delta(W)$ and is optimal for $V'$ in $X'$.
\item If $\psi$ has finite fibers, then $\delta(W') = \delta(W)$.
\end{enumerate}
\end{lem}

\begin{proof}
	For (1), let $W$ be optimal for $V$ in $X$ such that $W\cap V_0 \neq \emptyset$, where $\psi|_{V_0}: V_0 \to \psi(V_0)$ is smooth of relative dimension $n$. Let $U'$ be a subvariety of $X'$ such that $W' \subset U' \subset V'$ and $\delta(U') \leq \delta(W')$. Since $\langle W \rangle$ contains an irreducible component of a fiber of $\psi$, it follows from Lemma \ref{lem:dimspecial} that
	\[ \dim \langle W \rangle - \dim \psi(\langle W \rangle) =  \dim X - \dim \psi(X)=:m.\]
	
	Since $\psi(\langle W \rangle)$ contains $\langle W'\rangle$ thanks to Lemma \ref{lem:(pre)image}, we obtain that
	\begin{align*}
	\delta(U')\leq \delta(W') \leq \dim \psi( \langle W \rangle) - \dim W'  = \dim \langle W \rangle - \dim W' - m.
	\end{align*}
	
	 Let now $U$ be an irreducible component of $\psi^{-1}(U') \cap V = \psi|_V^{-1}(U')$ containing $W$. We have $U \cap V_0 \neq \emptyset$ since it contains $W \cap V_0$. Furthermore, $U \cap V_0$ is an irreducible component of $\psi|_{V_0}^{-1}(U' \cap \psi(V_0))$. Since $\psi|_{V_0}: V_0 \to \psi(V_0)$ is smooth of relative dimension $n$, it follows that $\dim U = \dim(U \cap V_0) = \dim(U' \cap \psi(V_0)) + n = \dim U' + n$. Taking $U' = W'$ and using that $W \subset U$, we have that $\dim W \leq \dim W' + n$. If $U'$ is again an arbitrary subvariety of $X'$ as above, we deduce that
	\begin{equation}\label{eq:defect}
	\delta(U')\leq \delta(W')\leq \dim \langle W \rangle - m- \dim W +n.
	\end{equation}
	Since $\psi$ is a distinguished morphism, all of its non-empty fibers are equidimensional of the same dimension, which, by the Fiber Dimension Theorem, must equal $m = \dim X - \dim \psi(X)$. It follows that $n$, which is the relative dimension of the smooth morphism $\psi|_{V_0}$, is at most $m$ and so \eqref{eq:defect} implies that the defect of $W'$ is at most the defect of $W$.
	
	We now want to prove that $W'$ is optimal for $V'$ in $X'$. Assume this is not the case; it follows that we can choose $U'$ as above with $W' \subsetneq U'$. Since $\dim U' > \dim W'$, we have 
	\[\dim U = \dim U' + n > \dim W' + n \geq \dim W.\]
	It follows that $W \subsetneq U$ and so $\delta(U) > \delta(W)$ by the optimality of $W$. Since any irreducible component of $\psi^{-1}(\langle U' \rangle)$ is special by Lemma \ref{lem:(pre)image}, we have $\dim \langle U\rangle \leq \dim \langle U'\rangle + (\dim X - \dim \psi(X))$. It follows that $\delta(U) \leq \delta(U') + (\dim X -\dim \psi(X)) - n$ because $\dim U = \dim U' + n$.
	
	Using \eqref{eq:defect}, we deduce that
\begin{align*}
	\delta(U')\leq \delta(W) - m  +n < \delta(U)- m  +n \leq \delta(U'),
\end{align*}
	which yields a contradiction and finishes the proof of (1). 
	
	For (2), we note that $\dim W' = \dim W$. Furthermore, we have $\langle W' \rangle \subset \psi(\langle W \rangle) \subset \psi(X)$ and $\langle W \rangle \subset \psi^{-1}(\langle W' \rangle)$. It follows that $\dim \langle W' \rangle \leq \dim \psi(\langle W \rangle)$ and
	\[ \dim \psi(\langle W \rangle) = \dim \langle W \rangle \leq \dim \psi^{-1}(\langle W' \rangle) = \dim \langle W' \rangle.\]
	Hence, we have that $\langle W' \rangle = \psi(\langle W \rangle)$ and $\delta(W') = \delta(W)$.
\end{proof}

\section{Weak finiteness and very distinguished categories} \label{sec:A5}

We can now introduce a fifth axiom that a distinguished category over $K$ can satisfy:

\begin{enumerate}[label={(A\arabic*)}]
\setcounter{enumi}{4}
\item \label{ax:5} \emph{Weak Finiteness} - If $X$ is a distinguished variety and $V \subset X$ is a subvariety, then there exists a finite set of pairs $(\phi,\psi)$ of distinguished morphisms $\phi: Y_\phi \to X$ and $\psi: Y_\phi \to Z_\psi$ such that for every subvariety $W \subset V$ that is weakly optimal for $V$ in $X$ there exist a pair $(\phi,\psi)$ in this set and $z \in Z_\psi(K)$ such that $\phi$ has finite fibers and $\langle W \rangle_{\ws}$ is an irreducible component of $\phi(\psi^{-1}(z))$.
\end{enumerate}

Axiom \ref{ax:5} is prominently not satisfied by $\mathfrak{C}_{\mathrm{add}}$ (or $\mathfrak{C}_{\mathrm{comm}}$): Consider $X = \mathbb{G}^3_{a,K}$ with affine coordinates $x,y,z$ and let $V \subset X$ be defined by $x^2 + y^2 = z^2$. Then $V$ is covered by lines, but it is not a plane, so every line it contains is weakly optimal for $V$ in $X$. If \ref{ax:5} were satisfied, then the set of the slopes of all these lines would be finite, but it is not.

Unlike \ref{ax:1} to \ref{ax:4}, Axiom \ref{ax:5} is not preserved when taking products of distinguished categories as defined at the end of Section \ref{sec:dc}: We will see later that $\mathfrak{C}_{\mathrm{ab}}$ over $\bar{\mathbb{Q}}$ satisfies \ref{ax:5}. Let $E$ be an elliptic curve over $\bar{\mathbb{Q}}$ and consider the object $(E^2,E^2)$ of $\mathfrak{C}_{\mathrm{ab}} \times \mathfrak{C}_{\mathrm{ab}}$. It is sent to $E^2 \times_{\bar{\mathbb{Q}}} E^2$ by the functor associated to $\mathfrak{C}_{\mathrm{ab}} \times \mathfrak{C}_{\mathrm{ab}}$. Let $\Delta \subset E^2 \times_{\bar{\mathbb{Q}}} E^2$ denote the diagonal. For every one-dimensional abelian subvariety $A \subset E^2$, we have that $W = (A \times_{\bar{\mathbb{Q}}} A) \cap \Delta$ is weakly optimal for $\Delta$ in $E^2 \times_{\bar{\mathbb{Q}}} E^2$ with $\langle W \rangle_{\ws} =  A \times_{\bar{\mathbb{Q}}} A$ and there are infinitely many such $A$.

Let $K\subset L$ be an extension of algebraically closed fields. Given a distinguished category $\mathfrak{C}$ over $K$ with associated functor $\mathcal{F}$, one can check that \ref{ax:1} to \ref{ax:4} are preserved under base change (see Table 1 on pp.~306--307 of \cite{MR3729254}) and we naturally obtain a distinguished category $\mathfrak{C}_L$ over $L$ by composing $\mathcal{F}$ with the base change functor. Distinguished varieties and morphisms in $\mathfrak{C}_L$ are then base changes of distinguished varieties and morphisms in $\mathfrak{C}$. Note that special subvarieties in $\mathfrak{C}_L$ are base changes of special subvarieties in $\mathfrak{C}$ while the same fact clearly does not hold for weakly special subvarieties.

We can now introduce a subclass of distinguished categories that will be used in the following.

\begin{defn}
Let $K$ be an algebraically closed field. A \emph{very distinguished category} over $K$ is a distinguished category $\mathfrak{C}$ over $K$ such that, for every extension of algebraically closed fields $K \subset L$ of finite transcendence degree, the base change $\mathfrak{C}_L$ satisfies \ref{ax:5}.
\end{defn}

As $\mathfrak{C}_{\mathrm{add}}$ and $\mathfrak{C}_{\mathrm{comm}}$ do not satisfy \ref{ax:5} as noted above, they certainly are not very distinguished. On the other hand, it is trivial that $\mathfrak{C}_{\mathrm{triv}}(X)$ is very distinguished. We do not know if there are distinguished categories that satisfy \ref{ax:5}, but are not very distinguished (the question was raised by Orr). Similarly, we do not know whether weak optimality is preserved under base change in an arbitrary distinguished category.

The categories $\mathfrak{C}_{\mathrm{ab}}$, $\mathfrak{C}_{\mathrm{tor}}$, and $\mathfrak{C}_{\mathrm{semiab}}$ are very distinguished as well: For $\mathfrak{C}_{\mathrm{ab}}$, see \cite{Remond} and \cite{HP}. For $\mathfrak{C}_{\mathrm{tor}}$, see Corollaire 3.7 in \cite{Poizat} (a similar finiteness statement is proven in \cite{BMZ07}; see also Corollary 3 in \cite{Zilber}). These results are over $\mathbb{C}$, but they imply the same statement over an arbitrary algebraically closed field (see the proof of Theorem \ref{thm:kirbyweakfiniteness} below). This also justifies the step in the proof of Proposition 4.1 on p. 276 of \cite{BD} where the results of \cite{Remond} are applied to an abelian variety over an arbitrary algebraically closed field of characteristic $0$.

In Theorem \ref{thm:kirbyweakfiniteness}, we will show that $\mathfrak{C}_{\mathrm{semiab}}$ over an algebraically closed field is very distinguished, using a result due to Kirby. This directly implies the same for $\mathfrak{C}_{\mathrm{ab}}$ and $\mathfrak{C}_{\mathrm{tor}}$. We first prove the following theorem.

\begin{thm}\label{thm:structuresemiab}
Let $X$ be a semiabelian variety over $\mathbb{C}$ and let $V \subset X$ be a subvariety. There is a finite set $\Sigma$ of semiabelian subvarieties of $X$, depending only on $X$ and $V$, such that for every $W \subset V$ that is weakly optimal for $V$ in $X$ there exist $Y \in \Sigma$ and $x \in W(\mathbb{C})$ with $\langle W \rangle_{\ws} = xY$.
\end{thm}

\begin{proof}
We induct on the dimension of $V$, the case $\dim V = 0$ being trivial. Let $W \subset V$ be weakly optimal for $V$ in $X$. We can assume without loss of generality that $W \neq V$. This implies that $W$ is an atypical component of the intersection $\langle W \rangle_{\ws} \cap V$ as defined in D\'efinition 2.1 in \cite{Chambert-Loir}.

By Th\'eor\`eme 2.3 in \cite{Chambert-Loir}, due to Kirby in \cite{Kirby}, there is a finite set $\Phi$ of semiabelian subvarieties of $X$, depending only on $X$ and $V$, such that $W \subset xY$ for some $Y \in \Phi$ and $x \in X(\mathbb{C})$ and furthermore
\[\dim Y + \dim W = \dim \langle W \rangle_{\ws} + \dim(V \cap xY).\]
This is equivalent to
\begin{equation}\label{eq:weakdefectofw}
\delta_{\ws}(W) = \dim Y - \dim(V\cap xY).
\end{equation}
We can assume without loss of generality that $x \in W(\mathbb{C})$.
We want to prove that $\langle W \rangle_{\ws} = xY$.

Consider the projection $\pi: V \to VY/Y \subset X/Y$. We have $\pi^{-1}(\pi(v)) = V \cap vY$ for $v \in V(\mathbb{C})$. It follows from the Fiber Dimension Theorem that there exists a closed subset $V_1 \subsetneq V$, the preimage of the complement of the open subset of $VY/Y$ from the theorem, such that $\pi^{-1}(\pi(V_1)) = V_1$ and $\pi^{-1}(\pi(v))$ is equidimensional for $v \in (V\backslash V_1)(\mathbb{C})$. Note that $W$ is contained in a fiber of $\pi$.

If $W \subset V_1$ (or equivalently: if $x \in V_1(\mathbb{C})$), then $W$ is weakly optimal for a component of $V_1$ in $X$ and we are done by induction. Otherwise, let $\tilde{W}$ be an irreducible component of $V \cap xY$ that contains $W$. As $V \cap xY = \pi^{-1}(\pi(x))$ is equidimensional since $x \not\in V_1(\mathbb{C})$, it follows from \eqref{eq:weakdefectofw} that
\[ \delta_{\ws}(W) = \dim Y - \dim(V\cap xY) = \dim Y - \dim \tilde{W} \geq \delta_{\ws}(\tilde{W}).\]
It follows from the weak optimality of $W$ that $W = \tilde{W}$, so $\dim Y = \dim \langle W \rangle_{\ws}$. Since $\langle W \rangle_{\ws} \subset xY$, we have $\langle W \rangle_{\ws} = xY$.
\end{proof}

\begin{thm}\label{thm:kirbyweakfiniteness}
The distinguished category of semiabelian varieties over an algebraically closed field $K$ is very distinguished.
\end{thm}

\begin{proof}
Let $K \subset L$ be an extension of algebraically closed fields and let $X$ be a semiabelian variety over $K$. There exist an algebraically closed subfield $K_0 \subset K$ and a semiabelian variety $X_0$ over $K_0$ such that $X = (X_0)_{K}$ and $K_0$ embeds into $\mathbb{C}$. We fix one such embedding. Let $V \subset X_L$ be a subvariety. Again, there exist an algebraically closed subfield $L_0 \subset L$, containing $K_0$, and a subvariety $V_0$ of $(X_0)_{L_0}$ such that $V = (V_0)_{L}$ and there is an embedding of $L_0$ into $\mathbb{C}$ that extends the one of $K_0$. We also fix one such embedding. The claim of the theorem now follows from Theorem \ref{thm:structuresemiab}, applied to $(V_0)_{\mathbb{C}}\subset (X_0)_{\mathbb{C}}$, together with the two following facts:

\begin{enumerate}
\item All semiabelian subvarieties of $(X_0)_{\mathbb{C}}$ are base changes of semiabelian subvarieties of $X_0$ since the set of torsion points of a semiabelian subvariety $Y$ of $(X_0)_{\mathbb{C}}$ is dense in $Y$ and this set is defined over $K_0$, and
\item if $W \subset V$ is a subvariety, then there exists an algebraically closed subfield $L_W \subset L$, containing $L_0$, and a subvariety $W_0$ of $(V_0)_{L_W}$ such that $W = (W_0)_{L}$ and there is an embedding of $L_W$ into $\mathbb{C}$ that extends the one of $L_0$. Furthermore, if $(W_0)_{\mathbb{C}} \subset x(Y_0)_{\mathbb{C}}$ for some $x \in X_0(\mathbb{C})$ and some semiabelian subvariety $Y_0$ of $X_0$, then $x$ can be chosen as the base change of an $L_W$-point of $W_0$. In particular, weak optimality is preserved under base change. \qedhere
\end{enumerate}

\end{proof}

We now would like to apply a theorem of Gao in \cite{G18} to see that the distinguished categories of connected pure Shimura varieties and connected mixed Shimura varieties of Kuga type over $\bar{\mathbb{Q}}$ are very distinguished. Unfortunately, the definition of a connected mixed Shimura datum of Kuga type and the definition of a connected pure Shimura datum in \cite{G18} do not coincide with our definitions here (which are the same as in \cite{G17}). In fact, the definitions in \cite{G18} give rise to full subcategories $\mathfrak{C}^0_{\textrm{mSvK}}$ and $\mathfrak{C}^0_{\textrm{pSv}}$ of $\mathfrak{C}_{\textrm{mSvK}}$ and $\mathfrak{C}_{\textrm{pSv}}$ respectively, where we consider only those mixed Shimura data that have generic Mumford-Tate group as defined in Definition 2.5 in \cite{G17}. By analogy, we also get a full subcategory $\mathfrak{C}^0_{\textrm{mSv}}$ of $\mathfrak{C}_{\textrm{mSv}}$. We do not know if these subcategories of distinguished categories are distinguished themselves, but we can still apply Definitions \ref{defn:special} and \ref{defn:weaklyspecial} to define (weakly) special subvarieties in them.

The following two lemmas will solve all arising technical issues. We continue identifying objects of $\mathfrak{C}_{\textrm{mSv}}$ (or of one of its subcategories) with their images under $\mathcal{F}_{\textrm{mSv}}$.

\begin{lem}\label{lem:differentdefinitions}
Let $K$ be an algebraically closed field and let $\mathfrak{C}, \mathfrak{C}'$ be two categories among $(\mathfrak{C}_{\textrm{mSv}})_K$, $(\mathfrak{C}_{\textrm{mSvK}})_K$, $(\mathfrak{C}_{\textrm{pSv}})_K$, $(\mathfrak{C}^0_{\textrm{mSv}})_K$, $(\mathfrak{C}^0_{\textrm{mSvK}})_K$, $(\mathfrak{C}^0_{\textrm{pSv}})_K$. Let $\mathfrak{V}, \mathfrak{V}'$ denote their respective classes of objects and suppose that $\mathfrak{V}'\subset\mathfrak{V}$. Let $X \in \mathfrak{V}'$; then, the (weakly) special subvarieties of $X$ as an object of $\mathfrak{C}'$ are the same as the (weakly) special subvarieties of $X$ as an object of $\mathfrak{C}$.
\end{lem}

\begin{proof}
As every morphism of $\mathfrak{C}'$ is also a morphism of $\mathfrak{C}$, it is clear that all (weakly) special subvarieties of $X$ as an object of $\mathfrak{C}'$ are also (weakly) special subvarieties of $X$ as an object of $\mathfrak{C}$. What remains to be proved is the converse direction. Let $X \in \mathfrak{V}'$, let $\phi: Y \to X, \psi: Y \to Z$ be morphisms of $\mathfrak{C}$, and let $z \in Z(K)$. We want to show that $\phi(Y)$ is a special subvariety of $X$ as an object of $\mathfrak{C}'$ and that every irreducible component $W$ of $\phi(\psi^{-1}(z))$ is a weakly special subvariety of $X$ as an object of $\mathfrak{C}'$.

Suppose first that $\mathfrak{C}$ is among the categories $(\mathfrak{C}_{\textrm{mSv}})_K$, $(\mathfrak{C}_{\textrm{mSvK}})_K$, $(\mathfrak{C}_{\textrm{pSv}})_K$ and $\mathfrak{C}' = \mathfrak{C}^0$ is the corresponding subcategory as defined above. Then, Proposition 2.6(1) in \cite{G17}, Proposition 2.2 in \cite{RohlfsSchwermer}, Lemma 2.2 in \cite{OrrThesis} (which is formulated for reductive groups, but holds with the same proof for arbitrary linear algebraic groups), and the definition of having generic Mumford-Tate group yield (base changes of) surjective Shimura embeddings $\phi_Y: Y^0 \to Y,\psi_Z: Z^0 \to Z$ in $\mathfrak{C}$ and a morphism $\psi^0: Y^0 \to Z^0$ of $\mathfrak{C}'$ such that $Y^0, Z^0 \in \mathfrak{V}'$ and $\psi \circ \phi_Y = \psi_Z \circ\psi^0$. Hence, $\phi \circ \phi_Y$ is a morphism of $\mathfrak{C}'$ and $\phi(Y) = (\phi \circ \phi_Y)(Y^0)$ is a special subvariety of $X$ as an object of $\mathfrak{C}'$. Furthermore, we have
\[\phi(\psi^{-1}(z)) = (\phi \circ \phi_Y)((\psi \circ \phi_Y)^{-1}(z)) = \bigcup_{z^0 \in \psi_Z^{-1}(z)}(\phi \circ\phi_Y)((\psi^0)^{-1}(z^0)).\]
Since $\psi_Z$ is a (base change of a) Shimura embedding, $\psi_Z^{-1}(z)$ is finite by Proposition \ref{prop:pinkfacts} and the desired claim follows.

Suppose now that $\mathfrak{C}, \mathfrak{C}'$ are both among the categories $(\mathfrak{C}_{\textrm{mSv}})_K$, $(\mathfrak{C}_{\textrm{mSvK}})_K$, $(\mathfrak{C}_{\textrm{pSv}})_K$. Note that, by Proposition 2.9 in \cite{G17}, being pure and being of Kuga type are inherited by Shimura subdata. These properties are also inherited by quotient Shimura data by Remark \ref{rmk:imageofu}. The desired claim for special subvarieties then follows from Lemma 5.11 in \cite{G17}.

After replacing $K$ by an algebraically closed subfield of finite transcendence degree over $\bar{\mathbb{Q}}$ such that $z$ is the base change of a point defined over that subfield, we can embed $K$ into $\mathbb{C}$. By Proposition 5.4 in \cite{G17}, we have that $W_{\mathbb{C}}$ is an irreducible component of $\phi'_{\mathbb{C}}((\psi'_{\mathbb{C}})^{-1}(z'))$, where $\phi': Y' \to X$ and $\psi': Y' \to Z'$ are (base changes of) a Shimura embedding and a quotient Shimura morphism respectively (in particular, $\phi',\psi'$ are morphisms of $\mathfrak{C}'$ and $\phi'$ is finite by Proposition \ref{prop:pinkfacts}) and $z' \in Z'(\mathbb{C})$. It follows from the finiteness of $\phi'$ that $\{z'\} = \psi'_{\mathbb{C}}(\tilde{W})$ for some irreducible component $\tilde{W}$ of $(\phi'_{\mathbb{C}})^{-1}(W_{\mathbb{C}})$. We therefore deduce that $z'$ is the base change of a $K$-point of $Z'$ and the desired claim follows.

Combining the two cases considered above now yields the lemma.
\end{proof}

\begin{lem}\label{lem:coverbygenericmumfordtategroup}
Let $K$ be an algebraically closed field, let $\mathfrak{C}$ be one of the categories $(\mathfrak{C}_{\textrm{mSv}})_K$, $(\mathfrak{C}_{\textrm{mSvK}})_K$, $(\mathfrak{C}_{\textrm{pSv}})_K$, and let $\mathfrak{C}^0$ be the corresponding subcategory as defined above. Let $\mathfrak{V}, \mathfrak{V}^0$ denote the respective classes of objects of $\mathfrak{C}, \mathfrak{C}^0$. Let $X \in \mathfrak{V}$. Then, there exists a (base change of a) surjective Shimura embedding $\phi: X^0 \to X$ such that $X^0 \in \mathfrak{V}^0$. If $V$ is a subvariety of $X$ and $W$ is weakly optimal for $V$ in $X$, then there exist irreducible components $V^0, W^0$ of $\phi^{-1}(V), \phi^{-1}(W)$ respectively such that $W^0 \subset V^0$, $W^0$ is weakly optimal for $V^0$ in $X^0$, and $\langle W \rangle_{\ws} = \phi(\langle W^0 \rangle_{\ws})$.
\end{lem}

We remark that, thanks to Lemma \ref{lem:differentdefinitions}, the operator $\langle \cdot \rangle_{\ws}$ associated to $X^0$ as well as the concept of weak optimality for subvarieties of $X^0$ here do not depend on whether we regard $X^0$ as an object of $\mathfrak{C}^0$ or as an object of $\mathfrak{C}$.

\begin{proof}
The existence of $\phi$ follows from Proposition 2.6(1) in \cite{G17}. If $U^0 \subset X^0$ is any subvariety, then Lemma \ref{lem:(pre)image} implies that $\langle \phi(U^0) \rangle_{\ws} \subset \phi(\langle U^0 \rangle_{\ws})$ as well as $\langle U^0 \rangle_{\ws} \subset\phi^{-1}(\langle \phi(U^0) \rangle_{\ws})$. Since $\phi$ is surjective, we deduce that $\langle \phi(U^0) \rangle_{\ws} = \phi(\langle U^0 \rangle_{\ws})$, and since $\phi$ is also finite by Proposition \ref{prop:pinkfacts}, this implies that $\delta_{\ws}(\phi(U^0)) = \delta_{\ws}(U^0)$.

Using again the surjectivity and the finiteness of $\phi$, we now choose an irreducible component $V^0$ of $\phi^{-1}(V)$ such that $\phi(V^0) = V$ and an irreducible component $W^0$ of $\phi|_{V^0}^{-1}(W)$ such that $\phi(W^0) = W$. Since $\phi$ is finite, $W^0$ must be an irreducible component of $\phi^{-1}(W)$ as well. Together with the finiteness of $\phi$, the previous paragraph implies that $W^0$ is weakly optimal for $V^0$ in $X^0$ and that $\langle W \rangle_{\ws} = \langle \phi(W^0) \rangle_{\ws} = \phi(\langle W^0 \rangle_{\ws})$.
\end{proof}

We can now prove the following theorem, which is essentially due to Gao:

\begin{thm}[Gao]\label{thm:gaoweakfiniteness}
The distinguished categories of connected pure Shimura varieties and connected mixed Shimura varieties of Kuga type over $\bar{\mathbb{Q}}$ are very distinguished.
\end{thm}

For pure Shimura varieties, this was essentially already established in \cite{DawRen}, later refined in \cite{BinyaminiDaw}. With the new results obtained independently by Chiu \cite{Chiu} and Gao-Klingler \cite{GK}, we expect that the proof of Theorem \ref{thm:gaoweakfiniteness} can be extended to show the same for the distinguished category of connected mixed Shimura varieties.

\begin{proof}
Note that, by Proposition 2.9 in \cite{G17}, being pure and being of Kuga type are inherited by Shimura subdata. These properties are also inherited by quotient Shimura data by Remark \ref{rmk:imageofu}. The required finiteness statement over $\mathbb{C}$ then follows from Theorem 8.2 in \cite{G18} together with Remark 5.3 and Lemma 5.6 in \cite{G17}, as well as Lemma \ref{lem:coverbygenericmumfordtategroup}.

We can go from $\mathbb{C}$ to an arbitrary algebraically closed field of finite transcendence degree over $\bar{\mathbb{Q}}$ thanks to the following: Let $X$ be a connected mixed Shimura variety. Let $U \subset W \subset X_{\mathbb{C}}$ be a chain of subvarieties, where $W$ is weakly special. By Proposition 5.4 in \cite{G17}, we have that $W$ is an irreducible component of $\phi_{\mathbb{C}}(\psi_{\mathbb{C}}^{-1}(z))$, where $\phi: Y \to X$ and $\psi: Y \to Z$ are a Shimura embedding and a quotient Shimura morphism respectively (in particular, $\phi$ is finite by Proposition \ref{prop:pinkfacts}) and $z \in Z(\mathbb{C})$. It follows from the finiteness of $\phi$ that $\{z\} = \psi_{\mathbb{C}}(\tilde{U})$ for some irreducible component $\tilde{U}$ of $\phi_{\mathbb{C}}^{-1}(U)$. If $U$ can be defined over some algebraically closed subfield of $\mathbb{C}$, the same holds for $z$. This implies that weak optimality is preserved under base change. The theorem then also follows since Theorem 8.2 in \cite{G18} yields precisely such pairs of Shimura embeddings and quotient Shimura morphisms.
\end{proof}

\section{Zilber-Pink and extensions of algebraically closed fields}\label{sec:zpext}

We now introduce a notation that allows us to easily express the fact that the analogue of the Zilber-Pink conjecture holds in a certain setting. It generalizes a notation from \cite{BD}.

\begin{defn}\label{defn:zp}
Let $K$ be an algebraically closed field and let $\mathfrak{C}$ be a distinguished category over $K$. Let $m, d$ be non-negative integers and let $X$ be a distinguished variety. We say that $\ZP(X,m,d)$ holds if every subvariety of $X$ of dimension at most $m$ contains at most finitely many optimal subvarieties of defect at most $d$.
\end{defn}

The distinguished category $\mathfrak{C}_{\mathrm{triv}}(Y)$ (for a fixed variety $Y$ over $K$) trivially satisfies $\ZP(X,m,d)$ for all non-negative integers $m$ and $d$ and all distinguished varieties $X$ since every distinguished variety in $\mathfrak{C}_{\mathrm{triv}}(Y)$ has only finitely many special subvarieties and an optimal subvariety for a subvariety $V$ in $X$ must be an irreducible component of the intersection of $V$ and some special subvariety of $X$.

We collect some auxiliary results that are going to be useful later.
 
 \begin{lem}[Lemma 2.2 in \cite{BD}]\label{lem:ausefullemma}
 	Let $K \subset L$ be an extension of algebraically closed fields such that $L$ has transcendence degree $1$ over $K$. Let $V$ be a variety over $K$ and let $W$ be a subvariety of $V_L$. Then, there exists a subvariety $W'$ of $V$ with $\dim W' \leq \dim W + 1$ such that $W \subset W'_L$.
 \end{lem}

\begin{lem}[Lemma 2.5 in \cite{BD}]\label{lem:smooth}
Let $K$ be an algebraically closed field and let $f: V \to W$ be a dominant morphism of algebraic varieties, defined over $K$. Then, there exists $V_0 \subset V$ open and dense such that $f(V_0)$ is open and dense in $W$ and $f|_{V_0}: V_0 \to f(V_0)$ is smooth.
\end{lem}

The following lemma is the analogue of Lemma 2.11 in \cite{BD} in our setting.

\begin{lem}\label{lem:basechange}
	Let $K \subset L$ be an extension of algebraically closed fields. Let $\mathfrak{C}$ be a distinguished category over $K$, let $X$ be a distinguished variety, and let $V$ be a subvariety of $X$. If $W$ is an optimal subvariety for $V_L$ in $X_L$, then there exists an optimal subvariety $W'$ for $V$ in $X$ such that $W = (W')_L$ and $\delta(W) = \delta(W')$.
\end{lem}

\begin{proof}
	Since a special subvariety is the image of a distinguished morphism, we have that any special subvariety of $X_L$ is the base change of a special subvariety of $X$. Therefore, if $V$ is a subvariety of $X$, any optimal subvariety for $V_L$ in $X_L$ is an irreducible component of an intersection $V_L\cap S_L$ for some special subvariety $S$ of $X$ and is then the base change of a subvariety $W'\subset V$ that must be optimal for $V$ in $X$ and of the same defect.
\end{proof}

The following lemma is an analogue of Lemma 2.6(2) in \cite{BD}.

\begin{lem}\label{lem:zilberpinkfiniteembedding}
Let $m$ and $d$ be non-negative integers. Let $K$ be an algebraically closed field and let $\mathfrak{C}$ be a distinguished category over $K$. Let $\phi: X \to Y$ be a distinguished morphism with finite fibers. Then $\ZP(Y,m,d) \Rightarrow \ZP(X,m,d)$.
\end{lem}

\begin{proof}
We induct on $m$, the case $m = 0$ being trivial. Let $V \subset X$ be a subvariety of dimension $m$ and let $V'$ be the closure of $\phi(V)$. By Lemma \ref{lem:smooth} and the Fiber Dimension Theorem, we can find $V_0 \subset V$ open and dense such that $\phi(V_0)$ is open and dense in $V'$ and $\phi|_{V_0}: V_0 \to \phi(V_0)$ is smooth of relative dimension $\dim V_0 - \dim \phi(V_0)$, which must be 0 as $\phi$ has finite fibers.

Let now $W \subset V$ be an optimal subvariety for $V$ in $X$ whose defect is at most $d$. If $W \cap V_0 = \emptyset$, then $W$ is optimal for a component of $V \backslash V_0$ and we are done by the inductive hypothesis. Otherwise, we can apply Lemma \ref{lem:imageisoptimal}(1) to deduce that the closure $W'$ of $\phi(W)$ is optimal for $V'$ in $Y$ and the defect of $W'$ is at most $d$. It follows from $\ZP(Y,m,d)$ that there are at most finitely many possibilities for $W'$ and hence for $W$, which is an irreducible component of $\phi^{-1}(W')$.
\end{proof}

If the distinguished morphism in Lemma \ref{lem:zilberpinkfiniteembedding} is also assumed to be surjective, then one can obtain an equivalence instead of just an implication.

\begin{lem}\label{lem:zilberpinkfinitecover}
Let $m$ and $d$ be non-negative integers. Let $K$ be an algebraically closed field and let $\mathfrak{C}$ be a distinguished category over $K$. Let $\phi: X \to Y$ be a surjective distinguished morphism with finite fibers. Then $\ZP(X,m,d) \Leftrightarrow \ZP(Y,m,d)$.
\end{lem}

\begin{proof}
Thanks to Lemma \ref{lem:zilberpinkfiniteembedding}, it suffices to prove that $\ZP(X,m,d)$ implies $\ZP(Y,m,d)$.

We induct on $m$, the case $m = 0$ being trivial. Let $V \subset Y$ be a subvariety of dimension $m$. Let $\tilde{V}$ be an irreducible component of $\phi^{-1}(V)$ that dominates $V$; we have $\dim \tilde{V} = \dim V \leq m$. By Lemma \ref{lem:smooth} and the Fiber Dimension Theorem, we can find $\tilde{V}_0 \subset \tilde{V}$ open and dense such that $\phi(\tilde{V}_0)$ is open and dense in $V$ and $\phi|_{\tilde{V}_0}: \tilde{V}_0 \to \phi(\tilde{V}_0)$ is smooth of relative dimension $\dim \tilde{V}_0 - \dim \phi(\tilde{V}_0) = 0$.

Let now $W \subset V$ be an optimal subvariety for $V$ in $Y$ whose defect is at most $d$. If $W \cap \phi(\tilde{V}_0) = \emptyset$, then $W$ is optimal for a component of $V \backslash\phi(\tilde{V}_0)$ and we are done by induction on $m$. Otherwise $W \cap \phi(\tilde{V}_0)$ is dense in $W$. Let $\tilde{W}$ be the closure in $\tilde{V}$ of a component of $\phi|_{\tilde{V}_0}^{-1}(W)$ that dominates $W$. By Lemma \ref{lem:imageisoptimal}(2), we have $\delta(\tilde{W}) = \delta(W) \leq d$.

Let $\tilde{U}$ be a subvariety of $\tilde{V}$ such that $\tilde{W} \subset \tilde{U} \subset \tilde{V}$, $\tilde{U}$ is optimal for $\tilde{V}$ in $X$, and $\delta(\tilde{U}) \leq \delta(\tilde{W})$. We have $\tilde{U} \cap \tilde{V}_0 \neq \emptyset$ since $\tilde{W} \cap \tilde{V}_0 \neq \emptyset$. By Lemma \ref{lem:imageisoptimal}, the closure $U$ of $\phi(\tilde{U})$ is optimal for $V$ in $Y$ of defect
\[\delta(U) = \delta(\tilde{U})\leq \delta(\tilde{W}) = \delta(W).\]
Furthermore, we have $W \subset U \subset V$. Since $W$ is optimal for $V$ in $Y$, it follows that $U = W$, hence $\tilde{W} = \tilde{U}$ is optimal for $\tilde{V}$ in $X$ and its defect is at most $d$. By $\ZP(X,m,d)$, there are at most finitely many possibilities for $\tilde{W}$ and hence also for $W$.
\end{proof}

The following theorem, which is one of our main results, is an analogue of Proposition 4.1 in \cite{BD}. In its statement, we use the notation from Definition \ref{defn:zp}.

\begin{thm}\label{thm:fieldofdef}
	Let $m$ and $d$ be non-negative integers. Let $K\subset L$ be an extension of algebraically closed fields and let $\mathfrak{C}$ be a very distinguished category over $K$. Let $X$ be a distinguished variety. Let $\mathcal{S}$ be a class of distinguished varieties such that $X \in \mathcal{S}$ and for every $X' \in \mathcal{S}$, every field extension $K \subset K'$ of finite transcendence degree, and every subvariety $V$ of $X'_{K'}$, the set in \ref{ax:5} can be chosen with all $Y_\phi,Z_\psi$ equal to base changes of distinguished varieties belonging to $\mathcal{S}$. Then, the following hold:
	\begin{enumerate}
	\item Suppose that $\ZP (X',m,d)$ holds for all $X' \in \mathcal{S}$. Then $\ZP (X_L,m,d)$ holds as well.
	\item Suppose that $\ZP (X',m-1,d)$ and $\ZP(X',m,d-1)$ hold for all $X' \in \mathcal{S}$. Let $V$ be a subvariety of $X_L$ of dimension at most $m$ which is not the base change of a subvariety of $X$ (in particular $K \neq L$). Then, $V$ contains at most finitely many optimal subvarieties of defect at most $d$.
	\end{enumerate}
\end{thm}

It is of course possible to take the whole class of objects of $\mathfrak{C}$ as $\mathcal{S}$, but for applications, it can be useful to take $\mathcal{S}$ smaller as we will see later.

\begin{proof}	
	We prove (1) and (2) in parallel: Let $V$ be a subvariety of $X_L$. We can find an algebraically closed subfield $L_1$ of $L$ that has finite transcendence degree over $K$ and a subvariety $V_1$ of $X_{L_1}$ such that $V = (V_1)_{L}$. If $W$ is any optimal subvariety for $V$ in $X_L$, then by Lemma \ref{lem:basechange} it is equal to $(W_1)_L$ for an optimal subvariety $W_1$ for $V_1$ in $X_{L_1}$ such that $\delta(W_1) = \delta(W)$. Hence, it suffices to prove the theorem under the assumption that $L$ has finite transcendence degree over $K$.

	We will actually prove the stronger statement that the conclusions of (1) and (2) hold with any $X' \in \mathcal{S}$ in place of $X$. Arguing by induction on the transcendence degree of $L$ over $K$, one can see that it is enough to prove (1) and (2) when $L$ has transcendence degree 1 over $K$.
	
	We proceed by induction on $m$. Clearly, $\ZP(X_L,0,d)$ holds for all $d$. Let $V$ be a subvariety of $X_L$ of dimension $m > 0$. We will deduce that $V$ contains at most finitely many optimal subvarieties of defect at most $d$ from either of the following hypotheses:
	\begin{enumerate}
	\item $\ZP (X'_L, m-1,d)$ and $\ZP (X',m,d)$ hold for all $X' \in \mathcal{S}$, or
	\item $\ZP (X'_L,m-1,d)$ and $\ZP(X',m,d-1)$ hold for all $X' \in \mathcal{S}$ and $V$ is not the base change of a subvariety of $X$.
	\end{enumerate}
	
	As all $X' \in \mathcal{S}$ satisfy the same hypotheses as $X$, it will follow that the conclusions of (1) and (2) hold for all $X' \in \mathcal{S}$ as desired.
	
	\medskip
	
	 If $V=V'_L$ for some $V'\subset X$, then we must be in case (1) and we are done by Lemma \ref{lem:basechange} and $\ZP(X,m,d)$. We will then assume that this is not the case.
	
	Let $V'_L$ be the smallest subvariety of $X_L$ that is the base change of some $V'\subset X$ and contains $V$. It exists and has dimension $m$ or $m+1$ by Lemma \ref{lem:ausefullemma} but the first case is not possible because it would imply that $V=V'_L$.
	
	Let $W \subset V$ be an optimal subvariety for $V$ in $X_L$ that has defect at most $d$. We can assume without loss of generality that $W \neq V$.
	
	We let $W'_L$ be the smallest subvariety of $X_L$ that is the base change of some $W'\subset X$ and contains $W$. By Lemma \ref{lem:ausefullemma}, we have either $W = W'_L$ or $\dim W'_L = \dim W + 1$.
	
	If $W = W_L'$, then $W$ is contained in $Z'_L \subset V$ for $Z'\subset X$ maximal among all finite unions of subvarieties $Z'' \subset X$ with $Z''_L \subset V$ (and equal to the closure of the union of all such $Z''$). Since $V\neq V'_L$, the dimension of $Z'_L$ is at most $\dim V-1$. Of course, $W$ is also optimal for the component of $Z'_L$ that contains it. It follows that $W$ lies in a finite set because $\ZP(X_L, m-1,d)$ holds. We can therefore assume that $W \subsetneq W'_L$ and so $\dim W'_L = \dim W + 1$.
	
	
	Recall that, by Lemma \ref{lem:basechange}, an optimal subvariety for $V'_L$ in $X_L$ is the base change of an optimal subvariety for $V'$ in $X$. Let $U'_L$ be such an optimal subvariety for $V'_L$ in $X_L$ that contains $W'_L$ and satisfies $\delta(U'_L) \leq \delta(W'_L)$. Note that $\langle V \rangle = \langle V'_L \rangle$ and $\langle W \rangle = \langle W'_L\rangle$ because, for instance, $V'_L\subset\langle V \rangle \cap V'_L $ by definition. It follows that $\delta(W'_L) = \delta(W)-1$, so $\delta(U'_L) \leq d-1$.
	
	We claim that $U'_L \neq V'_L$. If not, we could deduce that $\delta(V'_L) = \delta(U'_L) \leq \delta(W'_L)$. It would then follow that
	\[ \delta(V) = \delta(V'_L)+1 \leq \delta(W'_L)+1 = \delta(W),\]
	which contradicts the optimality of $W \subsetneq V$ for $V$. 
	
	We deduce that $U'_L \subsetneq V'_L$ and hence $U'_L \cap V \subsetneq V$, otherwise $U'_L\supset V$ would contradict the minimality of $V'_L$. Since $W \subset U'_L \cap V$ and $W$ has defect at most $d$ and is optimal for a component of $U'_L \cap V$ in $X_L$, it suffices to show that $U'$, and therefore $U'_L$, belongs to a finite set. Then we are done as any component of $U'_L \cap V$ that contains $W$ has dimension at most $m-1$ and $\ZP(X_L,m-1,d)$ holds.
				
	By the optimality of $U'$ for $V'$ in $X$ and Proposition \ref{prop:oimplieswo}, we have that $U'$ is also weakly optimal for $V'$ in $X$. It follows from \ref{ax:5} that there exists a finite set of pairs $(\phi,\psi)$ of distinguished morphisms $\phi: Y_\phi \to X$ and $\psi: Y_\phi \to Z_\psi$, depending only on $V'$, such that there exists a pair $(\phi,\psi)$ in this set and $z \in Z_\psi(K)$ such that $\phi$ has finite fibers and $\langle U' \rangle_{\ws}$ is an irreducible component of $\phi(\psi^{-1}(z))$. By our hypothesis on $\mathcal{S}$, we can assume that $Y_\phi,Z_\psi \in \mathcal{S}$. Since $\phi$ and $\psi$ vary in a finite set, we can assume them fixed.
		
	We are now going to use $\phi$ and $\psi$ to move the problem to the two distinguished varieties $Y_\phi$ and $Z_\psi$, where we will apply $\ZP(Y_\phi,m,d-1)$ and $\ZP(Z_\psi,m,d-1)$.
	
	Let $\hat{U}$ be an irreducible component of $\psi^{-1}(z) \cap \phi^{-1}(U')$ such that $\phi(\hat{U})$ is dense in $U'$. Since $\phi$ has finite fibers, we deduce that $\dim \hat{U} = \dim U'$ and therefore $\hat{U}$ must be an irreducible component of $\phi^{-1}(U')$. Let $\hat{V}$ be an irreducible component of $\phi^{-1}(V')$ that contains $\hat{U}$. As $\phi$ has finite fibers, it follows that $\dim \hat{V} \leq \dim V' \leq m+1$.
	
	Furthermore, $\hat{U}$ is optimal for $\hat{V}$: Otherwise, we could find $\tilde{U}$ such that $\hat{U} \subsetneq \tilde{U} \subset \hat{V}$ and $\delta(\tilde{U}) \leq \delta(\hat{U})$. Applying $\phi$ and taking closures yields a contradiction to the optimality of $U'$ for $V'$ thanks to Lemma \ref{lem:imageisoptimal}(2). The same lemma shows that $\delta(\hat{U}) = \delta(U')$.
	
	We have $\hat{U} \subset \psi^{-1}(z) \cap \hat{V}$ by construction. We want to show that $\hat{U}$ is an irreducible component of $\psi^{-1}(z) \cap \hat{V}$. If this were not the case, then we could find a subvariety $\tilde{U}$ such that $\hat{U} \subsetneq \tilde{U} \subset \psi^{-1}(z) \cap \hat{V}$. Since $\phi$ has finite fibers, it would follow that the closure of $\phi(\tilde{U})$ strictly contains $U'$, while still being contained in $\phi(\psi^{-1}(z))$ as well as in $V'$.
	
	But, since $\phi$ has finite fibers, \ref{ax:4} implies that every irreducible component of $\phi(\psi^{-1}(z))$ has the same dimension. Hence, the above contradicts $U'$ being weakly optimal for $V'$. It follows that $\hat{U}$ is an irreducible component of $\psi^{-1}(z) \cap \hat{V}$.

	We let $V''$ denote the closure of $\psi(\hat{V})$ in $Z_\psi$; this is a subvariety of $Z_\psi$. We know that $\{z\} = \psi(\hat{U})$.
	
	By Lemma \ref{lem:smooth} and the Fiber Dimension Theorem, we can find $\hat{V}_0 \subset \hat{V}$ open and dense such that $\psi(\hat{V}_0)$ is open and dense in $V''$ and $\psi|_{\hat{V}_0}: \hat{V}_0 \to \psi(\hat{V}_0)$ is smooth of relative dimension
	\[ n = \dim \hat{V}_0 - \dim \psi(\hat{V}_0) = \dim \hat{V} - \dim V''.\]
	
	We distinguish two cases: First, if $\hat{U}$ is contained in $\hat{V} \backslash \hat{V}_0$, then $\hat{U}$ is contained in one of finitely many subvarieties of $\hat{V}$ of dimension at most $m$. As $\hat{U}$ is optimal for $\hat{V}$, it is also optimal for that subvariety. We have
	\[ \delta(\hat{U}) = \delta(U') \leq \delta(W') = \delta(W)-1 \leq d-1.\]
	Since $Y_\phi \in \mathcal{S}$, we have that $\ZP(Y_\phi,m,d-1)$ holds. It follows that there are then only finitely many possibilities for $\hat{U}$ and hence for $U'$.
	
	Let us now assume that $\hat{U} \cap \hat{V}_0 \neq \emptyset$. Since $\hat{U}$ is an irreducible component of $\psi^{-1}(z) \cap \hat{V} = \psi|_{\hat{V}}^{-1}(z)$, we then have that $\hat{U} \cap \hat{V}_0$ is an irreducible component of $\psi|_{\hat{V}_0}^{-1}(z)$. It follows that $n = \dim(\hat{U} \cap \hat{V}_0) = \dim \hat{U}$.
	
	Since we have assumed that $W\neq W'_L$, we have $\dim W' > 0$ and hence $n = \dim \hat{U} = \dim U' > 0$.
	
	The special subvariety $\langle \hat{U} \rangle$ contains $\langle \hat{U} \rangle_{\ws}$. By construction, $\langle \hat{U} \rangle_{\ws}$ is contained in an irreducible component of $\psi^{-1}(z)$. As $\langle U' \rangle_{\ws}$ is equal to an irreducible component of $\phi(\psi^{-1}(z))$, $\langle \hat{U} \rangle_{\ws}$ must actually be equal to a component of $\psi^{-1}(z)$: Otherwise, it would follow for dimensional reasons that $U' \subset \phi(\langle \hat{U} \rangle_{\ws}) \cap \langle U' \rangle_{\ws} \subsetneq \langle U' \rangle_{\ws}$ since $\phi$ has finite fibers and all components of $\psi^{-1}(z)$ have the same dimension, a contradiction with Lemmas \ref{lem:(pre)image} and \ref{lem:intersection}.
	
	Since $\hat{U}$ is optimal for $\hat{V}$ and has defect at most $d-1$, it then follows from Lemma \ref{lem:imageisoptimal}(1) that $\{z\} = \psi(\hat{U})$ is optimal for $V''$ in $Z_\psi$ and has defect at most $d-1$. Furthermore, $V''$ is a subvariety of $Z_\psi$ of dimension
	\[ \dim \hat{V} - n \leq \dim \hat{V} - 1 \leq \dim V' -1 = m.\]
	
	As $Z_\psi$ lies in $\mathcal{S}$, we have that $\ZP(Z_\psi,m,d-1)$ holds, thus $z$ lies in a finite set. As $\hat{U}$ is a component of $\psi^{-1}(z) \cap \hat{V}$, it lies in a finite set as well and therefore the same holds for $U'$.
	\end{proof}
	
	We will now apply Theorem \ref{thm:fieldofdef}(2) to certain connected mixed Shimura varieties of Kuga type. For this, we first prove a lemma:
	
	\begin{lem}\label{lem:aoquotient}
	Let $K$ be an algebraically closed field and let $\mathfrak{C}$ be a distinguished category over $K$. Let $X$ be a distinguished variety and let $\phi: X \to Y$ be a surjective distinguished morphism. If $\ZP(X,m,0)$ holds for all non-negative integers $m$, then the same holds for $Y$.
	\end{lem}
	
	\begin{proof}
	Let $V \subset Y$ be a subvariety. If $S \subset V$ is optimal for $V$ in $Y$ and $\delta(S) = 0$, then $S$ is a maximal special subvariety of $V$. The lemma now follows from the fact that any irreducible component of $\phi^{-1}(S)$ is a maximal special subvariety of some irreducible component of $\phi^{-1}(V)$.
	\end{proof}
	
	In the following corollary, $\mathbb{H}_g$ denotes the Siegel upper half space of dimension $\frac{g(g+1)}{2}$ and $(\GSp_{2g},\mathbb{H}_g)$ denotes the connected pure Shimura datum associated to the moduli space of principally polarized abelian varieties of dimension $g$.
		
	\begin{cor}\label{cor:uncondzp}
	Let $X = \mathcal{F}_{\mathrm{mSvK}}(P,X^+,\Gamma)$ be a connected mixed Shimura variety of Kuga type of dimension $3$. Let $W$ be the unipotent radical of $P$ and let $(P,X^+)/W$ denote the quotient Shimura datum (see Remark \ref{rmk:quotientshimuradatum}). Suppose that $(P,X^+)/W$ is a Shimura subdatum of $(\GSp_{2g},\mathbb{H}_g)$ for some $g \in \mathbb{N}$. Let $\bar{\mathbb{Q}} \subset K$ be an extension of algebraically closed fields and let $C \subset X_K$ be a curve that is not the base change of a curve in $X$. Then $C$ contains at most finitely many optimal subvarieties.
	\end{cor}
	
	Corollary \ref{cor:uncondzp} applies for example to the moduli space of principally polarized abelian surfaces or to the (direct) product of the Legendre family of elliptic curves (see Section \ref{sec:legendre}) with the moduli space of elliptic curves. It also applies to the cube of the moduli space of elliptic curves, for which the corresponding statement has been proven by Pila in Theorem 1.4 in \cite{PilaFermat}.
	
	After the completion of this manuscript, Pila, Shankar, and Tsimerman announced a proof of the Andr\'e-Oort conjecture in full generality in \cite{Andre_Oort_Preprint}. As a consequence of their work, the hypothesis on $(P,X^+)/W$ in Corollary \ref{cor:uncondzp} can be removed.

	\begin{proof}
	Since $X$ has dimension $3$, we have that $\delta(C) \leq 2$ and any proper optimal subvariety of $C$ has defect at most $1$. By Theorem \ref{thm:gaoweakfiniteness}, the category of connected mixed Shimura varieties of Kuga type is very distinguished.
	
	We want to apply Theorem \ref{thm:fieldofdef}(2) with $m = d = 1$, choosing the class $\mathcal{S}$ there as small as possible. Note that $\ZP(X',0,1)$ is trivially satisfied for every $X' \in \mathcal{S}$ while $\ZP(X',1,0)$ is the Andr\'e-Oort conjecture for curves in $X'$. We have that $\ZP(X,m,0)$ holds for all non-negative integers $m$, i.e., that the Andr\'e-Oort conjecture holds for $X$, thanks to Theorem 5.2 in \cite{Tsimerman} and Theorem 13.6 in \cite{G17}.
	
	Theorem 8.2 in \cite{G18} and the proof of Theorem \ref{thm:gaoweakfiniteness} show that we first have to include in $\mathcal{S}$ all triples $((Q,Y^+)/N,\ast)$, where $(Q,Y^+)$ is a Shimura subdatum of $(P,X^+)$, $N$ is a normal algebraic subgroup of $Q$ whose reductive part is semisimple, and $(Q,Y^+)/N$ denotes the quotient Shimura datum (see Remark \ref{rmk:quotientshimuradatum}). This step must then be iterated for every connected mixed Shimura datum thus obtained.
	
	By Lemma \ref{lem:zilberpinkfiniteembedding}, $\ZP(\ast,m,0)$ is preserved when passing to a Shimura subdatum (by Lemma \ref{lem:zilberpinkfinitecover}, the validity of $\ZP(\ast,m,d)$ is independent from the choice of congruence subgroup). If we have $\ZP(\ast,m,0)$ for all non-negative integers $m$, this is also preserved when passing to a quotient Shimura datum by Lemma \ref{lem:aoquotient}. Hence, we have $\ZP(X',1,0)$ for all $X' \in \mathcal{S}$ as desired.
	\end{proof}

\section{Reduction to optimal singletons}\label{sec:redoptsing}

The following theorem is an analogue of Theorem 8.3 in \cite{DawRen} and Theorem 6.1 in \cite{BD} (see also \cite{HP}) and is similar to what was recently done by Cassani in \cite{Cass3} for connected mixed Shimura varieties of Kuga type. The reduction to optimal singletons (i.e., optimal subvarieties of dimension $0$) is useful if one wants to formulate and apply ``large Galois orbit" conjectures and o-minimal point counting, cf. Sections 8 to 10 in \cite{HP}.

\begin{thm}\label{thm:redoptsing}
 	Let $m$ and $d$ be non-negative integers. Let $K$ be an algebraically closed field and let $\mathfrak{C}$ be a distinguished category over $K$ that satisfies \ref{ax:5}. Let $X$ be a distinguished variety and let $\mathcal{S}$ be a class of distinguished varieties such that $X \in \mathcal{S}$ and for every $X' \in \mathcal{S}$ and every subvariety of $X'$ the set in \ref{ax:5} can be chosen with all $Y_\phi,Z_\psi \in \mathcal{S}$. Suppose that, for all $X'\in \mathcal{S}$, every subvariety of $X'$ of dimension at most $m$ contains at most finitely many optimal singletons of defect at most $d$. Then $\ZP (X,m,d)$ holds.
\end{thm}

\begin{proof}
	We proceed by induction on $m$. Clearly $\ZP (X,0,d)$ holds. We are going to deduce $\ZP (X,m,d)$ from $\ZP (X,m-1,d)$ and from the fact that all subvarieties of dimension at most $m$ of any $X'\in \mathcal{S}$ contain at most finitely many optimal singletons of defect at most $d$.
	
	We fix a subvariety $V$ of $X$ with $\dim V\leq m$ and a subvariety $W\subset V$ optimal for $V$ in $X$ with $\delta(W)\leq d$.
	
	By Proposition \ref{prop:oimplieswo}, $W$ is weakly optimal for $V$ in $X$ and \ref{ax:5} implies that there is a finite set of pairs $(\phi, \psi)$ of distinguished morphisms $\phi: Y_\phi \to X$ and $\psi: Y_\phi \to Z_\psi$, depending only on $V$, such that there exist a pair $(\phi,\psi)$ in this set and $z \in Z_\psi(K)$ such that $\phi$ has finite fibers and $\langle W \rangle_{\ws}$ is an irreducible component of $\phi(\psi^{-1}(z))$. By our hypothesis, we can assume that $Y_\phi,Z_\psi \in \mathcal{S}$.
	As the pair $(\phi,\psi)$ varies in a finite set, we can assume it is fixed.
		
	We now set $\hat W$ to be an irreducible component of $\psi^{-1}(z) \cap \phi^{-1}(W)$ such that $\phi(\hat{W})$ is dense in $W$ and $\hat V$ to be an irreducible component of $\phi^{-1}(V)$ that contains $\hat{W}$. We recall that $\phi$ has finite fibers and thus $\dim \hat{W} = \dim W$ and $\dim \hat{V} \leq \dim V $. Moreover, Lemma \ref{lem:imageisoptimal}(2) implies that $\delta( \hat{W}) = \delta(W) \leq d$.
	
	We claim that $\hat W$ is optimal for $\hat V$ in $Y_\phi$. If this were not the case, there would exist a $\hat U \supsetneq \hat W$ optimal for $\hat V$ and with $\delta (\hat U)\leq \delta (\hat W)$. Using Lemma \ref{lem:imageisoptimal}(2), we see that this would contradict the optimality of $W$ for $V$.

	By definition we have that $\hat W$ is contained in an irreducible component of $\psi^{-1}(z) \cap\hat V$. If this containment were strict, we could find a subvariety $\hat U$ with $\hat W\subsetneq \hat U \subset \psi^{-1}(z) \cap\hat V$ but then, as $\phi$ has finite fibers, $W$ would be strictly contained in the Zariski closure of $\phi(\hat U)\subset \phi(\psi^{-1}(z)) \cap V$. This contradicts the weak optimality of $W$ because every irreducible component of $\phi(\psi^{-1}(z))$ has the same dimension.
	Thus $\hat W$ has to be an irreducible component of $\psi^{-1}(z) \cap\hat V$.
	
	We now set $\tilde V\subset Z_\psi$ to be the Zariski closure of $\psi(\hat{V})$. Moreover, recall that $\psi(\hat{W})= \{z\}  $. By Lemma \ref{lem:smooth} and the Fiber Dimension Theorem, we can find $\hat{V}_0 \subset \hat{V}$ open and dense such that $\psi(\hat{V}_0)$ is open and dense in $\tilde V$ and $\psi|_{\hat{V}_0}: \hat{V}_0 \to \psi(\hat{V}_0)$ is smooth of relative dimension $n = \dim \hat{V}_0 - \dim \psi(\hat{V}_0) = \dim \hat{V} - \dim \tilde V$. Note that both $\hat{V}$ and $\tilde{V}$ come from a finite set that depends only on $V$, $\phi$, and $\psi$.
	
	We distinguish two cases. First, if $\hat W\subset \hat{V} \backslash \hat{V}_0$, then $\hat W$ is contained in one of finitely many proper subvarieties of $\hat{V}$ and is optimal for that subvariety. By Lemma \ref{lem:zilberpinkfiniteembedding} and the fact that, by the inductive hypothesis, $\ZP(X,m-1,d) $ holds, after recalling that $\delta (\hat W)\leq d$, we have at most finitely many possibilities for $\hat W$ and therefore at most finitely many possibilities for $W$.
	
	Let us now assume that $\hat{W} \cap \hat{V}_0 \neq \emptyset$. Since $\hat{W}$ is an irreducible component of $ \psi|_{\hat{V}}^{-1}(z)$, we must have that $\hat{W} \cap \hat{V}_0$ is an irreducible component of $\psi|_{\hat{V}_0}^{-1}(z)$ and thus $n=\dim(\hat{W} \cap \hat{V}_0) = \dim \hat{W}=\dim W$.
	Note that $\dim \tilde  V=\dim \hat V-n\leq \dim V-n\leq m$. 
	
	We now want to apply Lemma \ref{lem:imageisoptimal}(1) to $\psi$ and $\hat W \subset \hat V \subset Y_{\phi}$. For this we need to verify that $\langle \hat{W} \rangle$ contains an irreducible component of a fiber of $\psi$. This fiber can be taken to be nothing but $\psi^{-1}(z)$: we have that $\langle \hat{W} \rangle_{\ws}\subset \langle \hat{W} \rangle$ is contained in an irreducible component of $\psi^{-1}(z)$. If this containment were strict, we would have that
	\[ W \subset \phi(\langle \hat{W} \rangle_{\ws}) \cap \langle W \rangle_{\ws} \subsetneq \langle W \rangle_{\ws}\]
	since $\langle W \rangle_{\ws}$ is an irreducible component of $\phi(\psi^{-1}(z))$, $\phi $ has finite fibers, and all components of $\psi^{-1}(z)$ have the same dimension. By Lemmas \ref{lem:(pre)image} and \ref{lem:intersection}, every irreducible component of $\phi(\langle \hat{W} \rangle_{\ws}) \cap \langle W \rangle_{\ws}$ is weakly special, a contradiction.
	
	We can then apply Lemma \ref{lem:imageisoptimal}(1) and obtain that $\{z\}=\psi (\hat W)$ is optimal for $\tilde V$ in $Z_\psi $ with $\delta (\{z\})\leq \delta (\hat W)\leq d$. Since, by assumption, every subvariety of $Z_\psi$ of dimension at most $m$ contains at most finitely many optimal singletons of defect at most $d$, the point $z$ lies in a finite set.
	
	Finally, recall that $\hat W$ is a component of $\psi^{-1}(z) \cap\hat V$, thus lies in a finite set and the same holds for $W$.	
\end{proof}

\section{Reduction of Zilber's formulation to Pink's formulation}\label{sec:zilberpink}

In this section, we introduce two different formulations of the Zilber-Pink conjecture for a distinguished category and show that they are equivalent in the presence of \ref{ax:5}.

The following is Zilber's version of the conjecture (see \cite{Zilber}), in the formulation by Habegger and Pila in \cite{HP} (see also there for a form of the conjecture that is closer to Zilber's original conjecture; their proof of the equivalence of the two versions works in any distinguished category). 

\begin{conj}\label{conj:zilber}
Let $K$ be an algebraically closed field and let $\mathfrak{C}$ be a distinguished category over $K$. For every distinguished variety $X$ and all non-negative integers $m$ and $d$, $\ZP(X,m,d)$ holds.
\end{conj}

We next present a version that corresponds to Pink's Conjecture 1.2 in \cite{PinkUnpubl}. To formulate it, we introduce the following notation: If $X$ is a distinguished variety and $k \in \{0,\hdots,\dim X\}$, then $X^{[k]}$ denotes the union of all special subvarieties of $X$ of codimension at least $k$.

\begin{conj}\label{conj:pink}
Let $K$ be an algebraically closed field, let $\mathfrak{C}$ be a distinguished category over $K$, and let $m$ and $d$ be non-negative integers. Let $X$ be a distinguished variety and $V \subset X$ a subvariety of dimension at most $m$. Then $V \cap X^{[\dim X - \min\{\delta(V)-1,d\}]}$ is not Zariski dense in $V$.
\end{conj}

We point out that Conjectures \ref{conj:zilber} and \ref{conj:pink} certainly do not hold in an arbitrary distinguished category. Indeed, in $\mathfrak{C}_{\mathrm{add}}$, we can consider $X = \mathbb{G}^4_{a,K}$ with affine coordinates $x,y,z,w$ and $V \subset X$ defined by $x^2 + y^2 = z^2$, $x^3 + y^3 = w^3$. Then $V$ is covered by lines passing through the origin that are all optimal for $V$ in $X$, which contradicts Conjecture \ref{conj:zilber}. Furthermore, any such line is of dimension $1 = \delta(V) - 1$, so we contradict Conjecture \ref{conj:pink} as well.

\begin{lem}\label{lem:zilberimpliespink}
Let $K$ be an algebraically closed field, let $\mathfrak{C}$ be a distinguished category over $K$, let $X$ be a distinguished variety, and let $m$ and $d$ be non-negative integers. Then $\ZP(X,m,d)$ implies Conjecture \ref{conj:pink} for the given $\mathfrak{C}$, $X$, $m$, and $d$.
\end{lem}

\begin{proof}
Let $V \subset X$ be a subvariety of dimension at most $m$. Let $W$ be an irreducible component of the intersection of $V$ with a special subvariety of codimension at least $\dim X - \min\{\delta(V)-1,d\}$. Then
\[ \delta(W) \leq \min\{\delta(V) - 1,d\} -\dim W \leq \min\{\delta(V)-1,d\}.\]

There is a subvariety $U \subset V$, optimal for $V$ in $X$, such that $W \subset U$ and $\delta(U) \leq \delta(W) \leq d$. Since $\delta(W) < \delta(V)$, we have $U \neq V$. So $W$ is contained in the union of all proper optimal subvarieties for $V$ in $X$ of defect at most $d$, which by $\ZP(X,m,d)$ is a finite union.

We deduce that $V \cap X^{[\dim X - \min\{\delta(V)-1,d\}]}$ is not dense in $V$.
\end{proof}

The following theorem is an analogue of Theorem 1.9 in \cite{BD}. A similar statement for connected mixed Shimura varieties of Kuga type can be found in the work \cite{Cass3} of Cassani.

\begin{thm}\label{thm:zilberequivpink}
Let $K$ be an algebraically closed field, let $\mathfrak{C}$ be a distinguished category over $K$ that satisfies \ref{ax:5}, and let $m$ and $d$ be non-negative integers. Then Conjectures \ref{conj:zilber} and \ref{conj:pink} for $\mathfrak{C}$, $m$, and $d$ are equivalent.
\end{thm}

\begin{proof}
By Lemma \ref{lem:zilberimpliespink}, it suffices to show that Conjecture \ref{conj:pink} for $\mathfrak{C}$, $m$, and $d$ implies $\ZP(X,m,d)$ for every distinguished variety $X$. By Theorem \ref{thm:redoptsing}, it is then enough to show that every subvariety of dimension at most $m$ of a distinguished variety contains at most finitely many optimal singletons of defect at most $d$.

For this, we follow the proof of Theorem 1.9 in \cite{BD}: Let $X$ be a distinguished variety and let $V$ be a subvariety of $X$ of dimension at most $m$. We show the following claim by induction on $j \in \{0,\hdots,\dim V\}$:

\begin{claim}
The optimal singletons for $V$ in $X$ of defect at most $d$ are contained in a finite union of subvarieties of $V$ of dimension at most $\dim V - j$.
\end{claim}

This is obvious for $j = 0$.

Suppose that the claim holds for some $j < \dim V$. Let $W$ be one of the finitely many subvarieties of $V$ of dimension at most $\dim V - j$ that contain the optimal singletons for $V$ in $X$ of defect at most $d$. We can assume without loss of generality that $\dim W = \dim V -j$.

If $\{p\} \subset W$ is an optimal singleton for $W$ in $X$ of defect at most $d$, then $\langle \{p\} \rangle \subset \langle W \rangle$. Since $\dim W = \dim V -j > 0$, we have that $\{p\} \subsetneq W$ and therefore $\delta(\{p\}) = \dim \langle \{p\} \rangle < \delta(W)$. It follows that the codimension of $\langle \{p\} \rangle$ is greater than or equal to $k := \dim X - \min\{\delta(W)-1,d\}$. So the optimal singletons for $W$ in $X$ of defect at most $d$ are contained in $W \cap X^{[k]}$.

It then follows from Conjecture \ref{conj:pink} for $X$, $d$, and $W$ that $W \cap X^{[k]}$ is not dense in $W$ and so the same holds for the union of all optimal singletons for $W$ in $X$ of defect at most $d$. This implies that the optimal singletons for $W$ in $X$ of defect at most $d$ are contained in a proper closed subset of $W$ as desired. This establishes the claim by induction.

\medskip

Now taking $j = \dim V$ in the claim shows that the number of optimal singletons for $V$ in $X$ of defect at most $d$ is finite.
\end{proof}

\section{Zilber-Pink for a complex curve in a fibered power of the Legendre family}\label{sec:legendre}
Let $Y(2) = \mathbb{A}^{1}_{\bar{\mathbb{Q}}}\backslash\{0,1\}$. Let $\mathcal{E}$ denote the Legendre family of elliptic curves, defined in $Y(2) \times_{\bar{\mathbb{Q}}} \mathbb{P}^2_{\bar{\mathbb{Q}}}$ by $Y^2Z = X(X-Z)(X-\lambda Z)$, where $[X:Y:Z]$ are the projective coordinates on $\mathbb{P}^2_{\bar{\mathbb{Q}}}$ and $\lambda$ is the affine coordinate on $Y(2) \subset \mathbb{A}^{1}_{\bar{\mathbb{Q}}}$. For $g \in \mathbb{N}$, let $\mathcal{E}^g$ denote the $g$-th fibered power of $\mathcal{E}$ over $Y(2)$.

Set $I_2 = \begin{pmatrix} 1 & 0 \\ 0 & 1\end{pmatrix}$ and $J_{2} = \begin{pmatrix} 0 & 1 \\ -1 & 0\end{pmatrix}$. Let $P = \GL_{2,\mathbb{Q}}\ltimes \mathbb{Q}^{2g}$, where $\GL_{2,\mathbb{Q}}$ acts on $\mathbb{Q}^{2g} = (\mathbb{Q}^2)^{\oplus g}$ by acting tautologically on each of the $g$ summands. Then $\mathcal{E}^g$ is canonically a connected mixed Shimura variety of Kuga type, associated to the triple $(P, X^+,\Gamma\ltimes \mathbb{Z}^{2g})$, where $X^+$ is a connected component of the $P(\mathbb{R})$-conjugacy class of the homomorphism $h: \mathbb{S}_{\mathbb{C}} \to \GL_{2,\mathbb{C}} \ltimes \mathbb{C}^{2g}$ defined by
\[ \alpha \in \mathbb{S}(\mathbb{R}) = \mathbb{C}^\ast\mapsto h(\alpha) = ((\Re \alpha)I_{2}+(\Im \alpha)J_{2},0)\]
and
\[ \Gamma = \left\{ A = \begin{pmatrix}a & b \\ c & d\end{pmatrix} \in \SL_2(\mathbb{Z});~A \equiv \begin{pmatrix} 1 & 0 \\ 0 & 1 \end{pmatrix} ~\text{mod}~ 2 ~\text{ and }~ a \equiv d \equiv 1~\text{mod}~4 \right\}.\]
As a complex manifold, $X^+$ is isomorphic to $\mathbb{H}\times \mathbb{C}^g$.

\begin{thm}\label{thm:legendre}
Let $K$ be an algebraically closed field. Then $\ZP(\mathcal{E}^g_K,1,d)$ holds for all $g \in \mathbb{N}$ and all non-negative integers $d$.
\end{thm}

\begin{proof}
We want to use Theorem \ref{thm:fieldofdef} to deduce $\ZP(\mathcal{E}^g_K,1,d)$ from $\ZP(X,1,d)$ for $X$ equal to a power of a CM elliptic curve over $\bar{\mathbb{Q}}$ and $X$ equal to a fibered power of $\mathcal{E}$, which is known thanks to several authors.

By Theorem \ref{thm:gaoweakfiniteness}, the category of connected mixed Shimura varieties of Kuga type over ${\bar{\mathbb{Q}}}$ is very distinguished. We want to apply Theorem \ref{thm:fieldofdef}(1), choosing the class $\mathcal{S}$ there as small as possible. Theorem 8.2 in \cite{G18} and the proof of Theorem \ref{thm:gaoweakfiniteness} show that we first have to include all triples $((Q,Y^+)/N,\ast)$, where $(Q,Y^+)$ is a Shimura subdatum of $(P,X^+)$ and $N$ is a connected algebraic subgroup of $Q^{\der}$, normal in $Q$ (equivalently: $N$ is a normal connected algebraic subgroup of $Q$ whose reductive part is semisimple). Note that we can assume that $N$ is connected in the conclusion of Theorem 8.2 in \cite{G18} since replacing $N$ by its identity component does not change $N(\mathbb{R})^+$.

By Proposition 1.2.16 and its proof in \cite{GDiss} (cf. Proposition 3.4 in \cite{G15}), we have $Q = (I_2,v)(G_Q \ltimes V)(I_2,v)^{-1}$, where $G_Q \subset \GL_{2,\mathbb{Q}}$ is a reductive subgroup, $V$ is a $G_Q$-invariant vector subspace of $\mathbb{Q}^{2g}$, and $v\in \mathbb{Q}^{2g}$. Up to isomorphism, we can assume that $v = 0$. Furthermore, $G_{Q,\mathbb{R}}$ contains a $2$-dimensional torus (the image of $\mathbb{S}$) that contains the center of $\GL_{2,\mathbb{R}}$. It follows from Theorem 1 in \cite{NPT} that $G_Q$ is either a $2$-dimensional torus $T$ or $G_Q = \GL_{2,\mathbb{Q}}$. Thanks to Proposition 4.17 in \cite{MilneAG}, we can then assume (up to isomorphism) that $V$ is the direct sum of $g_Q$ summands of the direct sum $(\mathbb{Q}^2)^{\oplus g}$ for some non-negative integer $g_Q$ as these summands are all simple as representations of either $T$ or $\GL_{2,\mathbb{Q}}$.

Since $(Q,Y^+)$ is a connected mixed Shimura datum, we have $Q^{\der} = (G_Q)^{\der} \ltimes V$, so $Q^{\der} = V$ in case $G_Q =T$ and $Q^{\der} = \SL_{2,\mathbb{Q}} \ltimes V$ if $G_Q = \GL_{2,\mathbb{Q}}$.

We set $V_N = N \cap V$ and $G_N = N/V_N$, where we regard $G_N$ as an algebraic subgroup of $G_Q = Q/V$. Thanks to Proposition 4.17 in \cite{MilneAG}, we can find a $G_Q$-invariant vector subspace $V_N^{\perp} \subset V$ such that $V = V_N \oplus V_N^{\perp}$. Furthermore, $V_N^{\perp}$ can be chosen as the direct sum of some of the summands in the direct sum $(\mathbb{Q}^2)^{\oplus g_Q}$.

In case $G_Q =T$, we have that $N = V_N$ is a $T$-invariant vector subspace of $V$ and therefore $Q/N \simeq T \ltimes V_N^{\perp}$.

If $G_Q = \GL_{2,\mathbb{Q}}$, then $G_N \subset \SL_{2,\mathbb{Q}}$ is normal and connected and $V_N \subset V$ is a $\GL_{2,\mathbb{Q}}$-invariant vector subspace. Since $\SL_{2,\mathbb{Q}}$ is almost-simple (as defined in Definition 19.7 in \cite{MilneAG}), it follows that $G_N = \{1\}$ or $G_N = \SL_{2,\mathbb{Q}}$. 

In the first case, $N = V_N$ is a $\GL_{2,\mathbb{Q}}$-invariant vector subspace of $V$. Then, $Q/N \simeq \GL_{2,\mathbb{Q}} \ltimes V_N^{\perp}$.

We are then left with the case $G_N = \SL_{2,\mathbb{Q}}$. We have that $(-I_2,w) \in N(\mathbb{C})$ for some $w \in V(\mathbb{C})$ and hence
\[ (-I_2,w)(I_2,v)(-I_2,w)(I_2,-v) = (I_2,-2v) \in N(\mathbb{C})\]
for every $v\in V(\mathbb{C})$. This implies that $N = \SL_{2,\mathbb{Q}} \ltimes V=Q^{\der}$ and therefore $Q/N\simeq \mathbb{G}_{m,\mathbb{Q}}$.

Hence there are the following possibilities for $Q/N$ (up to isomorphism): $T \ltimes  V_N^{\perp}$, $\GL_{2,\mathbb{Q}} \ltimes  V_N^{\perp}$, and $\mathbb{G}_{m,\mathbb{Q}}$. In principle, we now have to iterate this step and again look at all quotients of Shimura subdata as above of these connected mixed Shimura data. But this yields nothing new, so we have found our $\mathcal{S}$. The above considerations also show that the special subvarieties of $\mathcal{E}^g$ are precisely the irreducible components of flat subgroup schemes (see Section 2 of \cite{BC}) and algebraic subgroups of CM fibers.

Instead of checking $\ZP(X,1,d)$ for all $X \in \mathcal{S}$, we can use Lemma \ref{lem:zilberpinkfinitecover} to fix one congruence subgroup for each connected mixed Shimura datum that occurs in $\mathcal{S}$. To be able to apply Theorem \ref{thm:fieldofdef}(1), we therefore only have to check that $\ZP(X,1,d)$ holds for $X$ equal to a point, $Y(2)$, a $j$-th power of a CM elliptic curve over $\bar{\mathbb{Q}}$ ($1 \leq j \leq g)$, or $\mathcal{E}^j$ ($1 \leq j \leq g$). This follows from the results of \cite{RemVia}, \cite{Viada2008}, \cite{galateau2010}, \cite{BC}, and \cite{Superbarro}.
\end{proof}

\section{Zilber-Pink for a complex curve in a semiabelian variety}\label{sec:semiabelian}

In this section, we will apply Theorem \ref{thm:fieldofdef} to deduce the Zilber-Pink conjecture for a complex curve in the base change of a semiabelian variety over $\bar{\mathbb{Q}}$ from the same statement for a curve defined over $\bar{\mathbb{Q}}$, which is known thanks to the first-named author, K\"uhne, and Schmidt.

\begin{thm}\label{thm:semiabelian}
Let $K$ be an algebraically closed field and let $G$ be a semiabelian variety over $\bar{\mathbb{Q}}$. Then $\ZP(G_K,1,d)$ holds for all non-negative integers $d$.
\end{thm}

\begin{proof}
By Theorem \ref{thm:kirbyweakfiniteness}, the category of semiabelian varieties over $\bar{\mathbb{Q}}$ is very distinguished. By Theorem 1.1 in the recent work of the first-named author, K\"uhne, and Schmidt \cite{BarKS}, the statement $\ZP(H,1,d)$ holds for every semiabelian variety $H$ over $\bar{\mathbb{Q}}$ and for all non-negative integers $d$. Theorem \ref{thm:semiabelian} now follows from Theorem \ref{thm:fieldofdef}(1), applied with $\mathcal{S}$ equal to the class of all semiabelian varieties over $\bar{\mathbb{Q}}$.
\end{proof}

\section*{Acknowledgements}

We thank Chris Daw for providing us with the crucial part of the proof of Lemma \ref{lem:milneshih}, David Masser for pointing out his paper \cite{Masser99} to us, and Martin Orr for pointing out the paper \cite{RohlfsSchwermer} to us. We thank Cassani, Ziyang Gao, Philipp Habegger, Lars K\"uhne, Ben Moonen, Martin Orr, Jonathan Pila, Harry Schmidt, Christian Urech, and Immanuel van Santen for relevant and helpful discussions and correspondence. We thank the anonymous referees for their comments, which helped to improve the exposition of this article.  The first-named author received support from the PRIN 2022 project 2022HPSNCR: Semiabelian varieties, Galois representations and related Diophantine problems and from the GNSAGA-INDAM group.
The second-named author was supported during his work on this article by the Swiss National Science Foundation as part of the project ``Diophantine Problems, o-Minimality, and Heights", no. 200021\_165525, as well as through the Early Postdoc.Mobility grant no. P2BSP2\_195703. He thanks the Mathematical Institute of the University of Oxford and his host there, Jonathan Pila, for hosting him as a visitor for the duration of this grant.

\bibliographystyle{acm}

\bibliography{Bibliography}

\def\cprime{$'$} \def\cprime{$'$}
\begin{thebibliography}{10}

\bibitem{SGA1}
{\em Rev\^{e}tements \'{e}tales et groupe fondamental ({SGA} 1)}, vol.~3 of
  {\em Documents Math\'{e}matiques (Paris)}.
\newblock Soci\'{e}t\'{e} Math\'{e}matique de France, Paris, 2003.
\newblock S\'{e}minaire de g\'{e}om\'{e}trie alg\'{e}brique du Bois Marie
  1960--61, directed by A. Grothendieck, with two papers by M. Raynaud, updated
  and annotated reprint of the 1971 original [Lecture Notes in Math., 224,
  Springer, Berlin].

\bibitem{Superbarro}
{\sc Barroero, F.}
\newblock C{M} relations in fibered powers of elliptic families.
\newblock {\em J. Inst. Math. Jussieu 18}, 5 (2019), 941--956.

\bibitem{BC}
{\sc Barroero, F., and Capuano, L.}
\newblock Linear relations in families of powers of elliptic curves.
\newblock {\em Algebra Number Theory 10}, 1 (2016), 195--214.

\bibitem{BD}
{\sc Barroero, F., and Dill, G.~A.}
\newblock On the {Z}ilber-{P}ink conjecture for complex abelian varieties.
\newblock {\em Ann. Sci. \'{E}c. Norm. Sup\'{e}r. (4) 55}, 1 (2022), 261--282.

\bibitem{BarKS}
{\sc Barroero, F., K\"{u}hne, L., and Schmidt, H.}
\newblock Unlikely intersections of curves with algebraic subgroups in
  semiabelian varieties.
\newblock {\em Selecta Math. (N.S.) 29}, 2 (2023), paper no. 18, 37 pages.

\bibitem{ArDyn}
{\sc Benedetto, R., Ingram, P., Jones, R., Manes, M., Silverman, J.~H., and
  Tucker, T.~J.}
\newblock Current trends and open problems in arithmetic dynamics.
\newblock {\em Bull. Amer. Math. Soc. (N.S.) 56}, 4 (2019), 611--685.

\bibitem{BinyaminiDaw}
{\sc Binyamini, G., and Daw, C.}
\newblock Effective computations for weakly optimal subvarieties.
\newblock Appeared online in J. Eur. Math. Soc. (JEMS),
  \url{https://doi.org/10.4171/JEMS/1408}, 2024.

\bibitem{BMZ99}
{\sc Bombieri, E., Masser, D., and Zannier, U.}
\newblock Intersecting a curve with algebraic subgroups of multiplicative
  groups.
\newblock {\em Int. Math. Res. Not. IMRN}, 20 (1999), 1119--1140.

\bibitem{BMZ06}
{\sc Bombieri, E., Masser, D., and Zannier, U.}
\newblock Intersecting curves and algebraic subgroups: conjectures and more
  results.
\newblock {\em Trans. Amer. Math. Soc. 358}, 5 (2006), 2247--2257.

\bibitem{BMZ07}
{\sc Bombieri, E., Masser, D., and Zannier, U.}
\newblock Anomalous subvarieties---structure theorems and applications.
\newblock {\em Int. Math. Res. Not. IMRN}, 19 (2007), Art. ID rnm057, 33 pages.

\bibitem{BMZ08}
{\sc Bombieri, E., Masser, D., and Zannier, U.}
\newblock On unlikely intersections of complex varieties with tori.
\newblock {\em Acta Arith. 133}, 4 (2008), 309--323.

\bibitem{Brion}
{\sc Brion, M.}
\newblock Commutative algebraic groups up to isogeny.
\newblock {\em Doc. Math. 22\/} (2017), 679--725.

\bibitem{Cass2}
{\sc Cassani}.
\newblock The {D}efect {C}ondition.
\newblock MODNET eprint, 2021.

\bibitem{Cass3}
{\sc Cassani}.
\newblock The {F}initeness {R}esult \`{a} la {B}ogomolov.
\newblock MODNET eprint, 2021.

\bibitem{Chambert-Loir}
{\sc Chambert-Loir, A.}
\newblock Relations de d\'{e}pendance et intersections exceptionnelles.
\newblock {\em Ast\'{e}risque}, 348 (2012).
\newblock S\'{e}minaire Bourbaki: Vol. 2010/2011. Expos\'{e}s 1027--1042, Exp.
  No. 1032, viii, 149--188.

\bibitem{CGMM}
{\sc Chatzidakis, Z., Ghioca, D., Masser, D., and Maurin, G.}
\newblock Unlikely, likely and impossible intersections without algebraic
  groups.
\newblock {\em Atti Accad. Naz. Lincei Rend. Lincei Mat. Appl. 24}, 4 (2013),
  485--501.

\bibitem{Chiu}
{\sc Chiu, K. C.~T.}
\newblock {A}x--{S}chanuel for variations of mixed {H}odge structures.
\newblock Appeared online in Math. Ann.,
  \url{https://doi.org/10.1007/s00208-024-02958-x}, 2024.

\bibitem{DawOrr1}
{\sc Daw, C., and Orr, M.}
\newblock Unlikely intersections with {$E\times{\rm CM}$} curves in
  {$\mathcal{A}_2$}.
\newblock {\em Ann. Sc. Norm. Super. Pisa Cl. Sci. (5) 22}, 4 (2021),
  1705--1745.

\bibitem{DawOrr2}
{\sc Daw, C., and Orr, M.}
\newblock Quantitative reduction theory and unlikely intersections.
\newblock {\em Int. Math. Res. Not. IMRN}, 20 (2022), 16138--16195.

\bibitem{DawRen}
{\sc Daw, C., and Ren, J.}
\newblock Applications of the hyperbolic {A}x-{S}chanuel conjecture.
\newblock {\em Compos. Math. 154}, 9 (2018), 1843--1888.

\bibitem{MilneShih}
{\sc Deligne, P., Milne, J.~S., Ogus, A., and Shih, K.}
\newblock {\em Hodge cycles, motives, and {S}himura varieties}, vol.~900 of
  {\em Lecture Notes in Mathematics}.
\newblock Springer-Verlag, Berlin-New York, 1982.

\bibitem{Eterovic_Scanlon_Preprint}
{\sc Eterovi\'{c}, S., and Scanlon, T.}
\newblock Likely intersections.
\newblock \url{https://arxiv.org/abs/2211.10592}, 2022.

\bibitem{galateau2010}
{\sc Galateau, A.}
\newblock Une minoration du minimum essentiel sur les vari\'et\'es
  ab\'eliennes.
\newblock {\em Comment. Math. Helv. 85}, 4 (2010), 775--812.

\bibitem{GDiss}
{\sc Gao, Z.}
\newblock {\em The mixed {A}x-{L}indemann theorem and its applications to the
  {Z}ilber-{P}ink conjecture}.
\newblock PhD thesis, Universiteit Leiden/Universit\'{e} Paris-Sud, 2014.

\bibitem{G15}
{\sc Gao, Z.}
\newblock A special point problem of {A}ndr\'e-{P}ink-{Z}annier in the
  universal family of {A}belian varieties.
\newblock {\em Ann. Sc. Norm. Super. Pisa Cl. Sci. (5) 17}, 1 (2017), 231--266.

\bibitem{G17}
{\sc Gao, Z.}
\newblock Towards the {A}ndre-{O}ort conjecture for mixed {S}himura varieties:
  the {A}x-{L}indemann theorem and lower bounds for {G}alois orbits of special
  points.
\newblock {\em J. Reine Angew. Math. 732\/} (2017), 85--146.

\bibitem{Gao19}
{\sc Gao, Z.}
\newblock Enlarged mixed {S}himura varieties, bi-algebraic system and some {A}x
  type transcendental results.
\newblock {\em Forum Math. Sigma 7\/} (2019), paper no. e16, 65 pages.

\bibitem{G18}
{\sc Gao, Z.}
\newblock Mixed {A}x-{S}chanuel for the universal abelian varieties and some
  applications.
\newblock {\em Compos. Math. 156}, 11 (2020), 2263--2297.

\bibitem{GK}
{\sc Gao, Z., and Klingler, B.}
\newblock The {A}x-{S}chanuel conjecture for variations of mixed {H}odge
  structures.
\newblock {\em Math. Ann. 388}, 4 (2024), 3847--3895.

\bibitem{GhiocaHsia}
{\sc Ghioca, D., and Hsia, L.-C.}
\newblock Torsion points in families of {D}rinfeld modules.
\newblock {\em Acta Arith. 161}, 3 (2013), 219--240.

\bibitem{GT}
{\sc Ghioca, D., and Tucker, T.~J.}
\newblock A reformulation of the dynamical {M}anin-{M}umford conjecture.
\newblock {\em Bull. Aust. Math. Soc. 103}, 1 (2021), 154--161.

\bibitem{GW}
{\sc G\"{o}rtz, U., and Wedhorn, T.}
\newblock {\em Algebraic geometry {I}}.
\newblock Advanced Lectures in Mathematics. Vieweg + Teubner, Wiesbaden, 2010.
\newblock Schemes with examples and exercises.

\bibitem{GWErr}
{\sc G\"ortz, U., and Wedhorn, T.}
\newblock Errata for ``{A}lgebraic geometry {I} ({E}dition 1)".
\newblock \url{https://algebraic-geometry.de/errata/1/}, 2021.

\bibitem{EGA_I}
{\sc Grothendieck, A.}
\newblock \'{E}l\'{e}ments de g\'{e}om\'{e}trie alg\'{e}brique. {I}. {L}e
  langage des sch\'{e}mas.
\newblock {\em Inst. Hautes \'{E}tudes Sci. Publ. Math.}, 4 (1960), 5--214.

\bibitem{HP}
{\sc Habegger, P., and Pila, J.}
\newblock O-minimality and certain atypical intersections.
\newblock {\em Ann. Sci. \'Ec. Norm. Sup\'er. (4) 49}, 4 (2016), 813--858.

\bibitem{Kirby}
{\sc Kirby, J.}
\newblock The theory of the exponential differential equations of semiabelian
  varieties.
\newblock {\em Selecta Math. (N.S.) 15}, 3 (2009), 445--486.

\bibitem{Klingler}
{\sc Klingler, B.}
\newblock Hodge loci and atypical intersections: conjectures.
\newblock To appear in the book ``Motives and complex multiplication"
  (Birkh\"auser), \url{https://arxiv.org/abs/1711.09387}, 2017.

\bibitem{LubVen}
{\sc Lubotzky, A., and Venkataramana, T.~N.}
\newblock The congruence topology, {G}rothendieck duality and thin groups.
\newblock {\em Algebra Number Theory 13}, 6 (2019), 1281--1298.

\bibitem{Masser99}
{\sc Masser, D.~W.}
\newblock Specializations of some hyperelliptic {J}acobians.
\newblock In {\em Number theory in progress, {V}ol. 1
  ({Z}akopane-{K}o\'{s}cielisko, 1997)}. de Gruyter, Berlin, 1999,
  pp.~293--307.

\bibitem{Maurin}
{\sc Maurin, G.}
\newblock Courbes alg\'ebriques et \'equations multiplicatives.
\newblock {\em Math. Ann. 341}, 4 (2008), 789--824.

\bibitem{MilneISV}
{\sc Milne, J.~S.}
\newblock Introduction to {S}himura varieties.
\newblock In {\em Harmonic analysis, the trace formula, and {S}himura
  varieties}, vol.~4 of {\em Clay Math. Proc.} Amer. Math. Soc., Providence,
  RI, 2005, pp.~265--378.

\bibitem{MilneAG}
{\sc Milne, J.~S.}
\newblock {\em Algebraic groups}, vol.~170 of {\em Cambridge Studies in
  Advanced Mathematics}.
\newblock Cambridge University Press, Cambridge, 2017.

\bibitem{MoonenLetter}
{\sc Moonen, B.}
\newblock Letter to {C}hristophe {C}ornut.
\newblock Private communication.

\bibitem{NPT}
{\sc Nguyen, K.~A., van~der Put, M., and Top, J.}
\newblock Algebraic subgroups of {${\rm GL}_2(\mathbb{C})$}.
\newblock {\em Indag. Math. (N.S.) 19}, 2 (2008), 287--297.

\bibitem{OV}
{\sc Onishchik, A.~L., and Vinberg, E.~B.}
\newblock {\em Lie groups and algebraic groups}.
\newblock Springer Series in Soviet Mathematics. Springer-Verlag, Berlin, 1990.
\newblock Translated from the Russian and with a preface by D. A. Leites.

\bibitem{OrrThesis}
{\sc Orr, M.}
\newblock {\em The {A}ndr{\'e}--{P}ink conjecture : {H}ecke orbits and weakly
  special subvarieties}.
\newblock PhD thesis, Universit{\'e} {P}aris {S}ud -- {P}aris {X}{I}, 2013.

\bibitem{PilaFermat}
{\sc Pila, J.}
\newblock On a modular {F}ermat equation.
\newblock {\em Comment. Math. Helv. 92}, 1 (2017), 85--103.

\bibitem{PilaUnpubl}
{\sc Pila, J.}
\newblock {\em Point-counting and the {Z}ilber-{P}ink conjecture}, vol.~228 of
  {\em Cambridge Tracts in Mathematics}.
\newblock Cambridge University Press, Cambridge, 2022.

\bibitem{PilaScanlon}
{\sc Pila, J., and Scanlon, T.}
\newblock Effective transcendental {Z}ilber-{P}ink for variations of {H}odge
  structures.
\newblock \url{https://arxiv.org/abs/2105.05845}, 2021.

\bibitem{Andre_Oort_Preprint}
{\sc Pila, J., Shankar, A.~N., Tsimerman, J., Esnault, H., and Groechenig, M.}
\newblock Canonical heights on {S}himura varieties and the {A}ndr{\'e}-{O}ort
  conjecture.
\newblock \url{https://arxiv.org/abs/2109.08788}, 2022.

\bibitem{PilaTsimerman}
{\sc Pila, J., and Tsimerman, J.}
\newblock Independence of {CM} points in elliptic curves.
\newblock {\em J. Eur. Math. Soc. (JEMS) 24}, 9 (2022), 3161--3182.

\bibitem{PilaZannier}
{\sc Pila, J., and Zannier, U.}
\newblock Rational points in periodic analytic sets and the {M}anin-{M}umford
  conjecture.
\newblock {\em Atti Accad. Naz. Lincei Rend. Lincei Mat. Appl. 19}, 2 (2008),
  149--162.

\bibitem{PinkDiss}
{\sc Pink, R.}
\newblock {\em Arithmetical compactification of mixed {S}himura varieties},
  vol.~209 of {\em Bonner Mathematische Schriften [Bonn Mathematical
  Publications]}.
\newblock Universit\"{a}t Bonn, Mathematisches Institut, Bonn, 1990.
\newblock Dissertation, Rheinische Friedrich-Wilhelms-Universit\"{a}t Bonn,
  Bonn, 1989.

\bibitem{Pink}
{\sc Pink, R.}
\newblock A combination of the conjectures of {M}ordell-{L}ang and
  {A}ndr\'e-{O}ort.
\newblock In {\em Geometric methods in algebra and number theory}, vol.~235 of
  {\em Progr. Math.} Birkh\"auser Boston, Boston, MA, 2005, pp.~251--282.

\bibitem{PinkUnpubl}
{\sc Pink, R.}
\newblock A common generalization of the conjectures of {A}ndr\'{e}-{O}ort,
  {M}anin-{M}umford, and {M}ordell-{L}ang.
\newblock \url{https://people.math.ethz.ch/~pink/ftp/AOMMML.pdf}, 2005.

\bibitem{PlatonovRapinchuk}
{\sc Platonov, V., and Rapinchuk, A.}
\newblock {\em Algebraic groups and number theory}, vol.~139 of {\em Pure and
  Applied Mathematics}.
\newblock Academic Press, Inc., Boston, MA, 1994.
\newblock Translated from the 1991 Russian original by Rachel Rowen.

\bibitem{Poizat}
{\sc Poizat, B.}
\newblock L'\'{e}galit\'{e} au cube.
\newblock {\em J. Symbolic Logic 66}, 4 (2001), 1647--1676.

\bibitem{MR3729254}
{\sc Poonen, B.}
\newblock {\em Rational points on varieties}, vol.~186 of {\em Graduate Studies
  in Mathematics}.
\newblock American Mathematical Society, Providence, RI, 2017.

\bibitem{Remond}
{\sc R{\'e}mond, G.}
\newblock Intersection de sous-groupes et de sous-vari\'et\'es. {III}.
\newblock {\em Comment. Math. Helv. 84}, 4 (2009), 835--863.

\bibitem{RemVia}
{\sc R{\'e}mond, G., and Viada, E.}
\newblock Probl\`eme de {M}ordell-{L}ang modulo certaines sous-vari\'et\'es
  ab\'eliennes.
\newblock {\em Int. Math. Res. Not. IMRN}, 35 (2003), 1915--1931.

\bibitem{RohlfsSchwermer}
{\sc Rohlfs, J., and Schwermer, J.}
\newblock Intersection numbers of special cycles.
\newblock {\em J. Amer. Math. Soc. 6}, 3 (1993), 755--778.

\bibitem{SatakeBook}
{\sc Satake, I.}
\newblock {\em Algebraic structures of symmetric domains}, vol.~4 of {\em
  Kan\^{o} Memorial Lectures}.
\newblock Iwanami Shoten, Tokyo; Princeton University Press, Princeton, N.J.,
  1980.

\bibitem{Schinzel89}
{\sc Schinzel, A.}
\newblock Reducibility of lacunary polynomials. {X}.
\newblock {\em Acta Arith. 53}, 1 (1989), 47--97.

\bibitem{SchinzelBook}
{\sc Schinzel, A.}
\newblock {\em Polynomials with special regard to reducibility}, vol.~77 of
  {\em Encyclopedia of Mathematics and its Applications}.
\newblock Cambridge University Press, Cambridge, 2000.
\newblock with an appendix by Umberto Zannier.

\bibitem{Tsimerman}
{\sc Tsimerman, J.}
\newblock The {A}ndr\'{e}-{O}ort conjecture for $\mathcal{A}_g$.
\newblock {\em Ann. of Math. (2) 187}, 2 (2018), 379--390.

\bibitem{Ullmo}
{\sc Ullmo, E.}
\newblock Structures sp\'{e}ciales et probl\`eme de {Z}ilber-{P}ink.
\newblock In {\em Around the {Z}ilber-{P}ink conjecture/{A}utour de la
  conjecture de {Z}ilber-{P}ink}, vol.~52 of {\em Panor. Synth\`eses}. Soc.
  Math. France, Paris, 2017, pp.~1--30.

\bibitem{Viada2003}
{\sc Viada, E.}
\newblock The intersection of a curve with algebraic subgroups in a product of
  elliptic curves.
\newblock {\em Ann. Sc. Norm. Super. Pisa Cl. Sci. (5) 2}, 1 (2003), 47--75.

\bibitem{Viada2008}
{\sc Viada, E.}
\newblock The intersection of a curve with a union of translated
  codimension-two subgroups in a power of an elliptic curve.
\newblock {\em Algebra Number Theory 2}, 3 (2008), 249--298.

\bibitem{Zannier}
{\sc Zannier, U.}
\newblock {\em Some {Problems of Unlikely Intersections in Arithmetic and
  Geometry}}, vol.~181 of {\em Annals of Mathematics Studies}.
\newblock Princeton University Press, 2012.
\newblock With appendixes by David Masser.

\bibitem{Zhang_1998}
{\sc Zhang, S.-W.}
\newblock Small points and {A}rakelov theory.
\newblock In {\em Proceedings of the {I}nternational {C}ongress of
  {M}athematicians, {V}ol. {II} ({B}erlin, 1998)\/} (1998), pp.~217--225.

\bibitem{Zilber}
{\sc Zilber, B.}
\newblock Exponential sums equations and the {S}chanuel conjecture.
\newblock {\em J. London Math. Soc. (2) 65}, 1 (2002), 27--44.

\bibitem{Zilber16}
{\sc Zilber, B.}
\newblock Model theory of special subvarieties and {S}chanuel-type conjectures.
\newblock {\em Ann. Pure Appl. Logic 167}, 10 (2016), 1000--1028.

\end{thebibliography}

\end{document}